\documentclass[10pt]{article}
\usepackage[T1]{fontenc}
\usepackage{tikz,tikz-3dplot}
\usetikzlibrary{math,calc}
\usepackage{amsthm,amsmath,amsfonts,amssymb,stmaryrd,bbm,verbatim}
\usepackage{mathtools}
\usepackage{enumitem}
\usepackage{mathrsfs} % for mathscr font
\usepackage{graphicx}
\usepackage{bbm}
\graphicspath{{figures}}
\usepackage{caption}
\usepackage{dsfont}
\usepackage{anyfontsize}
\usepackage{todonotes}
\usepackage{comment}
\usepackage{cite}
\usepackage{mathtools}
\mathtoolsset{showonlyrefs}
\usepackage[english]{babel}
\usepackage{url}
\usepackage[colorlinks=true, linkcolor=blue, urlcolor=black, citecolor=blue,pdfstartview=FitH]{hyperref}

\allowdisplaybreaks

\numberwithin{equation}{section}

\let\originalleft\left
\let\originalright\right
\renewcommand{\left}{\mathopen{}\mathclose\bgroup\originalleft}
\renewcommand{\right}{\aftergroup\egroup\originalright}

\newlength{\bibitemsep}
\newlength{\bibparskip}
\let\oldthebibliography\thebibliography
\renewcommand\thebibliography[1]{\oldthebibliography{#1}
	\setlength{\parskip}{\bibitemsep}
	\setlength{\itemsep}{\bibparskip}}

\DeclareMathOperator{\Tr}{Tr}

\DeclareMathOperator{\supp}{supp}

\DeclareMathOperator{\Arg}{Arg}

\renewcommand{\epsilon}{\varepsilon}

\renewcommand{\leq}{\leqslant}
\renewcommand{\geq}{\geqslant}

\newcommand{\di}{\textrm{d}}

\renewcommand{\P}{\mathbb{P}}

\newcommand{\R}{\mathbb{R}}
\newcommand{\C}{\mathbb{C}}
\newcommand{\N}{\mathbb{N}}
\newcommand{\Z}{\mathbb{Z}}

\newcommand{\cP}{\mathcal{P}}

\renewcommand{\ln}{\log}

\def\x{\mathbf{x}}
\def\y{\mathbf{y}}

\def\a{\mathbf{a}}

\def\xm{x_{\max}}

\newcommand{\vertiii}[1]{{\left\vert\kern-0.25ex\left\vert\kern-0.25ex\left\vert #1 
		\right\vert\kern-0.25ex\right\vert\kern-0.25ex\right\vert}}

\newcommand{\wt}[1]{\widetilde{#1}}

\theoremstyle{plain} %plain, definition, remark
\newtheorem{theorem}{Theorem}[section]
\newtheorem{lem}[theorem]{Lemma}
\newtheorem{cor}[theorem]{Corollary}

\newtheorem{proposition}[theorem]{Proposition}

\newtheorem*{theorem*}{Theorem}
\newtheorem{definition}[theorem]{Definition}

\newtheorem{remark}[theorem]{Remark}

\theoremstyle{definition}
\newtheorem{assump}{Assumption}
\newtheorem{rhp}[theorem]{RHP}
\newtheorem{problem}[theorem]{Model Problem}

\setlist[enumerate]{ align=left}

\xdefinecolor{PineGreen}{RGB}{5, 102, 50}

\author{Jonathan Husson\footnote{LMBP, Université Clermont Auvergne, 3 place Vasarely, 63178 Aubière, France. jonathan.husson@uca.fr}, Guido Mazzuca\footnote{Department of mathematics, Tulane University, 6823 St Charles Ave, New Orleans, LA 70118, USA. gmazzuca@tulane.edu}, \& Alessandra Occelli\footnote{LAREMA, Université d'Angers, 2 Boulevard Lavoisier, 49045 Angers, France. alessandra.occelli@univ-angers.fr}}
\title{Discrete and Continuous Muttalib--Borodin Process: Large Deviations and Limit Shape Analysis}
\date{\today}

\begin{document}
	\maketitle
    
\abstract{In this paper, we study the asymptotic behaviour of plane partitions distributed according to a $q^{\text{Volume}}$-weighted Muttalib--Borodin ensemble and its associated discrete point process. We establish a Large Deviation Principle for the process, explicitly characterizing the rate function. A defining feature of our model is the emergence of a strict upper bound on the macroscopic particle density, which translates the asymptotic analysis into a non-trivial constrained minimization problem. Through a rigorous Riemann--Hilbert analysis, we derive exact, closed-form formulas for the limit shape of the partitions across all parameter regimes. To the best of our knowledge, this represents the first time a constrained Riemann--Hilbert problem has been formulated and analytically solved for a bi-orthogonal ensemble. Our analysis allows to track the system through a macroscopic phase transition, computing the minimizer in both the subcritical and supercritical regimes. As a byproduct of our analysis, we obtain an explicit expression for the arctic curve that  separates the ``frozen'' and ``liquid'' regions of the limit shape. Furthermore, we reveal that the equilibrium measure exhibits a continuously varying exponent at the hard edge departing from the universal fixed exponents typically observed in classical random matrix theory.}

\section{Introduction}

In this paper, we  study the asymptotic behaviour of plane partitions, in particular we show that they satisfy a large deviation principle and we obtain an explicit expression of their asymptotic shape. We achieve that by considering the related discrete point process, called Muttalib--Borodin process, which is a particular bi-orthogonal ensemble. The main technical tool we use is the Riemann--Hilbert analysis, which allows us to solve a constrained minimization problem and to obtain an explicit expression for the limit shape of the plane partitions.

Plane partitions are ubiquitous in mathematics; they are not just a central object in combinatorics, but they also have several connections with the theory of integrable systems, random matrices and large deviations.
To study the asymptotic behaviour of integer partitions, one can usually consider a measure on these partitions, a natural one being the Plancherel measure; the typical goal is to understand their limit shape and the fluctuations around this limit. A milestone in the field was reached in 1999, when Baik, Deift and Johansson \cite{Baik1999} proved that the fluctuations of the length of the longest increasing subsequence in a random permutation (the same as the first entry of a uniformly distributed integer partition) are described by the GUE Tracy--Widom distribution \cite{TracyWidom}. 

Bi-orthogonal ensembles have been recently applied  to study the same phenomena for plane partitions \cite{Borodin1998,mut95, OR03}.
They arise as a natural extension of the well-studied orthogonal ensembles in mathematical physics and random matrix theory. Bi-orthogonal ensembles generalize this framework while maintaining some of its appealing features. They inherit the determinantal structure of the correlation functions, but, instead of being expressed through orthogonal polynomials, the kernels are constructed using bi-orthogonal polynomials \cite{Konhauser}. 
The determinantal structure of bi-orthogonal ensembles was first rigorously proved by Muttalib \cite{mut95}; he also introduced these ensembles in the context of random matrix theory and remarked their interest in physics. Some of their applications include: (modeling eigenvalues statistics of) disordered systems, such as systems with non-Hermitian Hamiltonians; interacting particle systems with less rigid symmetry constraints; in the context of quantum transport theory, they model transport properties of systems with correlated random potentials.
The main difference between bi-orthogonal ensembles and the classical orthogonal ones is that the first lack of a simple explicit Christoffel--Darboux formula for the bi-orthogonal polynomials.

\medskip

 Our contributions extend this picture to discrete and continuous Muttalib--Borodin processes, a particular bi-orthogonal ensembles arising from a weighted version of $q^{\text{Volume}}$-measure on plane partitions, with key findings including:
\begin{itemize}
    \item \textit{Large deviations principle (LDP)}: We establish a LDP for the discrete Muttalib--Borodin process, characterizing the rate function and identifying the minimizer under various regimes.
    \item \textit{Innovative Riemann--Hilbert problem (RHP) Analysis}: We address technical challenges in solving a constrained RHP, which is pivotal for understanding discrete bi-orthogonal ensemble.
    \item \textit{Novel limit shape analysis}: By relating the discrete Muttalib–Borodin process to plane partitions, we derive limit shapes under different parameter regimes, see Figure \ref{fig:plot_shape_partition}. A key finding is the characterization of the behavior near zero: unlike the fixed density exponents typically seen in random matrix ensembles, the exponent here varies across a wide range of values.
    \item \textit{Identification of the arctic curve}: We obtain an explicit expression for the \textit{arctic curve} of the discrete Muttalib--Borodin process; this curve naturally separate the point process in two region: a \textit{frozen region}, where the particle are as dense as possible, and a \textit{liquid region}, where the particles are ``free'' to move, see Figure \ref{fig:plot_density_beta_0}.
\end{itemize}
These results highlight the interplay between the geometry of plane partitions and the probabilistic properties of the Muttalib--Borodin process. Furthermore, our work suggests potential extensions, including fluctuation studies around the limit shape and large gap asymptotic.

\medskip
	We adopt the language of (plane and ordinary) partitions to state our results. We often encounter the q-Pochhammer symbols \cite[Ch. 5]{NIST:DLMF} of length $n\in\N\cup\{\infty\}$, defined as:
\begin{equation}
(x; q)_n=\prod_{0\leq i< n}(1-xq^i),
\end{equation}
where $q\in[0, 1)$.
\paragraph{Plane Partitions.} A plane partition $\Lambda$ is a matrix $(\Lambda_{i,j})_{1 \leq i \leq M \,1 \leq j \leq N}$ of non-negative integers satisfying the conditions:
$$\Lambda_{i,j} \geq \Lambda_{i,j+1} \quad \text{and} \quad \Lambda_{i,j} \geq \Lambda_{i+1,j}.$$
This arrangement can be visualized as stacks of cubes in a three-dimensional space, where the array corresponds to the number of cubes placed at each coordinate point of an $M\times N$ rectangular base. Plane partitions have applications in combinatorics, statistical mechanics, and representation theory.

They are are equivalently described by a sequence of interlacing integer partitions:
$$\{\lambda^{(t)}\}_{t=-M+1}^{N-1}\,: \quad \lambda^{(t-1)} \prec \lambda^{(t)}, \text{as } t\geq0, \quad \lambda^{(t)} \prec \lambda^{(t-1)}, \text{as } t<0,$$
where the interlacing condition $\lambda^{(s)} \prec \lambda^{(t)}$ means $\quad \lambda_1^{(s)} \geq \lambda_1^{(t)} \geq \lambda_2^{(s)} \geq \lambda_2^{(t)} \geq \cdots$.
This representation connects plane partitions with lozenge tilings, Schur functions, and determinantal point processes \cite{Gorin2021,Borodin2019}.

Given real parameters $ a\geq 0, 1>q$ and $\eta, \theta \geq 0$, we consider the following weight associated with a plane partition $\Lambda$:
\begin{equation}\label{eq:MBdist}\P(\Lambda)\propto  \left(aq^{\frac{\eta+\theta}{2}}\right)^\text{CentralVol} q^{\eta \cdot \text{LeftVol} + \theta \cdot \text{RightVol}},\end{equation}
where LeftVol, CentralVol, and RightVol represent the volumes of cubes in different regions of the plane partition. These weights are linked to $q$-deformations\cite{Macdonald} and the combinatorial geometry of partitions.
\vspace{1.5em}

\paragraph{Muttalib--Borodin Ensembles.} The Muttalib--Borodin ensemble (MBE) generalizes $\beta$-ensembles by introducing an additional interaction parameter $\theta>0$: the interacting potential $\Delta(x)$ is replaced by $\Delta(x)\Delta(x^\theta)$. The ensemble generated from the measure $\eqref{eq:MBdist}$ represents a slight generalization introducing two parameters $\eta,\theta>0$; one can think of it as a system with two-particle interactions, one between type $x^\eta_i$'s, one between type $x_i^\theta$'s. The probability density for $L_t$ points $0<x_1<\dots<x_{L_t}$ is given by:
\begin{equation}
\label{eq:MB_ensemble}
			\P(\x^{(t)} = \x) dx_1 \dots dx_{L_t}= \frac{1}{Z_c} \prod_{1 \leq i < j  \leq L_t} (x_j^\eta-x_i^\eta) (x_j^\theta-x_i^\theta) \prod_{1 \leq i \leq L_t} w_c(x_i) dx_i
\end{equation}
where $w_c(x)$ is a potential, and $Z_c$ is a normalization constant; $L_t$ is the length of the partition at time $t$. The interaction term  distinguishes MBEs from classical $\beta$-ensembles, making them suitable for modeling disordered conductors. 
The ensemble $(x(t))_t$ can be constructed as scaling limit of a discrete measure on plane partitions, called discrete Muttalib--Borodin processes (see \cite{fr05,BetOcc}).
These processes (MBPs) arise naturally in the study of plane partitions: each time slice of a plane partition corresponds to a discrete MBE, described by
\begin{equation}\label{eq:MBdistribution}
			\P(l^{(t)} = l) =\frac{1}{Z_d}\prod_{1 \leq i < j  \leq L_t} (Q^{l_j}-Q^{l_i}) (\tilde{Q}^{l_j} - \tilde{Q}^{l_i}) \prod_{1 \leq i \leq L_t} w_d(l_i)
\end{equation}
where $Z_d=\prod_{1\leq i\leq M}\prod_{1\leq j\leq N}(1-aQ^{i-\frac12}\tilde Q^{j-\frac12})^{-1}$ is the partition function and $Q = q^{\eta}$, $\tilde{Q} = q^{\theta}$ and $w_d(l_i)$ represent discrete weights derived from the volume contributions of the partitions, with
\begin{equation}\label{eq:weights_d}
			w_d(x) = \begin{cases}
				a^{x} (Q \tilde{Q})^{\frac{x}{2}} Q^{|t|{x}} (\tilde{Q}^{x-|t|+1}; \tilde{Q})_{N-(M-|t|)} & \text{if } t \leq 0, \\
				a^{x} (Q \tilde{Q})^{\frac{x}{2}} \tilde{Q}^{t{x}} (\tilde{Q}^{x+1}; \tilde{Q})_{N-t-M} & \text{if } t > 0 \text{ and } N-t \geq M, \\
				a^{x} (Q \tilde{Q})^{\frac{x}{2}} \tilde{Q}^{t{x}} (Q^{x+N-t-M+1}; Q)_{M-(N-t)} & \text{if } t > 0 \text{ and } N-t < M.
			\end{cases}
\end{equation}
In the limit 
\begin{equation}\label{eq:scaling}
q=e^{-\epsilon},\qquad a=e^{-\alpha\epsilon},\qquad \lambda_i(t)=-\frac{\ln x_i(t)}{\epsilon},\qquad\epsilon\to0+,\end{equation}
the discrete-space Muttalib–Borodin process $(l(t))_{-M+1\leq t\leq N-1}$ converges, in the sense of weak convergence of
finite dimensional distributions, to the process $(x(t))_{-M+1\leq t\leq N-1}$ supported in [0, 1].

As already observed by Muttalib~\cite{mut95}, each slice $l^{(t)}$ is a determinantal bi-orthogonal ensemble; moreover, as stated in \cite{BetOcc}, the whole time-extended (discrete and continuous) processes are determinantal. 

\begin{remark}
It is worth noticing that in view of the deformed interaction potential $\Delta(x^\eta)\Delta(x^\theta)$, one cannot derive the formula of the correlation Kernel and the asymptotic shape of the partition by means of classical Schur function theory as it happens, for example, in \cite{OR03}.
\end{remark}

\paragraph{Connection to Last Passage Percolation.} Plane partitions are intimately related to directed last passage percolation (LPP) models via the celebrated Robinson-Schensted-Knuth (RSK) correspondence \cite{Stanley1999}. In LPP models, random weights are assigned to the vertices of a lattice, and the length of a path is defined as the sum of these weights. We look for the longest path $L$ from a starting vertex to an endpoint. Note that the end point is not deterministic, but the path length is almost surely (a.s.) finite.
In the discrete setting, geometric random variables are often assigned to the lattice points, with weights $\omega_{i,j}\sim \text{Geom}(a Q^{i-1/2} \tilde Q^{j-1/2})$ where $a>0$ and $Q,\tilde Q$ control the inhomogeneity of the environment. The longest path length $L$ exhibits asymptotic fluctuations interpolating between Gumbel and Tracy--Widom distributions, depending on the parameters. This behavior was characterized in \cite{BetOcc} and previously encountered in a deformed GUE ensemble in \cite{JohanssonGumbelAiry}.

In the continuous setting, power-law distributed weights $\widehat{\omega}_{i,j}\sim\text{Pow}(\alpha + \eta(i-\frac12) +\theta(j-\frac12))$
replace the geometric weights. The longest path in this setting is asymptotically described by the hard-edge kernel of the Muttalib--Borodin process. For a more detailed description, see \cite{BetOcc}.

\tikzset{%
  pics/young diagram/.style={%
    code={%
      \foreach \k [count=\j from 0] in {#1} {%
        \foreach \i in {1, ..., \k}{%
          \draw (\i-1, -\j) rectangle ++(1, -1);
        }
      }
    }
  }
}

\tikzset{% needs tikz-3dplot
  pics/pCube/.style={% cube defined by the front lower vertex
    code={%
      \draw[fill=blue!50, canvas is yz plane at x=0] (-1, 0) rectangle ++(1, 1);
      \draw[fill=blue, canvas is zx plane at y=0] (0, -1) rectangle ++(1, 1);
      \draw[fill=blue!20, canvas is xy plane at z=1] (-1, -1) rectangle ++(1, 1);
      \filldraw[canvas is xy plane at z=1]
      (-1/2, -1/2) circle (.08);
    }
  },
  pics/pPartition/.style={%
    code={
      \foreach \partition [count=\i from 1] in {#1} {%
        \foreach \p [count=\j from 1] in \partition {%
          \foreach \k in {1, ..., \p} {%
            \path (\i, \j, \k -1) pic {pCube};
          }
        }
      }
    }
  },
  pics/pPartition+/.style={%
    code={
      \tikzmath{%
        integer \i, \j, \maxI, \maxJ, \M, \maxP;
        \maxI = 0;
        \maxJ = 0;
        \maxP = 0;
      }
      % find \M
      \foreach \partition [count=\i from 1] in {#1} {%
        \xdef\maxI{\i}
        \foreach \p [count=\j from 1] in \partition {%
          \ifnum\j>\maxJ
          \xdef\maxJ{\j}
          \fi
          \ifnum\i<2
          \ifnum\j<2
          \xdef\maxP{\p}
          \fi
          \fi  
        }
      }
      \ifnum\maxI>\maxJ
      \xdef\M{\maxI}
      \else
      \xdef\M{\maxJ}
      \fi
      % draw the base
      \fill[blue!10, canvas is xy plane at z=0]
      (0, 0) rectangle ++(\maxI, \maxJ);   
      \draw[canvas is xy plane at z=0] (0, 0) grid ++(\maxI, \maxJ);
      \foreach \x in {0.5,1.5,2.5,3.5,4.5,5.5} {
     	\foreach \y in {0.5,1.5,2.5,3.5,4.5,5.5} {
      		\fill[black] (\x,\y) circle (0.09); 
      	}
      }
      % draw the partition
      \foreach \partition [count=\i from 1] in {#1} {%
        \foreach \p [count=\j from 1] in \partition {%
          \foreach \k in {1, ..., \p} {%
            \path (\i, \j, \k -1) pic {pCube};
          }
        }
      }
    }
  }
}

\begin{figure}[!t]
  \centering
  \tdplotsetmaincoords{60}{45}
  \scalebox{0.65}{

\begin{tikzpicture}[evaluate={\a=7;}]
	\tdplotsetmaincoords{60}{120}  % latitude longitude (strangely defined)
	\begin{scope}[tdplot_main_coords]
		\draw[black, ->] (0, 0, 0) -- (\a, 0, 0) node[pos=1.05] {$x$};
		\draw[black, ->] (0, 0, 0) -- (0, \a, 0) node[pos=1.05] {$y$};
		\draw[black, ->] (0, 0, 0) -- (0, 0, 1.3*\a) node[pos=1.05] {$z$};
		
		\path (0, 0, 0) pic[]
		{pPartition+={{8,5,4,4,2,1},{6,5,3,3,2,1},{4,3,2,2,1},{4,2,1,1},{3,1},{1,1}}};    
	\end{scope}	
\end{tikzpicture}
}\quad
\begin{tikzpicture}[evaluate={%
      integer \Nleft, \Nright;
      \Nleft = 6;
      \Nright = 6;
    }, scale=.40]
    \draw[->] (-7,0)--(-7,14);
    \foreach \k in {0,...,14}{
      \node[left] at (-7,\k) {\footnotesize $\k$};
    }
    \foreach \k [parse=true] in {-\Nleft,...,\Nright}{%
      \draw[dashed, color=gray] (\k,0)--(\k,14);
      \node[below] at (\k,0) {\footnotesize $\k$};
    }
    \node at (-5,6) [circle,fill,inner sep=1.5pt]{};
    \foreach \k in {5,8}{%
        \node at (-4,\k) [circle,fill,inner sep=1.5pt]{};
        }
    \foreach \k in {3,5,9}{%
        \node at (-3,\k) [circle,fill,inner sep=1.5pt]{};
        }
    \foreach \k in {2,3,6,9}{%
        \node at (-2,\k) [circle,fill,inner sep=1.5pt]{};
        }
    \foreach \k in {1,2,4,7,11}{%
        \node at (-1,\k) [circle,fill,inner sep=1.5pt]{};
        }
    \foreach \k in {0,1,3,5,9,13}{%
        \node at (0,\k) [circle,fill,inner sep=1.5pt]{};
        }
    \foreach \k in {1,2,5,7,10}{%
        \node at (1,\k) [circle,fill,inner sep=1.5pt]{};
        }
    \foreach \k in {2,4,7,9}{%
        \node at (2,\k) [circle,fill,inner sep=1.5pt]{};
        }
    \foreach \k in {3,6,9}{%
        \node at (3,\k) [circle,fill,inner sep=1.5pt]{};
        }
    \foreach \k in {5,7}{%
        \node at (4,\k) [circle,fill,inner sep=1.5pt]{};
        }
    \foreach \k in {6}{%
        \node at (5,\k) [circle,fill,inner sep=1.5pt]{};
        }
 \end{tikzpicture}

\caption{A plane partition $\displaystyle \Lambda=\begin{pmatrix}
  8 & 5 & 4 & 4 & 2 & 1\\
  6 & 5 & 3 & 3 & 2 & 1\\
  4 & 3 & 2 & 2 & 1 & 0\\
  4 & 2 & 1 & 1 & 0 & 0\\
  3 & 1 & 0 & 0 & 0 & 0\\
  1 & 1 & 0 & 0 & 0 & 0
  \end{pmatrix}$ 
  with base in an $ M \times N$ rectangle for $(M,N)=(6, 6)$. We have LeftVol=$\sum_{i=-M}^{-1} |\lambda^{(i)}|=26$, CentralVol=$|\lambda^{(0)}|=15$, RightVol=$\sum_{i=1}^{N} |\lambda^{(i)}|=28$. To the right the corresponding particle configuration $\ell^{(t)}$.
  } 
  \label{fig:pp_mb}
\end{figure}
\paragraph{Connection to Particle Systems.} The sequence of diagonal partitions of $\Lambda$ defines a point process on \newline $\{ -M+1, \dots, -1, 0, 1, \dots, N -1\} \times \N$ where we place $L_t$ points at each time $-M \leq t \leq N$ \cite{Gorin2021}. The particle positions $l^{(t)}$ are given by the deterministic shift
	\begin{equation}
		l^{(t)}_i = \lambda^{(t)}_i + M - i, \quad 1 \leq i \leq L_t.
	\end{equation}
    We remark that each partition has length at most 
    \begin{equation}
    \label{eq:upper_partition}
         L_t=\begin{cases} M-|t| & -M\leq t\leq0,\\
    \min(M,N-t) & 0<t\leq N.
    \end{cases}
    \end{equation}
	As noted in Figure~\ref{fig:pp_mb}, the ensemble $(l^{(t)})_{t}$ is obtained by a shift of the positions of all horizontal lozenges in the plane partition.
	Obviously, also the particle system is determinantal, with the probability of finding particles at positions 
\((t_1, k_1), \dots, (t_n, k_n)\) given by:

\begin{equation}
\P \left( \bigcap_{i=1}^n \{ \text{Particle at } (t_i, k_i) \} \right) = \det \left[ K_d(t_i, k_i; t_j, k_j) \right]_{i,j=1}^n
\end{equation}
with an explicit correlation kernel $K_d(s, k; t, k')$.
Under the scaling \eqref{eq:scaling} we obtain a particle system on $(x(t))_t$ on $\{-M+1,\dots,N-1\}\times[0,1]$ (see Figure \ref{fig:pp_ips}) whose multi-point distribution is still described by an explicit kernel\cite{BetOcc} $K_c(s,x;t,y)$ as
\begin{equation}
\P \left( \bigcap_{i=1}^n \{ \text{Slice } t_i \text{ has a particle at } (x_i, x_i+\di x_i) \} \right) \prod_{i=1}^n \di x_i = \det \left[ K_c(t_i, x_i; t_j, x_j) \right]_{i,j=1}^n \prod_{i=1}^n \di x_i.
\end{equation}
Due to the discrete geometry of the lattice, the particle positions $l_i^{(t)}$ are subjected to a hard packing constraint (the discrete exclusion principle). In the macroscopic limit, under the exponential mapping, this constraint manifests as a strict upper bound on the continuous particle density $\mu(x)\leq (\beta \kappa x)^{-1}$. The spatial separation between regions where this constraint is active (the \textit{frozen region}, where particles are tightly packed) and inactive (the \textit{liquid region}) is delineated by the \textit{arctic curve} \cite{Kenyon2007}. One of the main novelties of our result is the exact analytic characterization of this phase transition for the Muttalib--Borodin ensemble. To the best of our knowledge, this is the first time a constrained variational problem exhibiting this phenomenon has been rigorously solved for a bi-orthogonal ensemble.

\subsection{Main results and techniques: Large Deviation Principles and Riemann--Hilbert Analysis}

In this paper, we consider the Muttalib--Borodin ensemble \eqref{eq:MBdistribution} in the regime 

\begin{equation}
    \label{eq:transformation}
    q=e^{-\epsilon},\qquad a=e^{-\alpha\epsilon},\qquad x_i^{(t)}=e^{-\epsilon l_i^{(t)}},\qquad\epsilon\to0^+.
\end{equation}
We are interested in the regime as the length of the partition approaches infinity, so we consider $M=\gamma^2 N$ and $t=\xi N$. By considering the weights of the marginal distribution~\eqref{eq:weights_d}, we realize that we must consider  three different regimes of $t$:
\begin{equation}
t=\xi N\in\begin{cases}
	[-M+1,0]\\(0,N-M]\\(N-M,N]
\end{cases}\,,
\end{equation}
which corresponds to three different regimes for $\xi$
\begin{equation}
    \xi\in\begin{cases}
	(-\gamma^2,0]\\(0,1-\gamma^2]\\(1-\gamma^2,1]\end{cases}\,.
\end{equation}
This allows us to rewrite the various regimes of $L_t$ \eqref{eq:upper_partition} as:
\begin{equation}
L_\xi=\begin{cases}
	\gamma^2N-|\xi|N, & \xi\in(-\gamma^2,0]\\\gamma^2N, &\xi\in(0,1-\gamma^2]\\N(1-\xi), &\xi\in(1-\gamma^2,1]
\end{cases}\,.
\end{equation}
Let us define the empirical measure for the discrete and the ``continuous'' model:
\begin{equation}
\mu_N^{(\xi)}=\frac{1}{L_\xi}\sum_{i=1}^{L_\xi}\delta_{\frac{l_i^{(\xi)}}{N}},\quad \widehat\mu_N^{(\xi)}=\frac{1}{L_\xi}\sum_{i=1}^{L_\xi}\delta_{x_i^{(\xi)}}.
\end{equation} 
In view of the definition of $l_i^{(\xi)}$, we  notice that for all $i$, $l_i^{(\xi)} \geq M - L_\xi$ which means that approximately $l_i^{(\xi)}/ N  \geq \gamma^2 - \kappa$, with $\kappa = \kappa(\xi)= L_\xi/N$, therefore, setting $\varepsilon=\frac{\beta}{N} + o(N^{-1})$,  $x_i^{(\xi)}  \leq e^{ - \beta( \gamma^2 - \kappa) + o(N^{-1})}$. So, $\widehat{\mu}_N^{(\xi)}$ has a support in $[0, e^{ - \beta( \gamma^2 - \kappa)}]$ asymptotically. Let us define, for $h\in(0,1)$,

 \[ \mathcal{P}^{\beta}([0,h]) = \left\{ \mu \in \mathcal{P}([0,h])\; : \; \frac{d \mu}{ dx }(x) \leq \frac{1}{\beta x}\right \}.  \]

\noindent \textbf{Our first result} is to show that the measures $\widehat \mu_N^{(\xi)}=\frac{1}{L_\xi}\sum_{i=1}^{L_\xi}\delta_{x_i^{(\xi)}}$ satisfy a Large deviation principle, and they concentrate on a measure $\mu(dx)\in\mathfrak{P}=\mathcal{P}^{\beta {  \kappa}}([0,{e^{ - \beta( \gamma^2 - \kappa)}}])$.

\begin{theorem}\label{thm:mainLDP}
Consider the measures $\widehat{\mu}_N^{(\xi)}$, they satisfy a large deviation principle in $\mathcal{P}(\R^+)$ with speed $N^2$ and good rate function $J^{(\xi)}=I^{(\xi)} - \inf I^{(\xi)}$, where the rate function $I^{(\xi)}$ is defined as 

\begin{equation}
	\label{eq:functional}
	 I^{(\xi)}( \mu) = -H^{(\xi)}(\mu) - K^{(\xi)}(\mu) - { M^{(\xi)}(\mu)},\end{equation}
where $\kappa=\kappa(\xi)=L_\xi/N$, and $H^{(\xi)}(\mu)$, $K^{(\xi)}(\mu)$, $M^{(\xi)}(\mu)$ have the following forms:
\begin{equation}
	H^{(\xi)}(\mu)=\frac{\kappa^2}{2} \int \int\left(\ln | x^{\theta} - y^{\theta}| + \ln | x^{\eta} - y^{\eta}|\right) d \mu(x) d\mu(y)\, ;
\end{equation}
\begin{itemize}
	\item[(i)] if $\xi\in (-\gamma^2,0]$, $\kappa=\gamma^2-|\xi|$ and
	\begin{equation}K^{(\xi)}(\mu)={\kappa}\int\int_{\gamma^2-1}^{|\xi|}\ln(1-x^{\theta}e^{{ \beta}\theta u})du\, d\mu(x)\,, \qquad 
	M^{(\xi)}(\mu)= \kappa \eta |\xi| \int  \ln(x)\, d\mu(x);
	\end{equation}
	
	\item[(ii)] if $\xi\in(0,1-\gamma^2]$, $\kappa=\gamma^2$ and
	\begin{equation}K^{(\xi)}(\mu)={\kappa}\int\int_{0}^{1-\gamma^2-\xi}\ln(1-x^{\theta}e^{-{ \beta}\theta u})du\, d\mu(x)\,, \qquad M^{(\xi)}(\mu)= \kappa \theta \xi \int  \ln(x)\, d\mu(x);\end{equation}
	
	\item[(iii)] if $\xi\in(1-\gamma^2,1]$, $\kappa=1-\xi$ and
	\begin{equation}K^{(\xi)}(\mu)={\kappa}\int\int_{{(1-\gamma^2-\xi)}}^{0}\ln(1-x^{\eta}e^{-\beta\eta u})du\, d\mu(x)\,, \qquad M^{(\xi)}(\mu)= \kappa \theta \xi \int  \ln(x)\, d\mu(x).\end{equation}
\end{itemize}
Furthermore, the rate function $I^{(\xi)}$ is strictly convex and has a unique minimizer $\mu(dx)\in \mathfrak{P}$.
\end{theorem}

\begin{remark}
    \label{rem:concentration}
    We observe that the rate function $I^{(\xi)}(\mu)$ is infinite if the measure $\mu$ has a discrete component, therefore we deduce that the equilibrium measure, i.e. the minimizer of the rate function, is absolutely continuous with respect to the Lebesgue measure.
\end{remark}

\begin{remark}
\label{rem:simplifications}
	One can get rid of $\theta$ by considering $\widehat{\nu}_N$ the pushforward of 
	$\widehat{\mu}_N$ by $x \to x^{1/\theta}$. This is equivalent to consider  $y_i^{(\xi)} = e^{ - \epsilon'_N l_i^{(\xi)}}$ with $\epsilon'_N = \epsilon_N\theta$, therefore $\widehat{\nu}_N = L_\xi^{-1} \sum \delta_{y_i^{(\xi)}}$. This new measure satisfies the same large deviations principle with the same rate function, but we replace $\theta$ by $1$, $\eta$ by $\eta/\theta$ and $\beta$ by $\theta \beta$. An analogous result holds for $\eta$ in place of $\theta$. We use this property to find an explicit minimizer of \eqref{eq:functional}.
 \end{remark}

\noindent We notice that the space of equilibrium measures is subject to the constraint $\mu(x)\leq (\beta \kappa x)^{-1}$. This feature naturally divides the phase space into two distinct regions depending on the values of the parameter $\beta$. A \textit{subcritical} region, where the upper constraint is not active, meaning that the strict inequality holds for $x\in \supp(\mu(x))$; and a \textit{supercritical} region, where there are intervals where the upper constraint is active, meaning that there are intervals $(c,d)\subseteqq \supp(\mu(x))$ such that $\mu(x)=\frac{1}{\beta \kappa x}$ for $x\in(c,d)$. Through Riemann--Hilbert analysis, we are able to explicitly compute the minimizer of $I^{(\xi)}$ in the subcritical case \textit{and} in the supercritical case. As a byproduct of our analysis, borrowing the term from dimer models \cite{Kenyon2007}, we are able to fully describe  what we call the \textit{arctic curve}, a curve in the $(\xi,x)$--plane which divide the space in two regions: the frozen region and the liquid region.

One of the main objects that makes this analysis possible is the function $J_{c_0,c_1}(s)$

\begin{equation}
\label{eq:J_c0_c1}
			J_{c_0,c_1}(s) = (c_1s+c_0) \left(\frac{s+1}{s}\right)^\frac{1}{\nu}\,, \qquad \nu\geq 1
\end{equation}
which has the following properties 
\begin{lem}
		\label{lem:J_easy}
		Consider the mapping $J_{c_0,c_1}(s)$ \eqref{eq:J_c0_c1}, if the branch-cut is chosen is such a way that $J_{c_0,c_1}(s)$ is analytic in $\mathbb{C}\setminus[-1,0]$ and $J_{c_0,c_1}(s)\sim c_1s$ as $s\to\infty$, then the following holds
		\begin{enumerate}
			\item $J_{c_0,c_1}(s)$ has two critical points $s_a\leq-1, s_b\geq0$
			
			\begin{equation}
   \label{eq:crit_points}
			\begin{split}
			& s_a = -\frac{\nu-1}{2\nu} - \frac{1}{2\nu c_1}\sqrt{4c_0c_1\nu + c_1^2(\nu-1)^2} \\
			& s_b = -\frac{\nu-1}{2\nu} + \frac{1}{2\nu c_1}\sqrt{4c_0c_1\nu + c_1^2(\nu-1)^2}
			\end{split}
			\end{equation}
			 which are mapped to  $a=J_{c_0,c_1}(s_a)$ and $b=J_{c_0,c_1}(s_b)$; 
			 
			\item $J_{c_0,c_1}(s)$ is real for $s\in (-\infty,s_a]\cup [s_b,+\infty)$ and along two complex conjugate arcs $\sigma_+,\sigma_-$ joining $s_a$ and $s_b$;
			\item both $\sigma_+$ and $\sigma_-$ are bijectively mapped in $[a,b]$;
			\item defining $\mathbb{H}_\nu = \{ z\in \mathbb{C}\,\vert\, -\frac{\pi}{\nu} \leq \Arg(z) \leq \frac{\pi}{\nu}\}$ and let $D$ be the area enclosed by the union of $\sigma_+,\sigma_-$, then $J_{c_0,c_1}\,:\, D\setminus[-1,0] \to \mathbb{H}_\nu\setminus[a,b]$ and $J_{c_0,c_1}\,:\, \mathbb{C} \setminus \overline D \to \mathbb{C}\setminus[a,b] $ are two bijections.
		\end{enumerate}
	\end{lem}
	\noindent The proof of this lemma can be found in \cite{Claeys2014,Charlier2022}, see also Fig \ref{fig:RHP}, from this moment on we call $\sigma = \sigma_+ \cup \sigma_-$. 
    
    The minimization problem \eqref{eq:functional} is not in standard form, one can reduce to it by mapping $x\to x^\zeta$ for $\zeta=\theta,\eta$. By doing this, we are naturally lead to consider the two following model problems  

    \begin{problem}
\label{mod_prob_1_intro}
    Let $\nu >1$, consider the functional $\mathcal{I}_\nu[\omega_{\nu}]$ defined as

    \begin{equation}
    	\label{eq:simplified_functional_1_intro}
       \mathcal{I}_\nu[\omega_{\nu}] =  - \frac{1}{2} \int_0^1 \int_0^1\left(\ln (| x^{\nu} - y^{\nu}|) + \ln( | x - y|)  \right)\omega_{\nu}(dx) \omega_{\nu}(dy)\,  - \frac{1}{\kappa}\int_0^1\int_{n_1}^{n_2} \ln(1 - x^\nu e^{-\beta \alpha u}) d u \omega_{\nu}(dx)- m_1 \int_0^1 \ln(x) \omega_{\nu}(dx) 
    \end{equation}
where $\alpha\,,\rho>0,m_1\geq 0, n_2\geq n_1$ and assume $n_1\alpha = -\rho(\gamma^2-\kappa)$, find $\omega_{\nu}(dx)\in \mathcal{P}^{\beta\rho\kappa}([0,e^{-\rho\beta(\gamma^2-\kappa)}])$, such that it minimize the previous functional. 
\end{problem}

\begin{problem}
\label{mod_prob_2_intro}
    Let $\nu >1$, consider the functional $\mathcal{I}_1[\omega_1]$ defined as

    \begin{equation}
    	\label{eq:simplified_functional_2_intro}
       \mathcal{I}_1[\omega_1] =  - \frac{1}{2} \int_0^1 \int_0^1\left(\ln (| x^{\nu} - y^{\nu}|) + \ln( | x - y|)  \right)\omega_1(dx) \omega_1(dy)\,  - \frac{1}{\kappa}\int_0^1\int_{n_1}^{n_2} \ln(1 - x e^{-\beta \alpha u}) d u \omega_1(dx)- m_1 \int_0^1 \ln(x) \omega_1(dx) 
    \end{equation}
where $\alpha\,,\rho>0,m_1\geq 0, n_2\geq n_1$ and assume $n_1\alpha = -\rho(\gamma^2-\kappa)$, find $\omega_1(dx)\in \mathcal{P}^{\beta\rho\kappa}([0,e^{-\rho\beta(\gamma^2-\kappa)}])$, such that it minimize the previous functional. 
\end{problem}

To explicitly express the minimizers for these two model problems, we rely on the conformal map $J_{c_0,c_1}(s)$ introduced in \eqref{eq:J_c0_c1}, its critical points $s_a, s_b$ \eqref{eq:crit_points}, and its conformal inverse $I_+(x)$ evaluated along the upper branch cut $\sigma_+$. 
For a given set of parameters $(\nu, \alpha, \beta, \kappa, \rho, m_1, n_1, n_2)$, we first introduce the algebraic constants and spectral roots. 

\noindent For \textbf{Model Problem \ref{mod_prob_1_intro}} ($\omega_\nu$), we define:
\begin{equation}\label{eq:intro_constants_1}
\begin{split}
    A &= \exp\left[ \frac{\alpha\beta\kappa}{\nu} \left( \nu + 1 + m_1 + \frac{\nu}{\kappa}(n_2-n_1) \right) \right], \quad B = \exp\left[ \frac{\alpha\beta\kappa}{\nu} (1 + m_1) \right], \\
    s_1 &= A\left(\frac{B-1}{A-B}\right),\qquad s_2 = \frac{B-1}{A-B}, \\
    K_1 &= \left[ e^{n_1\alpha\beta} \left(\frac{A(B-1)}{B(A-1)} \right) \right]^{\frac{1}{\nu}},\qquad K_2 = \left[ e^{n_2\alpha\beta} \left( \frac{B-1}{A-1} \right) \right]^{\frac{1}{\nu}}.
\end{split}
\end{equation}

\noindent  For \textbf{Model Problem \ref{mod_prob_2_intro}} ($\omega_1$), the parameters change to:
\begin{equation}\label{eq:intro_constants_2}
\begin{split}
    A &= \exp\left[ \alpha\beta\kappa \left( \nu + 1 + m_1 + \frac{1}{\kappa}(n_2-n_1) \right) \right], \quad B = \exp\left[ \alpha\beta\kappa\left(1 + m_1 + \frac{n_2-n_1}{\kappa}\right) \right], \\
    s_1 &= A\left(\frac{B-1}{A-B}\right),\qquad s_2 = \frac{B-1}{A-B}, \\
    K_1 &= e^{n_2\alpha\beta} \left(  \frac{A(B-1)}{B(A-1)} \right)^{\frac{1}{\nu}},\qquad K_2 = e^{n_1\alpha\beta} \left( \frac{B-1}{A-1} \right)^{\frac{1}{\nu}}.
\end{split}
\end{equation}
In both cases, the constants $c_0$ and $c_1$, that define the  map $J_{c_0,c_1}(s)$, are fixed by the system:
\begin{equation}
    c_1 = \frac{K_1 - K_2}{s_1 - s_2}, \qquad c_0 = \frac{K_2 s_1 - K_1 s_2}{s_1 - s_2}\,,
\end{equation}
which consequently fixes the spectral edges $a = J_{c_0,c_1}(s_a)$ and $b = J_{c_0,c_1}(s_b)$.

The equilibrium measure exhibits a phase transition depending on whether the hard upper constraint $x_{\max} = e^{-\rho \beta(\gamma^2-\kappa)}$ becomes active. This happens exactly when $s_1=s_b$. For this reason we define the subcritical and supercritical regime as follows

\begin{itemize}
\item[] \textbf{Subcritical Regime:} for the Model Problem \ref{mod_prob_1_intro} the subcritical regime refers to the values of $\beta$ such that $s_1 \leq s_b$, for the Model Problem \ref{mod_prob_2_intro} to the values of $\beta$ such that $s_2\geq s_b$.
\item[] \textbf{Supercritical Regime:} for the Model Problem \ref{mod_prob_1_intro} the supercritical regime refers to the values of $\beta$ such that $s_1 > s_b$, for the Model Problem \ref{mod_prob_2_intro} to the values of $\beta$ such that $s_2 < s_b$.

\end{itemize}

\begin{remark}
    It is worth noticing that one can express the phase transition also in terms of $\beta$; specifically, we can define the critical value $\beta_c$ as the \textbf{unique} value of $\beta$ such that $s_1=s_b$ for Model Problem \ref{mod_prob_1_intro} and $s_2=s_b$ for Model Problem \ref{mod_prob_2_intro}. Then for $0\leq \beta < \beta_c$ we are in the subcritical regime, while for $\beta > \beta_c$ we are in the supercritical regime. 
\end{remark}

We summarize the explicit limit shapes in the following theorem (proved in Section \ref{sec:rhp}).

\begin{theorem}[Explicit Limit Shapes]\label{thm:intro_main_shapes}
Let $\omega \in \{\omega_\nu, \omega_1\}$ be the unique equilibrium measure minimizing Model Problem \ref{mod_prob_1_intro} or \ref{mod_prob_2_intro}. 
\begin{enumerate}[label=\roman*.]
    \item \textbf{Subcritical Regime:} The upper constraint is globally inactive. The measure is supported on a single band $(a, b)$ and admits the density:
    \begin{equation}
        \omega(x) = \frac{1}{\pi \beta \rho\kappa x}\Arg\left(\frac{s_1-I_+(x)}{s_2-I_+(x)}\right) \mathds{1}_{x \in (a, b)}\,.
    \end{equation}
    
    \item \textbf{Supercritical Regime:} The measure saturates the upper constraint on the interval $(b, x_{\max})$. 
    The density takes the form:
    \begin{equation}
        \omega(x) = \begin{cases}
            \frac{1}{\pi \beta \rho\kappa x}\Arg\left(\frac{s_1-I_+(x)}{s_2-I_+(x)}\right)  \quad & x\in (a,b) \\
            \frac{1}{\beta\rho\kappa x} \quad & x\in (b, x_{\max})
        \end{cases}\,.
    \end{equation}
\end{enumerate}
\end{theorem}

Finally, the equilibrium measure $\mu(x)$ of the generalized Muttalib--Borodin process across the full temporal evolution $\xi$ is recovered by evaluating the appropriate model problem with the parameter identification outlined below:

\begin{cor}[Parameter Identification]
\label{cor:intro_identification}
The equilibrium measure $\mu(x)$ is given by $\eta x^{\eta-1}\omega_\nu(x^\eta)$ when applying Model Problem \ref{mod_prob_1_intro}, or by $\eta x^{\eta-1}\omega_1(x^\eta)$ when applying Model Problem \ref{mod_prob_2_intro} (symmetrically with $\theta$ if $\eta > \theta$). The parameters map exactly as follows:
\begin{enumerate}
    \item[\textbf{Case 1:}] $\theta > \eta$ 
    \begin{enumerate}[label=(\roman*)]
        \item  $\xi\in(-\gamma^2,0]$: Use MP--\ref{mod_prob_1_intro} with $\nu=\frac{\theta}{\eta}$, $\kappa=\gamma^2 - |\xi|$, $m_1=\frac{|\xi|}{\kappa}$, $n_1=-|\xi|$, $n_2=1-\gamma^2$, $\alpha=\theta$, $\rho=\eta$. Then $\mu(x)=\eta x^{\eta-1}\omega_\nu(x^\eta)$.
        \item  $\xi\in(0,1-\gamma^2]$: Use MP--\ref{mod_prob_1_intro} with $\nu=\frac{\theta}{\eta}$, $\kappa=\gamma^2$, $m_1=\frac{\theta \xi}{\eta \kappa}$, $n_1=0$, $n_2=1-\gamma^2-\xi$, $\alpha=\theta$, $\rho=\eta$. Then $\mu(x)=\eta x^{\eta-1}\omega_\nu(x^\eta)$.
        \item  $\xi\in(1-\gamma^2,1]$: Use MP--\ref{mod_prob_2_intro} with $\nu=\frac{\theta}{\eta}$, $\kappa=1-\xi$, $m_1=\frac{\theta\xi}{ \eta\kappa}$, $n_1=1-\gamma^2-\xi$, $n_2=0$, $\alpha=\eta$, $\rho=\eta$. Then $\mu(x)=\eta x^{\eta-1}\omega_1(x^\eta)$.
    \end{enumerate}
    \item[\textbf{Case 2: }] $\eta > \theta$ 
    \begin{enumerate}[label=(\roman*)]
        \item  $\xi\in(-\gamma^2,0]$: Use MP--\ref{mod_prob_2_intro} with $\nu=\frac{\eta}{\theta}$, $\kappa=\gamma^2 - |\xi|$, $m_1=\frac{|\xi|\eta}{\kappa\theta}$, $n_1=-|\xi|$, $n_2=1-\gamma^2$, $\alpha=\theta$, $\rho=\theta$. Then $\mu(x)=\theta x^{\theta-1}\omega_1(x^\theta)$.
        \item $\xi\in(0,1-\gamma^2]$: Use MP--\ref{mod_prob_2_intro} with $\nu=\frac{\eta}{\theta}$, $\kappa=\gamma^2$, $m_1=\frac{\xi}{\kappa}$, $n_1=0$, $n_2=1-\gamma^2-\xi$, $\alpha=\theta$, $\rho=\theta$. Then $\mu(x)=\theta x^{\theta-1}\omega_1(x^\theta)$.
        \item  $\xi\in(1-\gamma^2,1]$: Use MP--\ref{mod_prob_1_intro} with $\nu=\frac{\eta}{\theta}$, $\kappa=1-\xi$, $m_1=\frac{\xi}{ \kappa}$, $n_1=1-\gamma^2-\xi$, $n_2=0$, $\alpha=\eta$, $\rho=\theta$. Then $\mu(x)=\theta x^{\theta-1}\omega_\nu(x^\theta)$.
    \end{enumerate}
\end{enumerate}
\end{cor}

Up to our knowledge, this is the first time that an explicit solution for the equilibrium measure of a Muttalib--Borodin ensemble with an upper constraint is found. The explicit nature of the solution allows us to describe the shape of the plane partition across its full temporal evolution $\xi$ both in the subcritical and in the supercritical regimes. In Figure \ref{fig:plot_density} we plot the shape of the equilibrium measure $\mu(x)$ for several values of the parameters.

  \begin{figure}
        \centering
        \includegraphics[width=\linewidth]{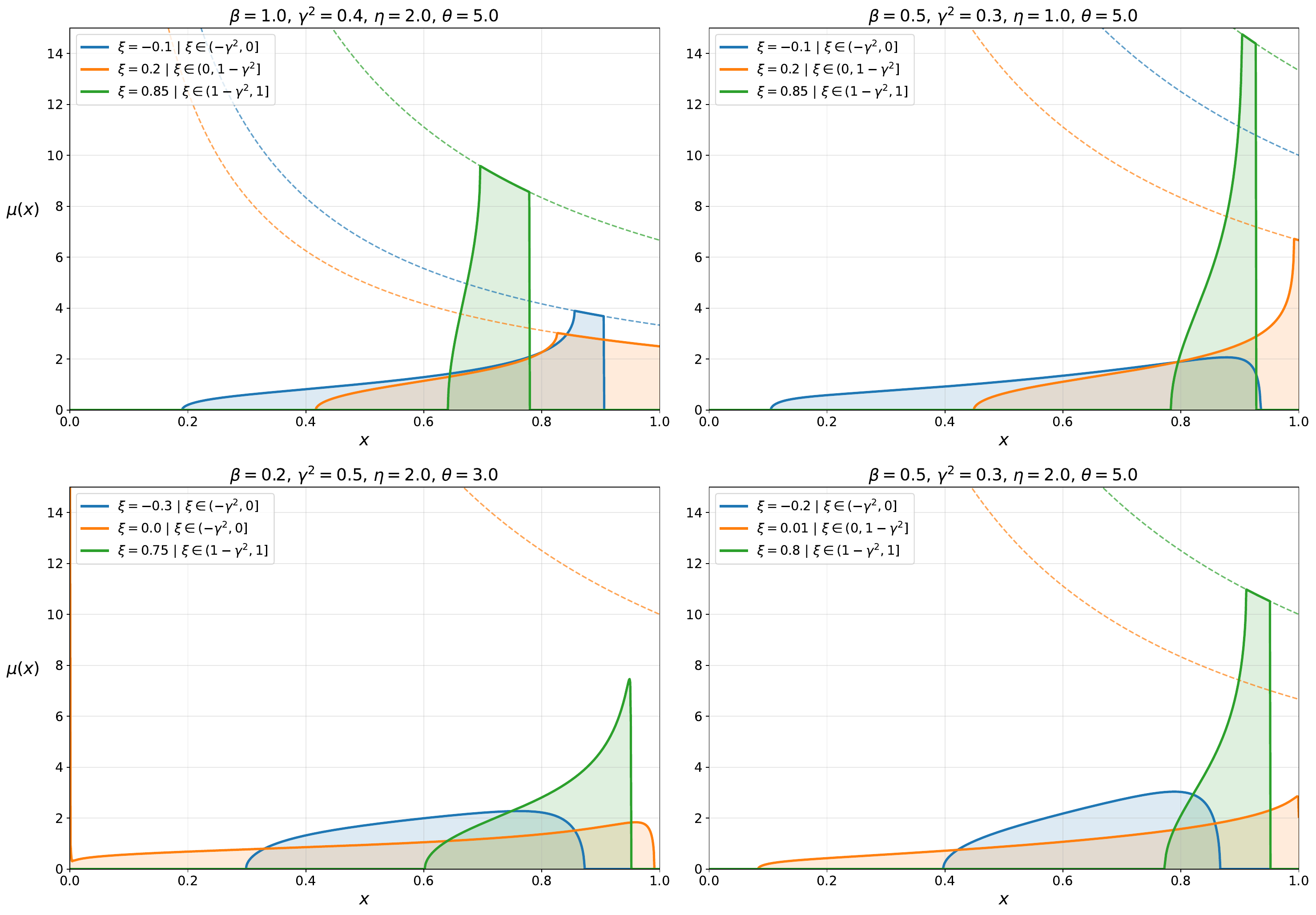}
        \caption{Several plots of the density functions $\mu(x)$.}
        \label{fig:plot_density}
    \end{figure}

    \begin{remark}
    In the subcritical regime, we notice that assuming $\xi\neq 0,1-\gamma^2$ then

    \begin{align}
        &J_{c_0,c_1}(s) = a + \frac{J''(s_a)}{2}(s-s_a)^2 + o((s-s_a)^3) \quad s\to s_a \\
         &J_{c_0,c_1}(s) = b + \frac{J''(s_a)}{2}(s-s_b)^2 + o((s-s_b)^3) \quad s\to s_b
    \end{align}
    therefore $\omega(x)$ decays as a square-root nearby the endpoints.
    In all regimes, if $\xi=0$, then $s_a=-1$, $a=0$ and there exists a constant $C_0$ such that
    \begin{equation}
        \omega(x) \sim C_0 x^{- \frac{1}{\nu +1}}\,, \quad x\to 0^+\,,
    \end{equation}
    this is the same behavior found in \cite{Claeys2014}, where the author notices that this is not the behaviour of the equilibrium measure of random matrix ensemble where the typical exponent is $1/2$.  
    In particular, this implies that if $\xi=0$ then there exists a constant $\widetilde C_0$ such that

    \begin{equation}
        \mu(x) \sim \wt C_0 x^{\frac{\theta\eta}{\theta+\eta} -1}\,, \quad x\to 0^+\,.
    \end{equation}
    We notice that, since $\theta,\eta>0$, $\mu(x)$ is always integrable and the exponent $\frac{\theta\eta}{\theta+\eta} -1 \in (-1,+\infty)$. This behaviour is different from the classical random matrix ensembles, where the decay is typically $1/2$. In a more general setting, one can have equilibrium measures with rational $\pm \frac{p}{q}$ decay \cite{Bergre2009UniversalSL}, but our exponent ranges over the interval $(-1,\infty)$.  
    \end{remark}

    Given the previous result, we are able to fully describe the density plot of the generalized Muttalib--Borodin process throughout its full temporal evolution $\xi$ by numerically computing the equilibrium measure $\omega$ for the appropriate model problem and then applying the parameter identification of Corollary \ref{cor:intro_identification}, see Figure \ref{fig:plot_density_beta_0}.

    \begin{figure}
        \centering
        \includegraphics[width=\linewidth]{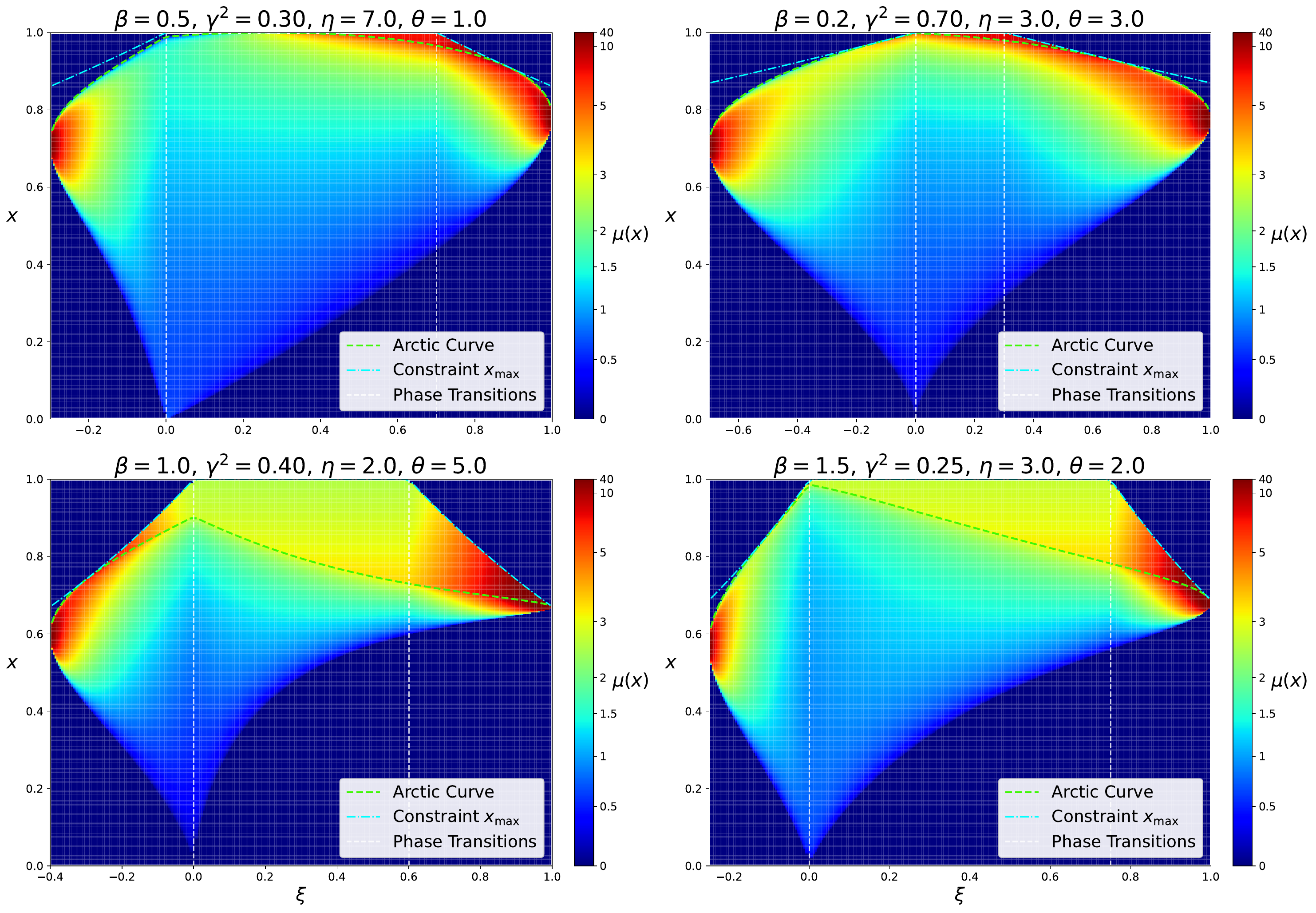}
        \caption{Generalized Muttalib--Borodin process across its full temporal evolution. The arctic curve is plotted in green, above this curve the particles are ``frozen'', meaning that they are as dense as possible, while below this curve they are ``free''.}
        \label{fig:plot_density_beta_0}
    \end{figure}

    \begin{remark}
        We notice that one can obtain the asympotic shape of the plane partition by inverting the relation \eqref{eq:transformation}, deducing that 

        \begin{equation}
            \frac{1}{N}\sum_{j=1}^N \delta_{\frac{\ell^{(\xi)}_j}{N}} \xrightarrow{N\to\infty} \nu(\lambda) \,d\lambda= \beta\mu(e^{-\beta \lambda})e^{-\beta \lambda}\,d\lambda\,.
        \end{equation}
        We plot this asymptotic shape in Figure \ref{fig:plot_shape_partition} for different values of the parameters.
    \end{remark}

    \begin{figure}
        \centering
        \includegraphics[width=\linewidth]{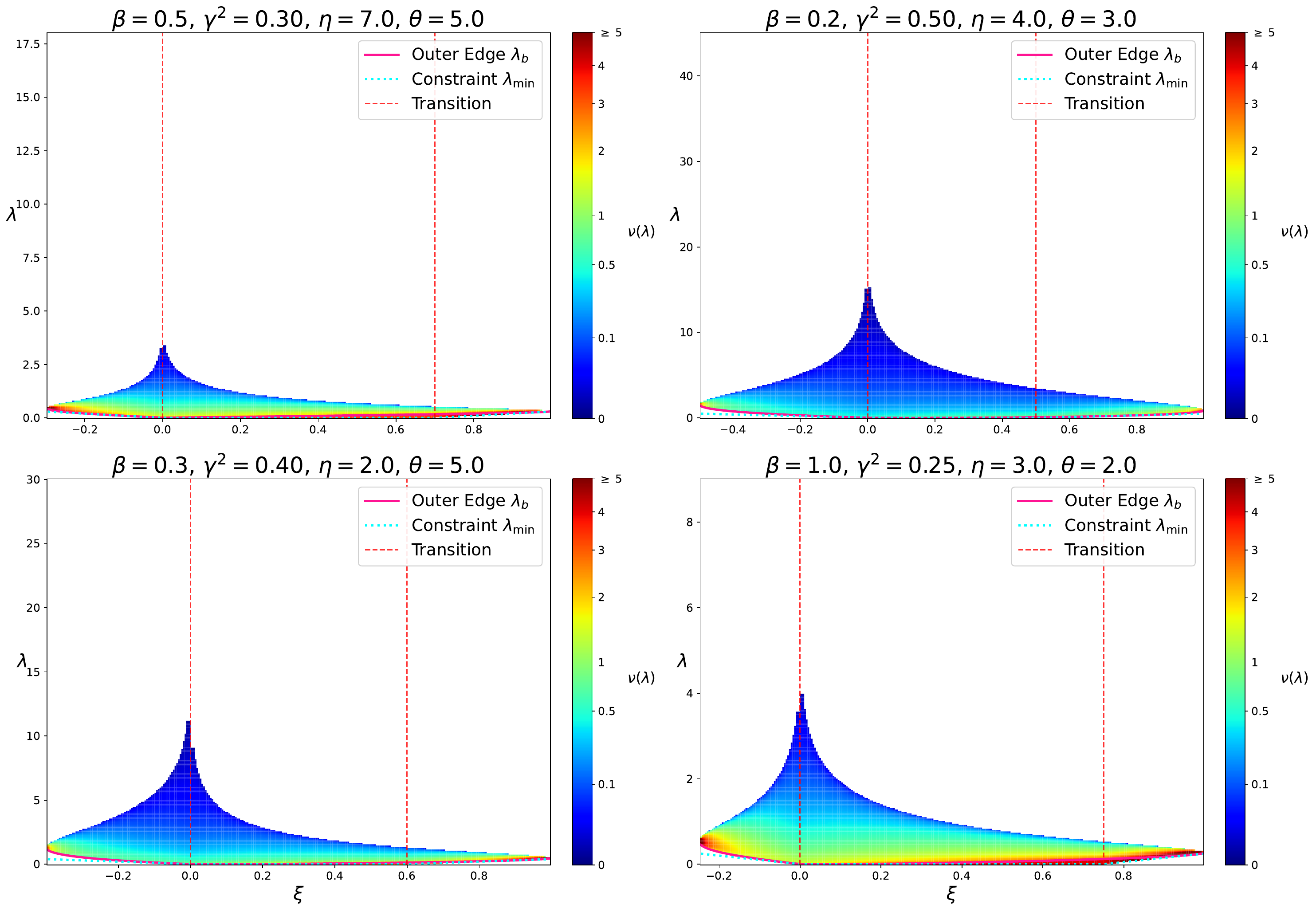}
        \caption{Several plots of the asymptotic shape $\nu(\lambda)$.}
        \label{fig:plot_shape_partition}
    \end{figure}

    Finally, one can also obtain the explicit expression of the arctic curve, which is the curve in the $(\xi,x)$ plane that divides the frozen region (where the upper constraint is active) from the liquid region (where the upper constraint is inactive). This curve is given, up to the parameter identification in Corollary \ref{cor:intro_identification},  by the curve $(\xi, J_{c_0,c_1}(s_b))$, see Figure \ref{fig:plot_density_beta_0}.

    We now briefly discuss the techniques used to obtain the results of this paper and what is the general strategy of the proofs.

\paragraph{Large deviation techniques.}
The distribution of the discrete Muttalib-Borodin ensemble closely resembles the distribution of $\beta$-ensembles. 
Those ensembles are  $N$-tuples $x=(x_1,\dots,x_N)$ points on the real line distributed according to the distribution 
\[\frac{1}{Z}\Delta(x)^{\beta} e^{ - N \sum_{i=1}^N V(x_i) } d x_1 \cdots d x_N \]
where $\Delta(x) = \prod_{ 1 \leq i < j \leq N} |x_i -x_j|$, $V$ is a potential and $Z$ a normalization constant. When one considers $\beta =1,2$, this ensemble represents the eigenvalue distribution of a random matrix whose law is $Z^{-1} e^{ - N \Tr V((H))} dH$ where $dH$ is the Lebesgue measure on the set of $N \times N$ real symmetric matrices ($\beta=1$) or complex Hermitian matrices ($\beta =2$). 
To study the limit behaviour of such ensembles, one can use the theory of large deviations \cite{Dembo2010}. More precisely, introducing for every $N \in \N$ the (random) empirical measure $ \displaystyle \hat{\mu}_N = \frac{1}{N}\sum_{i=1}^N \delta_{x_i}$, and $I_V$ the functional on $\mathcal{P}(\R)$ defined as

\[ I_V[\mu] = - \frac{\beta}{2} \int \int \ln(|x-y|) d \mu(x) d \mu(y) + \int V(x) d \mu(x) +C \in [0, + \infty] \] 
saying the sequence of such measures satisfies a large deviation with speed (usually) $N^2$ and some rate function $I_V$ means informally that for every probability measure $\mu$ on $\R$

\[ \P[ \hat{\mu}_N \approx \mu] = e^{ - N^2 I_V(\mu)} \]

If the function $I_V$ has a unique minimizer $\mu_{eq}$ (also called \emph{equilibrium measure}), such a large deviation principle gives in fact a law of large numbers with $\mu_{eq}$ as a limit.

Such results have been proven for confining potential $V$ (meaning that the measure $\mu_{eq}$ has compact support) in \cite{BAGuionnet1997} and for non-confining compact support in \cite{Hardy2012}. Similar results for the eigenvalues of Haar-distributed unitary matrices were also proved in \cite{HiaiPetz2000} and for the eigenvalues of Ginibre matrices in \cite{BA1998}. There are two main differences between the the discrete Borodin-Muttalib ensemble in this paper and the classical $\beta$-ensembles:
\begin{itemize}
	\item The term 
	$\Delta(x)$ will be replaced by $\Delta(x^{\theta})\Delta(x^{\eta})$ where $\eta>0,\theta >0$ and $x^{\theta}= (x_1^{\theta},\dots, x_N^{\theta})$.
	\item The particles $x_1,\dots, x_N$ do not lie on the whole real line, but on a discrete subset which will have roughly the form $\{ e^{-\beta \ell /N} : \ell \in \N^* \}$  for a given $\beta >0$ and the Lebesgue measure $dx$ is replaced by the counting measure on this subset. 
\end{itemize}
The first difference, which is a consequence of the the bi-orthogonal structure of our model (with $x_1,\dots,x_N$ still lying on the real line), was investigated by \cite{ESS2011,Butez2017}. One then still gets a large deviation principle by replacing the logarithmic term  $  \displaystyle \int \int \ln(|x-y|) d \mu(x) d \mu(y) $ in $I_V$ by $ \displaystyle \int \int \ln(|x^{\theta}-y^{\theta}|) d \mu(x) d \mu(y) $ and $  \displaystyle \int \int \ln(|x^{\eta}-y^{\eta}|) d \mu(x) d \mu(y) $. In fact, further generalizations were made for more general settings (see for instance, \cite{Chafai2014} for particles in $\R^N$ for general two-particles interactions and \cite{GZ2019,Berman2018} for generalizations to particles lying in more abstract topological spaces).

Regarding the discrete aspect, similar models were investigated for particles lying in $\{\ell/ N: \ell \in \Z \}$ (see for instance \cite{BGG2017} as well as \cite{Joh00,Feral2008} for large deviation principles). A feature of those models is that the limit points of $\hat{\mu}_N$ have to be measures that have a density with respect to the Lebesgue measure that is bounded by $\ell^{-1}$. The model we consider has analogous features, but we must consider a different discretization. 

Finally, one can also mention that large deviation principles also exists directly for the profile (or in other words the height function) of some plane partition models. We refer for instance to  \cite{CoLaPr1998} for plane partitions in a given box and to Lecture 22 and 23 in \cite{Gorin2021} for plane partitions on a $N\times N $ square weighted according to $q^{Volume}$. This last model is the one that is closer to our own. We nevertheless chose to study the slices of those partitions since then we can use Riemann-Hilbert techniques to get a description of the limit profile.

\paragraph{Riemann--Hilbert problem analysis.} Riemann--Hilbert problems (RHPs) provide a fundamental framework for deriving explicit formulas for some relevant quantities in various applications. Generally, an RHP is a boundary value problem in which one seeks a (matrix-valued) complex function that satisfies prescribed boundary conditions along a contour, with a normalization condition \cite{Ablowitz2003}. They have been fruitfully applied in the theory of integrable systems. Specifically, using this tool, one can get precise asymptotic for Orthogonal Polynomials and Discrete Orthogonal polynomials, see \cite{Deift2000,DiscreteOPbook,Kuijlaars2003,Deift2006RiemannHilbertMI} and the reference therein, and explicit solutions to (stochastic) integrable  PDE such as the Nonlinear Schr\"odinger equation, the Korteweg-De Vries equations, the Modified Korteweg-De Vries equation and the Kardar-Parisi-Zhang equation \cite{Babelon2003,Deift2002,Cafasso2021,Ablowitz2003nls,gkogkou2025,gkogkou2026painleveuniversalityclassesmaximal}. Other fields where the Riemann--Hilbert approach was extensively applied are Random matrix theory \cite{Deift2000,Deift1999} and determinantal point processes \cite{Borodin03}. In these contexts, RHPs are used to find explicit formulas for the equilibrium measure of classical random matrix ensembles \cite{bleher2008,eynard2015random} - which in most cases is equivalent to finding a minimizer of some logarithmic potential \cite{Saff2024} - and to compute some relevant probabilistic quantity, such as the gap probability and the largest eigenvalue/particle distribution \cite{Baik1999,Deift2000}. In connection to our work, RHPs were also applied to Muttalib--Borodin ensembles. In \cite{Kuijlaars2018TheLU,Molag2021}, the authors obtained the asymptotic behaviour of the correlation kernel in the case $\nu=\frac{1}{r}$, $r\in \N$, to do so, the authors rephrase this problem as a $(r+1)\times(r+1)$ RHP.  In \cite{Claeys2014,Charlier2022,Wang2022} the authors used this technique to obtain an explicit expression for the equilibrium measure of the Laguerre and Jacobi Muttalib--Borodin ensemble in the non-constrained one-cut regime, meaning that the equilibrium measure is supported on one segment $(a,b)$. Specifically, the authors find the minimizer $\mu(dx)\in \cP([0,1])$ - the space of probability measures in $(0,1)$-  of 

\[
I_V[\mu] = \frac{1}{2}\int\int_{\mathcal{I}^2} \ln(|x^\nu - y^\nu|)\mu(dx)\mu(dy) + \frac{1}{2}\int\int_{\mathcal{I}^2} \ln(|x - y|)\mu(dx)\mu(dy) + \int_\mathcal{I} V(x)\mu(dx)\,,
\]
where in \cite{Claeys2014} $\mathcal{I} = (0,+\infty)$, $V(x)$ satisfying some specific properties and $\nu\geq 1$, while in \cite{Charlier2022} $\mathcal{I} = (0,1)$, $V(x)=0$ and $\nu>0$. Following a standard procedure, they showed that the previous minimization problem is equivalent to a RHP involving two distinct functions

\[ g(z) = \int_{\mathcal{I}}\ln(|x -y|)\mu(dx)\,, \quad g_\nu(z) = \int_{\mathcal{I}}\ln(|x^\nu -y^\nu|)\mu(dx)\,.\]
To solve this problem in the case  $\nu\geq1$, the authors of \cite{Claeys2014} introduced the map $J_{c_0,c_1}(s)$ \eqref{eq:J_c0_c1} to transform the RHP for the function $g(z),g_\nu(z)$ into a RHP for only one function $M(z)$, this allowed them to find the explicit expression of $g(z),g_\nu(z)$ and $\mu(dx)$. In \cite{Charlier2022}, the author generalized this approach to the case $0<\nu<1$.
More recently, in \cite{wang2025biorthogonalpolynomialsrelatedquantum,Wang2025}, the authors considered a more general version of the Muttalib--Borodin ensemble. They obtained an explicit expression for the equilibrium measure via a vector-valued RHP and studied the transition regime between hard and soft-edge. 
In this paper, we enforce the RHP analysis to get an explicit expression of the equilibrium measure for a Jacobi-like Muttalib--Borodin ensemble in the non-constrained and constrained one-cut case, see Theorem \ref{thm:intro_main_shapes}.

The remaining part of the paper is organized as follows: in Section \ref{sec:ldp} we prove Theorem \ref{thm:mainLDP} and in section \ref{sec:rhp} we prove Theorem \ref{thm:intro_main_shapes} and Corollary \ref{cor:intro_identification}.
\begin{figure}[!t]
  \centering

  \begin{tikzpicture}[evaluate={%
      integer \Nleft, \Nright;
      \Nleft = 6;
      \Nright = 6;
    }, scale=.40]
    \draw[->] (-7,0)--(-7,14);
    \foreach \k in {0,...,14}{
      \node[left] at (-7,\k) {$\k$};
    }
    \foreach \k [parse=true] in {-\Nleft,...,\Nright}{%
      \draw[dashed,thick, color=gray] (\k,0)--(\k,14);
      \node[below] at (\k,0) {$\k$};
    }
    \node at (-5,6) [circle,fill,inner sep=1.5pt]{};
    \foreach \k in {5,8}{%
        \node at (-4,\k) [circle,fill,inner sep=1.5pt]{};
        }
    \foreach \k in {3,5,9}{%
        \node at (-3,\k) [circle,fill,inner sep=1.5pt]{};
        }
    \foreach \k in {2,3,6,9}{%
        \node at (-2,\k) [circle,fill,inner sep=1.5pt]{};
        }
    \foreach \k in {1,2,4,7,11}{%
        \node at (-1,\k) [circle,fill,inner sep=1.5pt]{};
        }
    \foreach \k in {0,1,3,5,9,13}{%
        \node at (0,\k) [circle,fill,inner sep=1.5pt]{};
        }
    \foreach \k in {1,2,5,7,10}{%
        \node at (1,\k) [circle,fill,inner sep=1.5pt]{};
        }
    \foreach \k in {2,4,7,9}{%
        \node at (2,\k) [circle,fill,inner sep=1.5pt]{};
        }
    \foreach \k in {3,6,9}{%
        \node at (3,\k) [circle,fill,inner sep=1.5pt]{};
        }
    \foreach \k in {5,7}{%
        \node at (4,\k) [circle,fill,inner sep=1.5pt]{};
        }
    \foreach \k in {6}{%
        \node at (5,\k) [circle,fill,inner sep=1.5pt]{};
        }
   
  \end{tikzpicture}
    \hspace{4em}
      \begin{tikzpicture}[evaluate={%
      integer \Nleft, \Nright;
      \Nleft = 6;
      \Nright = 6;
    }, scale=.40]
    \draw[-] (-7,0)--(-7,14);
    \draw[thick, color=gray] (-7.1,0)--(-6.9,0); 
    \node[left] at (-7,0) {$0$};
    \draw[thick, color=gray] (-7.1,7)--(-6.9,7); 
    \node[left] at (-7,7) {$0.5$};
    \draw[thick, color=gray] (-7.1,14)--(-6.9,14); 
    \node[left] at (-7,14) {$1$};
    \foreach \k [parse=true] in {-\Nleft,...,\Nright}{%
      \draw[dashed,thick, color=gray] (\k,0)--(\k,14);
      \node[below] at (\k,0) {$\k$};
    }
    \node at (-5,0.3679*14) [circle,fill,inner sep=1.5pt]{};
    \foreach \k in {0.4346*14,0.2636*14}{%
        \node at (-4,\k) [circle,fill,inner sep=1.5pt]{};
        }
    \foreach \k in {0.6065*14,0.4346*14,0.2231*14}{%
        \node at (-3,\k) [circle,fill,inner sep=1.5pt]{};
        }
    \foreach \k in {0.7165*14,0.6065*14,0.3679*14,0.2231*14}{%
        \node at (-2,\k) [circle,fill,inner sep=1.5pt]{};
        }
    \foreach \k in {0.8465*14,0.7165*14,0.5134*14,0.3114*14,0.16*14}{%
        \node at (-1,\k) [circle,fill,inner sep=1.5pt]{};
        }
    \foreach \k in {0.1145*14,0.2231*14,0.4346*14,0.6065*14,0.8465*14,1*14}{%
        \node at (0,\k) [circle,fill,inner sep=1.5pt]{};
        }
    \foreach \k in {0.8465*14,0.7165*14,0.4346*14,7,0.1889*14}{%
        \node at (1,\k) [circle,fill,inner sep=1.5pt]{};
        }
    \foreach \k in {0.7165*14,0.5134*14,0.3114*14,0.2231*14}{%
        \node at (2,\k) [circle,fill,inner sep=1.5pt]{};
        }
    \foreach \k in {0.6065*14,0.3679*14,0.2231*14}{%
        \node at (3,\k) [circle,fill,inner sep=1.5pt]{};
        }
    \foreach \k in {0.4346*14,0.3114*14}{%
        \node at (4,\k) [circle,fill,inner sep=1.5pt]{};
        }
    \foreach \k in {0.3679*14}{%
        \node at (5,\k) [circle,fill,inner sep=1.5pt]{};
        }
   %%% particles have to be rescaled: TO FINISH
  \end{tikzpicture}
  \caption{To the left: the particle configuration $(l(t))_t$ on $\{-M+1,\dots,N-1\}\times\N$ corresponding to the plane partition in Figure \ref{fig:pp_mb}. To the right: the rescaled particle configuration $(x(t))_t$ on $\{-M+1,\dots,N-1\}\times[0,1]$.
  } 
  \label{fig:pp_ips}
\end{figure}

\paragraph{Acknowledgments.}
The authors want to thank  Dan Betea, Mattia Cafasso, and Tom Claeys for the fruitful discussions.
A.O. wants to thank Daniel Naie for his help with the pictures. A.O. and J.H. have met at SLMath (former MSRI) during the thematic semester ``Universality in random matrix theory and interacting particle systems'', where they have started to discuss the topic of large deviations for bi-orthogonal ensembles.
A.O. and G.M. thank the Institute Mittag-Leffler for offering an opportunity of an in-person discussion during the trimester ``Random Matrices and Scaling Limits''.
G.M. was partially supported by the Swedish Research Council under grant no. 2016-06596 while the author was in residence at Institut Mittag-Leffler in Djursholm, Sweden during the fall semester of 2024. A.O. was partially supported by the ULIS project (2023-09915) funded by  Region Pays de la Loire and by the ERC-2019-ADG Project: 884584 (LDRAM).

\section{Large deviation principle of the plane partition}
\label{sec:ldp}

In this section, we prove Theorem \ref{thm:mainLDP}.
Let us first analyze the two terms composing the rate function $I^{(\xi)}(\mu)$ \eqref{eq:functional}, one coming from the ``particle'' interaction term, one from the single ``particle'' potential (see interpretation of the plane partition as a particle system).

Taking the logarithm of the double product in~\eqref{eq:MBdistribution}, we get (to simplify the notation, we suppress the apex $(\xi)$ in $\ell_i^{(\xi)}$ and in $\widehat\mu^{(\xi)}_N$)

\begin{equation}
\ln\prod_{ 1 \leq i,j \leq L_\xi \atop i\neq j}(q^{\eta\ell_j}-q^{\eta\ell_i})(q^{\theta\ell_j}-q^{\theta\ell_i})
=\sum_{ 1 \leq i,j \leq L_\xi\atop i\neq j}\ln(q^{\eta\ell_j}-q^{\eta\ell_i})+\ln(q^{\theta\ell_j}-q^{\theta\ell_i})
\end{equation}
In the $q\to1^-$ limit~\eqref{eq:transformation}, this is equal to 
\begin{equation}
\sum_{ 1 \leq i,j \leq L_\xi \atop i\neq j}\ln(x_j^{\eta}-x_i^{\eta})+\ln(x_j^{\theta}-x_i^{\theta})
\approx N^2\left(\frac{L_\xi}{N}\right)^2\frac12 \left[\iint\limits_{x\neq y}\ln(x^\eta-y^\eta)d\widehat\mu_N(x)d\widehat\mu_N(y)+\iint\limits_{x\neq y}\ln(x^\theta-y^\theta)d\widehat\mu_N(x)d\widehat\mu_N(y)\right]\,.
\end{equation}
Now let us look at the contribution coming from the potential term \eqref{eq:weights_d}. We first observe that the common factor $a^{\ell_i}q^{(\eta+\theta)\ell_i}$ is negligible if compared to the other terms in the rate function.
The term $Q^t$ (or $\tilde Q^t$) contributes only with a linear factor, resulting in the term $M^{(\xi)}(\mu)$ in \eqref{eq:functional}.
So the non-trivial terms left to analyze are  the q-Pochhammer symbols \cite[Chapter 17.2]{NIST:DLMF}.
\begin{itemize}
\item[(i)] When $\xi\leq0$, we have
\begin{equation}
	\begin{aligned}
		&(q^{\theta(\ell_i+1-|\xi|N)};q^\theta)_{N(1-\gamma^2+|\xi|)}=\prod_{j=1}^{N(1-\gamma^2+|\xi|)}(1-q^{\theta(\ell_i-|\xi|N)}q^{\theta j})\\
		=&\prod_{j=1}^{N(1-\gamma^2+|\xi|)}\left(1-\exp\left(-\epsilon\left(\theta\left(-\frac{\ln x_i}{\epsilon}-|\xi|N+j\right)\right)\right)\right)=\prod_{j=1}^{N(1-\gamma^2+|\xi|)}\big(1-x_i^\theta e^{-\frac{\beta}{N}[-|\xi|N+j]}\big).
	\end{aligned}
\end{equation}
Once we take the logarithm
\begin{equation}
	\begin{aligned}
		\sum_{j=1}^{N(1-\gamma^2+|\xi|)}\ln\big(1-x_i^\theta e^{{\theta}|\xi|}e^{-\theta\frac{1+j}{N}}\big)\approx	N\int_0^{1-\gamma^2+|\xi|}\ln\big(1-x_i^\theta e^{-\beta|\xi|}e^{-{\beta\theta}s}\big)ds.
	\end{aligned}
\end{equation}

By the change of variable $u=|\xi|-s$ then, summing over $i$ and dividing by $N^2$ we recover the term $K^{(\xi)}$ in the result.
\item[(ii)] When $\xi\in(0,1-\gamma^2]$, we have
\begin{equation}
	\begin{aligned}
		&(q^{\theta(\ell_i+1});q^\theta)_{N(1-\gamma^2-\xi)}=\prod_{j=1}^{N(1-\gamma^2-\xi)}(1-q^{\theta(\ell_i+j)})\\
		=&\prod_{j=1}^{N(1-\gamma^2-\xi)}\Big(1-x_i^\theta e^{-\epsilon\theta j}\Big).
	\end{aligned}
\end{equation}
Once we take the logarithm
\begin{equation}
	\begin{aligned}
		\sum_{j=1}^{N(1-\gamma^2-\xi)}\ln\big(1-x_i^\theta e^{-\beta\theta\frac{j}{N}}\big)\approx	{N}\int_0^{1-\gamma^2-\xi}\ln\big(1-x_i^\theta e^{-{\beta\theta}s}\big)ds.
	\end{aligned}
\end{equation}

\item[(iii)] When $\xi\in(1-\gamma^2,1]$, we have
\begin{equation}
	\begin{aligned}
		&(q^{\eta(\ell_i+N+1-\xi N-\gamma^2 N)},q^\eta)_{N(\gamma^2-1+\xi)}=\prod_{j=1}^{N(\gamma^2-1+\xi)}(1-q^{\eta(\ell_i+j+N(1-\gamma^2-\xi))}q^{\theta j})\\
		=&\prod_{j=1}^{N(\gamma^2-1+\xi)}\Big(1-x_j^\eta e^{-\epsilon\eta(N(1-\gamma^2-\xi)+j)}\Big).
	\end{aligned}
\end{equation}
Once we take the logarithm
\begin{equation}
	\begin{aligned}
		\sum_{j=1}^{N(\gamma^2+\xi-1)}\ln\big(1-x_i^\eta e^{-\eta\beta(1-\xi-\gamma^2)}e^{-\beta\eta\frac{j}{N}}\big)\approx	{N}\int_0^{\gamma^2+\xi-1}\ln\big(1-x_i^\eta e^{-\eta(1-\xi-\gamma^2)}e^{-\beta\eta s}\big)ds.
	\end{aligned}
\end{equation}
We conclude by the change of variable $u=1-\xi-\gamma^2-s$.
\end{itemize}
This heuristic already shows that the functional $J^{(\xi)}$ in Theorem \ref{thm:mainLDP} is a good candidate for the rate function of the model. We first show that it is indeed a ``good'' rate function.
\begin{proposition}
	Consider $J^{(\xi)}$ in Theorem \ref{thm:mainLDP}, it is a good rate function, i.e. it is lower semi-continuous and its level sets $\{\mu : J^{(\xi)}(\mu)\leq C\}$ are compact. Furthermore, it is strictly convex, so it has a unique minimizer.
\end{proposition}
\begin{proof}
	We need to prove only that $I^{(\xi)}(\mu)$ \eqref{eq:functional} is lower semi-continuous. Indeed, since $\mathfrak{P}$ is compact for the weak topology, it follows automatically that the level sets are compact.
	
	Strict convexity of $I^{(\xi)}(\mu)$ comes from the fact that we  can write 
	\[ I^{(\xi)}(\mu) = \frac{\kappa^2}{2} \Big(\mathcal{E}(p_{\eta}* \mu) + \mathcal{E}(p_{\theta}* \mu) \Big)+K^{(\xi)}(\mu) +M^{(\xi)}(\mu)   \]
	where $p_a (x) = x^a$. We notice that $K^{(\xi)}(\mu)$ and $M^{(\xi)}(\mu)$ are linear terms and $\mathcal{E}(\mu):= - \int \ln| x- y| d \mu(x) d\mu(y)$,  is strictly convex (see the proof of Lemma 2.6.2 and in particular equation 2.6.19 in \cite{AGZ}).
	
	To prove the lower semi-continuity, we follow a standard argument (see once again the proof of Lemma 2.6.2 in ~\cite{AGZ}) and we approximate $I^{(\xi)}(\mu)$ by a continuous analogue denoted $I^{(\xi)}_{\texttt{M}}(\mu)$ obtained replacing in $I^{(\xi)}(\mu)$ $- \ln$ by $(-\ln )\wedge \texttt{M}$ for $\texttt{M}\geq 0$, here $x\wedge y = \min(x,y)$. Then, $I^{(\xi)}(\mu)=\sup_{\texttt{M}} I^{(\xi)}_{\texttt{M}}(\mu)$ and $I^{(\xi)}(\mu)$ is lower semi-continuous. 
    \end{proof}
	Let us denote $\N_t := [ M- L_\xi, + \infty[\cap \N$. We will not directly work with the measure $\P$ on the set of strictly increasing $N$-tuple $\ell$ but with the unrenormalized measure $\overline{\P}$ on $\N_t^{L_\xi}$ defined by

	\begin{equation} \overline{\P}( \ell) =  \prod_{1 \leq i,j \leq N }  | Q ^{\ell_j} - Q ^{\ell_i}|^{1/2} | \tilde{Q} ^{\ell_j} - \tilde{Q} ^{\ell_i}|^{1/2} \prod_{ 1 \leq i \leq N} w_d(\ell_i) \end{equation}

	Here a $L_\xi$-tuple sampled according to $\overline{\P}$ is not increasingly ordered a priori, but by symmetry of the formula, it is easy to see that sampling $\ell$ according to $\overline{\P}$ and reordering it is equivalent (up to renormalizing) to sampling $\ell$ by 
	$\P$. So sampling $\widehat{\mu}_N$ through $\overline{\P}$ is the same as sampling it through $\P$.
	For this measure, we will prove large deviation upper and lower bounds with rate function $I^{(\xi)}(\mu)$, which are stated in Lemmas \ref{lem:UB} and \ref{lem:LB}. From this we obtain that $1/N^2 \ln Z_d$ converges to $-\inf I^{(\xi)}(\mu)$ and consequently the large deviation priciple for $J^{(\xi)}$. 

\begin{remark}
One can rule deviations outside $\mathfrak{P}$. Indeed, if $ \mu \notin \mathfrak{P}$, then there is an interval $ I = ]a,b[$ such that $a > 0$ and $ b <1$ and such that $\mu( I) > \int_a^b (\beta x)^{-1} dx = (\ln b - \ln a)/ \beta$. However, if we call $(l_i)_{i\leq L_\xi}$ the increasing reordering of $(\ell_i)_{i \leq L_\xi}$ since $\widehat{\mu_N} = (L_\xi)^{-1} \sum_{i=1}^N \delta_{x_i}$ where $x_i = e^{ - \epsilon_N l_i}$, $\ell_{i+1} > l_{i}$ and $l_i\in \mathbb{N}_t$, then 
\[ \widehat{\mu}_N(I)= L_\xi^{-1} \#\{ i \in [1,N] :  -\frac{ \ln b}{ \epsilon_N} < l_i <  -\frac{ \ln a}{ \epsilon_N} \}  \leq \frac{1}{L_\xi} \Big( \frac{ \ln b - \ln a }{\epsilon_N} +1 \Big) \]
Since $\lim_{N \to \infty} N \epsilon_N = \beta $ there exists $c > 0$ such that for $N$ large enough, $\P[ \widehat{\mu}_N(I)  > \mu(I) - c ] = 0$, which implies that for any distance $d$ on $\mathfrak{P}$ which indices a metric in the weak topology, there is $c >0$ such that for $N$ large enough $\P[ d( \widehat{\mu}_N, \mu) \leq c] = 0$.
\end{remark}
  
To show that $I^{(\xi)}(\mu)$ is the LDP rate function for the sequence of measures $\widehat \mu_N$, we must show the so-called \textit{Large deviation upper and lower bounds}. Specifically, we must prove that for any $\mu \in \mathfrak{P},$ and $\delta > 0$ 
	\begin{equation}
    \label{eq:sup_inf_lim}
	    \lim_{\delta \to 0} \limsup_{N \to \infty}\frac{1}{N^2}  \ln \overline{\P} [ d( \widehat{\mu}_N, \mu) \leq \delta ]  \leq  - I^{(\xi)}(\mu)\,,\qquad  \lim_{\delta\to0}\liminf_{N\to\infty} \frac{1}{N^2} \ln \overline{\P}  [ d( \widehat\mu_N, \mu) \leq \delta]  \geq - I^{(\xi)}( \mu).
	\end{equation}
  Since the proof for the case {$\beta=0$} is more involved, we postpone it to the end of the section. Here we consider the case $\beta > 0 $.
{\begin{remark} At several points during the proof, we will for convenience's sake abuse the notations and identify $t/N,L_\xi/N, N\epsilon, M/N $ to their respective limits, $\xi,\kappa,\gamma^2$ and $\beta$. Since all those limits are positive and finite, this has not consequence on the proof as it only introduces errors of order $\exp(o(N^2))$. When we consider the case $\beta =0$, we detail the necessary adaptations.
    \end{remark}}
  
    We split the proof of the inequalities \eqref{eq:sup_inf_lim} in the following lemmas.

\begin{lem}[Large deviation upper bound]\label{lem:UB}
	For any $\mu \in \mathfrak{P},$ and $\delta > 0$ 
	\[ \lim_{\delta \to 0} \limsup_{N \to \infty}\frac{1}{N^2}  \ln \overline{\P} [ d( \widehat{\mu}_N, \mu) \leq \delta ]  \leq  - I^{(\xi)}(\mu) \]
\end{lem}

\begin{proof}

To simplify the notation we drop the apex $(\xi)$.
We let 
\begin{equation} f(x,y) = - \frac{1}{2} \Big( \ln|x^{\theta} - y^{\theta}| + \ln|x^{\eta} - y^{\eta}| \Big)  % -\frac{\alpha + (\theta + \eta)/2}{2} \Big( \ln | x | + \ln|y| \Big)  
\end{equation}
and for $\texttt{M} > 0$ 
\begin{equation} f_{\texttt{M}}(x,y) = f(x,y) \wedge \texttt{M}  \,.\end{equation}
Using this notation, we have that 
\begin{equation} H^{(\xi)}(\mu) = \kappa^2 \int\int f(x,y) d \mu(x) d\mu(y).\end{equation}
For $\texttt{M} >0$ and $\mu \in \mathfrak{P}$, we let 
\begin{equation} H^{(\xi)}_{\texttt{M}}(\mu) :=  \kappa^2  \int\int f_{\texttt{M}}(x,y) d \mu(x) d\mu(y). \end{equation}
For any $N \in \N_t$ every $\mu \in \frak{P}$ an  $\ell \in \N_t^{L_\xi}$, we denote $\mu_{\ell} = L_\xi^{-1} \sum_{i=1}^{L_\xi} \delta_{x_i}$. Then using the definition of $\overline{\P}$ we can write down
\begin{equation}  \overline{\P}[ d(\widehat\mu_N,\mu) \leq \delta] = \sum_{ l \in \N_t^{L_\xi}} \exp (-  N^2 W_N(\ell) ) \mathds{1}_{\{ d(\mu_{\ell}, \mu) \leq \delta\}} \prod_{1 \leq i \leq L_\xi} a^{\ell_i} (Q\tilde{Q})^{\ell_i/2} \end{equation} 
with $W_N( \ell) = \infty$ if $\ell_i = \ell_j$ for some $i \neq j$ and otherwise:
\begin{equation}
	W_N(\ell) = W_N^{(1)}(\ell) + W_N^{(2)}(\ell), 
\end{equation}
where
\begin{equation}  W_N^{(1)}( \ell) =  - \frac{1}{2 N^2} \sum_{\substack{1 \leq i ,  j \leq L_\xi \\ i \neq j } }  \ln | q^{l_i\theta} - q^{l_j\theta}| + \ln| q^{l_i\eta} - q^{l_j\eta}|\,, \qquad
	W_N^{(2)}(\ell)=- \frac{1}{N^2} \ln {\frac{w_d(\ell)}{a^\ell(Q\tilde Q)^\frac{\ell}{2}}} \,.
\end{equation}
After the change of variable $\ell_i=-\frac{N}{\beta}\ln x_i$, we obtain the previous expressions in terms of $x_i$ as (with a slight abuse of notation)
\begin{align}  
&W_N^{(1)}(x) =  - \frac{1}{2 N^2} \sum_{\substack{1 \leq i ,  j \leq L_\xi \\ i \neq j } }  \ln | x_i^{\theta} - x_j^{\theta}| + \ln| x_i^{\eta} - x_j^{\eta}| ,
\\
\label{eq:W2N}
&W_N^{(2)}(x)=- \frac{\eta |\xi|}{N} \sum_{i=1}^{\kappa N}\ln x_i - \frac{1}{N^2} \sum_{i=1}^{\kappa N}\sum_{j=1}^{N(1-\gamma^2-|\xi|)}\ln (1-x_i^\theta e^{-\beta |\xi|} e^{-\beta\theta\frac{j}{N}})
\end{align}
for case $(i)$, and analogously for the other two cases.
For any $\texttt{M} >0$ we have that 

\begin{equation}
	W^{(1)}_N(x) \geq H_{\texttt{M}}( \mu_{\ell}) - L_\xi^{-1} \texttt{M} .
\end{equation} 
Since $-\ln x$ is decreasing in $x$, we can bound from above the Riemann sums in \eqref{eq:W2N} by the integral and obtain that 
\begin{equation}
	W^{(2)}_N(x) \geq K(\mu_{\ell})+M(\mu_{\ell}).
\end{equation}
Therefore, we can then write 

\begin{equation}  \label{eq:UB1}
	\overline{\P}[ d(\widehat\mu_N,\mu) \leq \delta] \leq  \sum_{ l \in \N_t^{L_\xi} : d(\mu_{\ell},\mu) \leq \delta} \exp ( - N^{2} H_{\texttt{M}}(\mu_\ell) + \texttt{M} L_\xi +  K(\mu_{\ell})+M(\mu_{\ell}))   \prod_{1 \leq i \leq {L_\xi}} a^{\ell_i}(Q\tilde Q)^{\frac{\ell_i}{2}} 
\end{equation} 
Now choose $L>0$ such that $ L < I(\mu)$. Since $H(\mu) = \sup_{\texttt{M}} H_{\texttt{M}}(\mu) $, there is $\texttt{M} >0$ such that $L < H_{\texttt{M}}(\mu)+K(\mu)+M(\mu)$. 
 Then using the continuity of $H_{\texttt{M}}$ and the lower semi-continuity of $K$ and $M$, there is $\delta >0$ such that $H_{\texttt{M}}(\mu ') + K(\mu')+M(\mu')> L $ for any $\mu'$ such that $d(\mu',\mu) \leq \delta$.  So, putting everything together we get 
\begin{equation}  \label{eq:UB2}
	\overline{\P}[ d(\widehat\mu_N,\mu) \leq \delta] \leq  \sum_{ l \in \N_t^{L_\xi}} \exp ( - N^{2} L + \texttt{M} L_\xi)   \prod_{1 \leq i \leq {L_\xi}} a^{\ell_i}(Q\tilde Q)^{\frac{\ell_i}{2}}.
\end{equation}
Finally we use that 
\begin{equation} \label{eq:UB3} \sum_{ l \in \N_t^{L_\xi}} \prod_{1 \leq i \leq {L_\xi}} a^{\ell_i}(Q\tilde Q)^{\frac{\ell_i}{2}} = \left(\sum_{i=M-L_\xi}^{+\infty}  a^{i}(Q\tilde Q)^{\frac{i}{2}}\right)^{L_\xi} = \left( \frac{(a Q \tilde{Q})^{ M -L_\xi}}{ 1 - a Q \tilde{Q} }\right)^{L_\xi} = \exp ( O( N \ln N)).  \end{equation} 
In the end we have 

\begin{equation} \limsup_{N \to \infty} \frac{1}{N^2} \ln \overline{\P}[ d(\widehat\mu_N,\mu) \leq \delta] \leq - L. \end{equation} 
Since this is valid for every $L < I(\mu)$, we get our upper bound.

\end{proof}

\begin{lem}[Large deviation lower bound]\label{lem:LB}
	For any $\delta > 0$ and $\mu \in \mathfrak{P}$ we have for any distance $d$ that metricizes the weak topology on $\mathfrak{P}$
	
	\begin{equation}  \liminf_N \frac{1}{N^2} \ln \overline{\P}  [ d( \widehat\mu_N, \mu) \leq \delta]  \geq - I^{(\xi)}( \mu). \end{equation} 
\end{lem}

\begin{proof}
To simplify the notation, we drop the apex ${(\xi)}$. It is sufficient to find a sequence $(\tilde{\ell}^N)_{N \in \N}$ such that $\tilde{\ell}^N \in \N_t^{L_\xi}$ such that the sequence $\nu_N = L_\xi^{-1}  \sum_{i=1}^{L_\xi} \delta_{ e^{ - \epsilon \tilde{\ell}^N_i}}$ converges weakly toward $\mu$ and:
\begin{equation}  \liminf_N \frac{1}{N^2} \ln \overline{\P}[ \tilde{\ell}^N ] \geq - I( \mu). \end{equation}

{First, we can assume that $\mu( \{ 0 \}) = 0$ (if not we have $I( \mu) = + \infty$ and the result is obvious). Then we consider the measure $\lambda$ defined on $\R^+$ as $\lambda([a,b]) = \mu([ e^{-b}, e^{ -a}])$. In particular, $\lambda$ is a probability measure on $[ \beta(\gamma^2 - \kappa), + \infty[$ such that $\lambda \ll \text{Leb}_{\R^+}$ and its density is less that $(\beta \kappa)^{-1}$. Let us call $\mathfrak{Q}$ the set of such measures. The bijection that to such a measure $\mu$ associate the measure $\lambda$ and its inverse are both continuous for the weak topology on $\mathfrak{P} \setminus \{ \mu: \mu(\{0 \}) >0 \}$ and $\mathfrak{Q}$ (it is indeed the push-forward by the function $- \ln$, which is a continuous function from $]0, e^{ - \beta(\gamma^2 - \kappa)}]$ to $[ \beta(\gamma^2 - \kappa), + \infty[$).}

We then have using a change a variables that
\[ I(\mu) = \tilde{I}( \lambda) = - \tilde{H}( \lambda) - \tilde{K}( \lambda) - \tilde{M}( \lambda), \]
where

\[\tilde{H}( \lambda) =  \frac{\kappa^2}{2} \int \ln | e^{ - \theta x} - e^{ - \theta y}| d \lambda(x) d \lambda(y) + \frac{\kappa^2}{2} \int \ln | e^{ - \eta x} - e^{ - \eta y}| d \lambda(x) d \lambda(y) 
\]
and where if $\xi \leq 0$,

	\begin{equation}\tilde{K}(\lambda)=\kappa\int\int_{\gamma^2-1}^{|\xi|}\ln(1-e^{(-x + \beta u)\theta})du\, d\lambda(x);\end{equation}
 \begin{equation} \tilde{M}(\lambda)= \kappa \eta |\xi| \int  x \, d\lambda(x);
 \end{equation}
and with similar definitions for $0 <\xi \leq 1 - \gamma^2$ and $\xi > 1- \gamma^2$. 
Let us assume that the lower bound holds for $\mu \in\mathfrak{P}$ such that $\lambda$ is compactly supported in some interval $[\beta(\gamma^2 - \kappa), \texttt{M}]$ for some $\texttt{M} >0$. We will verify this statement at the end of the proof in \ref{lem:useless_lemma}.
Also, let us assume that the following proposition is true:
\begin{proposition}\label{prop:CompactSupport}
	Let $\lambda \in \mathfrak{P}$ such that $\tilde{I}( \lambda) < + \infty$. There exists a family of compactly supported measure $(\lambda_{\texttt{M}})_{\texttt{M} >0}$ with densities bounded above by $( \beta \kappa)^{-1}$ such that $\lambda_{\texttt{M}}$ converges weakly toward $\lambda$ and $\tilde{I}( \lambda_{\texttt{M}})$ converges toward $\tilde{I}( \lambda)$ when $\texttt{M}$ goes to $\infty$.
\end{proposition} 
Let $\lambda \in \mathfrak{P}$ such that $\tilde{I}^{(\xi)}( \lambda) < + \infty$. 
Let $\varepsilon >0$ and $\delta >0$. Using the Proposition above, we can find $\lambda' \in \mathcal{P}$ that is compactly supported and  such that $\tilde{I}^{(\xi)}( \lambda') \leq \tilde{I}^{(\xi)}( \lambda) + \varepsilon$. 
Furthermore, if we denote $\mu'$ the measure defined by $\mu'([a,b])= \lambda'([- \ln a, -\ln b])$ since the function $\lambda \mapsto \mu$ is continuous, we can also assume that $\lambda'$ is such that $d( \mu',\mu) \leq \delta /2 $.

Then, there exists a sequence $(\tilde{\ell}^N)_{N \in \N}$ such that 
$\tilde{\ell}^N \in \N_t^{L_\xi}$ such that the sequence $\nu_N = L_\xi^{-1}  \sum_{i=1}^{L_\xi} \delta_{ e^{ - \epsilon \tilde{\ell}^N_i}}$ converges weakly toward $\mu'$. For $N$ large enough we have that $d(\nu_N,\mu')\leq \delta/2$ which implies $d(\nu_N,\mu)\leq \delta$. It follows then that :
\begin{eqnarray*}
\frac{1}{N^2} \ln \overline{\P}[ d(\hat{\mu}_N, \mu) \leq \delta  ] & \geq & \frac{1}{N^2} \ln \overline{\P}[ d(\hat{\mu}_N, \mu) \leq \delta  ] \\
& \geq & \frac{1}{N^2} \ln \overline{\P}[ \hat{\mu}_N = \nu_N ]  \\
& \geq & \frac{1}{N^2} \ln \overline{\P}[ \tilde{\ell}^N]
\end{eqnarray*}
Taking the the $\liminf$ in the inequality above, we have that 
\begin{eqnarray*}
\liminf_{N \to \infty}\frac{1}{N^2} \ln \overline{\P}[ d(\hat{\mu}_N, \mu) \leq \delta  ] & \geq & -I( \mu') \\
& \geq & -I( \mu) - \varepsilon
\end{eqnarray*}
Optimizing in $\varepsilon >0$ gives us the result. 

\end{proof}

Now we prove the two claimed statements.

\begin{lem}\label{lem:useless_lemma} The lower bound holds for $\mu \in\mathfrak{P}$ such that $\lambda$ is compactly supported in some interval $[0, \texttt{M}]$ where $\texttt{M}>0$. 
	\end{lem}
\begin{proof}
	
We will follow the step 1 of the proof of Lemma 2.16 in \cite{DasDim2022}. Let us look at $F_{\lambda}(x) = \lambda([0, x ])$. For $N \in \N$, for $N \in \N$, $1 \leq i \leq L_\xi$ we denote the following quantiles of $\lambda$

\[ y_i^N = \inf \left\{ t \in [0, + \infty], F_{\lambda}(\xi) = \frac{i - 1/2}{L_\xi} \right\} \]
and also $y_0=\beta(\gamma^2 - \kappa)$ and $y_{L_\xi + 1} = \texttt{M} + \beta \kappa$
define for $1 \leq i \leq L_\xi $,
\begin{equation}  \tilde{\ell}^N_{i }  = \sup  \{ j \in \N_t , \epsilon j  \leq y_{ L_\xi -i +1}. \} \end{equation}
Then we have that $(\tilde{\ell}^N_{i })_{ 1 \leq i \leq L_\xi}$ is a strictly decreasing sequence of integers.

Since the density of $\lambda$ is bounded above by $(\beta \kappa)^{-1}$, we have that $|F_{\lambda}(x) - F_{\lambda}(y)| \leq (\beta \kappa)^{-1} |x -y|$ which implies that $ (L_\xi)^{-1} \sum_{i =1}^{L_\xi} \delta_{ \epsilon_N \tilde{\ell}^N_i}$ converges toward $\lambda$ and so the sequence $\nu_N$ converges toward $\mu$.

Using the same notation as in upper bound lemma, the goal is to prove
\begin{equation}  \limsup_N W_N( \tilde{\ell}^N) \leq I( \mu) = \tilde{I}( \lambda).
 \end{equation} 
First let us compare $W_N^{(1)}( \tilde{\ell}^N)$ with $- \tilde{H}^{(\xi)}$. 
We have 
\begin{equation} W_N^{(1)}( \tilde{\ell}^N) = - \frac{1}{2 N^2} \sum_{\substack{1 \leq i ,  j \leq L_\xi \\ i \neq j }} \left( \ln \left| \frac{e^{- \epsilon\tilde{\ell}^N_i\theta} - e^{- \epsilon\tilde{\ell}_j\theta}}{ \epsilon(\tilde{\ell}^N_i - \tilde{\ell}^N_j)} \right| + \ln \left| \frac{e^{- \epsilon \tilde{\ell}^N_i\eta} - e^{- \epsilon \tilde{\ell}_j\eta}}{ \epsilon(\tilde{\ell}^N_i - \tilde{\ell}^N_j)} \right|\right) -  \frac{1}{ N^2} \sum_{\substack{1 \leq i ,  j \leq L_\xi \\ i \neq j } } \ln ( \epsilon(\tilde{\ell}^N_i - \tilde{\ell}^N_j)). \end{equation}

Let 
\begin{equation}
	Y^{(1)}_N( \tilde{\ell}^N )=  - \frac{1}{2 N^2} \sum_{\substack{1 \leq i ,  j \leq L_\xi \\ i \neq j }} \left( \ln \left| \frac{e^{- \epsilon\tilde{\ell}^N_i\theta} - e^{- \epsilon\tilde{\ell}_j\theta}}{ \epsilon(\tilde{\ell}^N_i - \tilde{\ell}^N_j)} \right| + \ln \left| \frac{e^{- \epsilon \tilde{\ell}^N_i\eta} - e^{- \epsilon \tilde{\ell}_j\eta}}{ \epsilon(\tilde{\ell}^N_i - \tilde{\ell}^N_j)} \right|\right) 
	\end{equation}
and 
\begin{equation}
	Y^{(2)}_N( \tilde{\ell}^N ) = -   \frac{1}{ N^2} \sum_{\substack{1 \leq i ,  j \leq L_\xi \\ i \neq j } } \ln ( \epsilon(\tilde{\ell}^N_i - \tilde{\ell}^N_j)).
	\end{equation}
Since $\displaystyle f_a : (x,y) \mapsto \ln \Big( \Big| \frac{e^{-x a} -e^{-y a}}{ x -y} \Big| \Big)$  is continuous and bounded on $[0,\texttt{M}]^2$ and since that $f_a(x,x) = \ln |a|$, we have that 

\[  \lim_{N \to \infty} \Big(Y^{(1)}_N( \tilde{\ell}^N ) - \frac{\ln(\eta \theta)L_\xi}{2N^2} \Big) = - \frac{\kappa^2}{2} \int\int \ln \Big( \Big| \frac{e^{-x \theta} -e^{-y \theta}}{ x -y} \Big| \Big) + \ln \Big( \Big| \frac{e^{-x \eta} -e^{-y \eta}}{ x -y} \Big| \Big) d\lambda(x) d\lambda(y). \]

To bound $Y_N^{(2)}$ we now follow Step 2 from the proof of Lemma 2.16 from \cite{DasDim2022}. i.e 

\begin{equation}
	\begin{aligned}
		- \frac{N^2}{2} Y_N^{(2)}(\tilde{\ell}^N ) + \sum_{1 \leq i < j \leq L_\xi} \frac{1}{\tilde{\ell}_i^N - \tilde{\ell}_j^N} \geq \sum_{1 \leq i < j \leq L_\xi} \ln( y_{L_\xi -i +1} - y_{L_\xi -j +1}).
		\end{aligned}
	\end{equation}
Furthermore we have that 
\[ \sum_{1 \leq i < j \leq L_\xi} \frac{1}{\tilde{\ell}_i^N - \tilde{\ell}_j^N}  = O(N \ln N) \]
and using equation the proof of (2.38) from \cite{DasDim2022}, we have 
we have 
\begin{eqnarray*} L_\xi^2\int\int_{ x > y} \ln| x - y| d \lambda(x) d \lambda(y) &\leq& \sum_{i=1 }^{L_\xi+1} \sum_{j=i+1}^{L_\xi+1} \ln( y_{j} - y_{i-1}) + \frac{1}{2}\sum_{i=1}^{L_\xi +1} \ln(y_i - y_{i-1}).
\end{eqnarray*}
Using the fact that $\texttt{M}+ \beta \kappa > y_{j } - y_i >(2L_\xi)^{-1} ( \beta \kappa)$ for $j >i$ (the $1/2$ factor is here to take the case $i=0,j=1$ into account), we have that  $ \ln( y_i - y_j )  \leq O( \ln N)$ and therefore  
\begin{eqnarray*} L_\xi^2\int\int_{ x > y} \ln| x - y| d \lambda(x) d \lambda(y) &\leq& \sum_{1 \leq i,j \leq L_\xi \atop i \neq j } \ln( y_{j} - y_{i}) + O(N \ln N).
\end{eqnarray*}
Putting everything together we get that 

\[ Y^{(2)}_N( \tilde{\ell}^N )  \leq  - \kappa^2 \int\int \ln| x - y| d \lambda(x) d \lambda(y) + o(1) \] 
and so, puting the limits of $ Y^{(1)}_N( \tilde{\ell}^N ) $ and $ Y^{(2)}_N( \tilde{\ell}^N ) $ together

\[ W_N^{(1)}(  \tilde{\ell}^N)  \leq - \tilde{H}( \lambda) +o(1). \]

For $W_N^{(2)}$ since the support of $\lambda$ is in $[\beta( \gamma^2 - \kappa), \texttt{M}]$ and $\epsilon \tilde{\ell}_1^N \leq \texttt{M}$, we have that the first summand converges toward $\tilde{M}( \lambda)$. Further in we consider the following Riemann sums 
\[ F_N(x) = \frac{1}{N} \sum_{i=1}^{N(1-\gamma^2-|\xi|)}\ln (1-e^{-\theta x-\beta |\xi|-\beta\theta\frac{j}{N}}). \]
We have that on $[\beta( \gamma^2 - \kappa),\texttt{M}]$, the sequence $F_N$ converges uniformly toward 
\[ F(x) =  \int_{\gamma^2-1}^{|\xi|} \ln( 1 - e^{(-x + \beta u)\theta}) du \]
and so, 
the second summand converges toward $\tilde{K}^{(\xi)}( \lambda)$. So we do have that $ \limsup_N W_N( \tilde{\ell}^N) = \tilde{I}(\lambda)$
Ergo, one has that for any $\delta > 0$ and $N$ large enough:

\begin{eqnarray} \overline{\P}[ d( \widehat{\mu}_N, \mu) \leq \delta]  &\geq& \overline{\P}[ \ell^N = \tilde{\ell}^N ] \\
&\geq & {\exp ( - N^2( W_N(\tilde{\ell}_N)) \Big(a \sqrt{Q \tilde{Q} } \Big)^{ \sum_{1 \leq i \leq L_\xi} \tilde{\ell}^N_i}}  \\
	&\geq& \exp ( - N^2( \tilde{I}(\lambda) + o(1))) \Big(a \sqrt{Q \tilde{Q} } \Big)^{ \sum_{1 \leq i \leq L_\xi} \tilde{\ell}^N_i} \\
	&\geq & \exp ( - N^2( I^(\mu) + o(1))),
\end{eqnarray} 
where we used that since $\epsilon \tilde{\ell}_{i}^N =O(N)$, then 

\[ \Big(a \sqrt{Q \tilde{Q} } \Big)^{ \sum_{1 \leq i \leq L_\xi}\tilde{\ell}^N_i} \geq \exp( O(N)). \]

\end{proof}

Finally, the following proposition concludes the proof of \eqref{eq:sup_inf_lim}, in the case $\beta\neq0$.

\begin{proof}[Proof of Proposition \ref{prop:CompactSupport}]
	In this proof, we will denote $a = \beta( \gamma^2 - \kappa)$. We define the following measure as a compact approximation of $\lambda$ for $\texttt{M} > a+ \beta \kappa = \beta\gamma^2 $
	\begin{equation}
		\lambda_{\texttt{M}}=\frac{1}{\beta\kappa}Leb |_{[a,\texttt{M'}]}+\lambda |_{[\texttt{M'},\texttt{M}]},
	\end{equation}
and $\texttt{M'}$ is such that $(\beta\kappa)^{-1} \texttt{M'}+\lambda([\texttt{M'},\texttt{M}])=1$ (using the fact that $\lambda$ has no atoms and the intermediate value theorem, it is easy to see that such a \texttt{M'} exists). Clearly such a measure converges to $\lambda$ as $\texttt{M}\to\infty$.
	We are left to prove that $\tilde{I}( \lambda_{\texttt{M}})$ converges toward $\tilde{I}( \lambda)$.
	
	Let us analyze separately the three terms composing  $\tilde{I}( \lambda)$. We start with $\tilde{M}(\lambda)$, which is of the form $c\int x d\lambda$.
	If we compare  $\tilde{M}(\lambda)$ and  $\tilde{M}(\lambda_{\texttt{M}})$, their difference gives
	\begin{equation}
		\begin{aligned}
			&\int_0^{\texttt{M'}}x\, d\lambda(x) + \int_{\texttt{M'}}^{+\infty} x\,d\lambda(x)-\int_0^{\texttt{M'}}x\, dx-\int_{\texttt{M'}}^{\texttt{M}}x\, d\lambda(x)\\
			=&\int_0^{\texttt{M'}}x\, d\lambda(x) + \int_{\texttt{M}}^{+\infty} x\,d\lambda(x) -\frac{(\texttt{M'})^2}{2},
		\end{aligned}
	\end{equation}
	which converges to 0 a $\texttt{M}\to\infty$.
	
	For $\tilde{K}(\lambda)$ the result is immediate, since the function $\displaystyle x\mapsto\int \ln(1-e^{-(s+x)})ds$ is continuous and bounded, and $\lambda_{\texttt{M}}\to\lambda$ as $\texttt{M} \to\infty$.
	
	We turn to $\tilde{H}(\lambda)$. We observe that we can decompose $\tilde{H}(\lambda)$ by $\tilde{H}(\lambda_{\texttt{M}})$ in the sum the following integrals:
	\begin{equation}
		\begin{aligned}
			&I_1=-\frac{1}{2} \int_{\texttt{M'}}^{\texttt{M}}\int_{\texttt{M'}}^{\texttt{M}} \ln | e^{ - \theta x} - e^{ - \theta y}| d \lambda(x) d \lambda(y) - \frac{1}{2}  \int_{\texttt{M'}}^{\texttt{M}}\int_{\texttt{M'}}^{\texttt{M}} \ln | e^{ - \eta x} - e^{ - \eta y}| d \lambda(x) \,d \lambda(y),\\
			&I_2=-(\beta\kappa)^{-1}\int_{\texttt{M'}}^{\texttt{M}}\int_a^{\texttt{M'}} \ln | e^{ - \theta x} - e^{ - \theta y}| d \lambda x d \lambda(y) - (\beta\kappa)^{-1}\int_{\texttt{M'}}^{\texttt{M}}\int_a^{\texttt{M'}} \ln | e^{ - \eta x} - e^{ - \eta y}| dx\, d\lambda(y),\\
			&I_3=-\frac{1}{2}(\beta\kappa)^{-1} \int_a^{\texttt{M'}}\int_a^{\texttt{M'}} \ln | e^{ - \theta x} - e^{ - \theta y}| dx\,dy - \frac{1}{2}(\beta\kappa)^{-1} \int_a^{\texttt{M'}}\int_a^{\texttt{M'}} \ln | e^{ - \eta x} - e^{ - \eta y}| dx\,dy.
		\end{aligned}
	\end{equation}
	Obviously, since the integrand is a positive function, $I_1$ converges to $\tilde{H}(\lambda)$ as $\texttt{M}\to\infty$ (and $\texttt{M'}\to a$). For the same reason $I_3$ converges to 0.
	For $I_2$, we perform a further decomposition as
	\begin{equation}
		-(\beta\kappa)^{-1}\int_{\texttt{M'}}^{\texttt{M'}+1}\int_a^{\texttt{M'}} \ln | e^{ - \theta x} - e^{ - \theta y}| d x d \lambda(y) -(\beta\kappa)^{-1}\int_{\texttt{M'}+1}^{\texttt{M}}\int_a^{\texttt{M'}} \ln | e^{ - \theta x} - e^{ - \theta y}| dx d \lambda(y),
	\end{equation}
	plus the same replacing $\theta$ by $\eta$ and $x$ by $y$. Again, since the integrand is bounded, the second term goes to 0, while the first one converges to 0 by the integrability of $\ln|x-y|$ over a compact set and the fact that the density is bounded by $(\beta\kappa)^{-1}$.
\end{proof}

Now we can prove the large deviation principle for the peak of each integer partition $l_1^{(t)}$ in the bulk; we only show it for the case $t=0$, but the proof can be easily adapted.

\begin{proposition}\label{prop:LDPl1} In the case $t=0$ (and therefore $\xi=0$). we have that 
	\begin{equation} \lim_{N \to \infty} \frac{1}{N^2} \ln \P[ \epsilon_N l_1 \leq c] = - F_c \end{equation}
	where $F_c = \min_{ \mu \in \mathcal{P}[e^{ -c},1]\cap \mathfrak{P}} J^{(0)}( \mu) $. 
\end{proposition}

\begin{proof}
First, we observe that his result comes from the fact that 
\[ \{  \epsilon_N l_1 \leq c \} = \{ \widehat{\mu}_N( [ e^{-c},1]) =1 \}.\]
However it is not a direct consequence of the previous LDP since $\{ \mu( [ e^{-c},1]) =1 \}$ isn't an open set of $\mathcal{P}([0,1])$. So a little more work is necessary.  
We express the probability in the statement as
\[ \P[ \epsilon_N l_1 \leq c]
= \frac{\bar{\mathbb{P}}[ \epsilon_N l_1 \leq c ]}{\bar{\mathbb{P}}[ \Omega]} = \frac{\bar{\mathbb{P}}_c[ \Omega]}{\bar{\mathbb{P}}[ \Omega]}
\]
where $\Omega$ denotes the entire space of configuration and $\bar{\mathbb{P}}_c$ is defined as the restriction of  $\bar{\mathbb{P}}$ to configurations in $l \in[0,cN]^N$. We already know that 
\[ \lim_{N \to \infty} \frac{1}{N^2} \log \bar{\P} [ \Omega] = 
 - \inf_{\mu\in \mathfrak{P}} I^{(0)}(\mu) \]
 So we are left with finding $\lim_{N \to \infty}N^{-2} \log \bar{\mathbb{P}}_c[ \Omega]$. 
To do this, we can follow the same steps as for $\bar\P[\Omega]$ and notice that the only difference in the proof will concern the upper bound estimate, where in all the sums of eq. \eqref{eq:UB1}, \eqref{eq:UB2} and \eqref{eq:UB3} instead of considering $l\in\N_t^{L_\xi}$ we take $l\in[0,cN]^N$. In the end we get that 
\[ \lim_{N \to \infty} \frac{1}{N^2} \log \bar{\P} [ \Omega] = 
 - \inf_{\mu\in \mathfrak{P}\cap \mathcal{P}([ e^{ -c},1]} I^{(0)}(\mu) \]
 and from there we get the result.

\end{proof}

\subsection{Generalization to $\beta =0$}

 To conclude the proof of Theorem \ref{thm:mainLDP}, we must consider the case  $\beta =0$ or, equivalently $\lim_{N \to \infty}\epsilon N = 0$. This can be easily achieved with some small adjustments to the previous proof. 
\begin{theorem}\label{thm:LDPlt_beta0}
	Let us assume that $\lim_{N \to \infty} N \epsilon_N =0$ and such that $\ln( \epsilon_N) \gg -N $.
	$\widehat{\mu}_N^{(\xi)}$ satisfy a large deviation principle in $\mathfrak{P} = \mathcal{P}([0,1])$ with speed $N^2$ and good rate function $J^{(\xi)}=I^{(\xi)} - \inf I^{(\xi)}$, where the rate function $I^{(\xi)}$ is defined as 

	\begin{equation}
		I^{(\xi)}( \mu) = -H^{(\xi)}(\mu) - K^{(\xi)}(\mu) - { M^{(\xi)}(\mu)},\end{equation}
	where $\kappa=\kappa(\xi)=L_\xi/N$, and where the definitions of  $K^{(\xi)}(\mu)$ are generalized to $\beta =0$ the following way:

	\begin{itemize}
		\item[(i)] if $\xi\in (-\gamma^2,0]$, $\kappa=\gamma^2-|\xi|$ and
		\begin{equation}K^{(\xi)}(\mu)={\kappa}( 1 - \gamma^2 - \xi)\int\ln(1-x^{\theta})\, d\mu(x);\end{equation}

		\item[(ii)] if $\xi\in(0,1-\gamma^2]$, $\kappa=\gamma^2$ and
		\begin{equation}K^{(\xi)}(\mu)={\kappa}( 1 - \gamma^2 - \xi)\int\ln(1-x^{\theta})\, d\mu(x);\end{equation}
		
		\item[(iii)] if $\xi\in(1-\gamma^2,1]$, $\kappa=1-\xi$ and
		\begin{equation}K^{(\xi)}(\mu)={\kappa}(\xi + \gamma^2 -1) \int\ln(1-x^{\eta})\, d\mu(x).\end{equation}
	\end{itemize}
\end{theorem}

Though the assumption that $ \ln( \epsilon_N) \gg - N $ is technical, its presence will become clear over the course of the proof.  We list here the modification one must make, first regarding the upper bound:

\begin{enumerate}
	\item First, one has to adapt the expression of $W_N^{(2)}(x)$ in equation \eqref{eq:W2N} by replacing $\beta$ with $N \epsilon_N$.
	\item Then, one can still write the equations \eqref{eq:UB2} and \eqref{eq:UB3} but this last equation then becomes

\begin{equation} \label{eq:UB4} \sum_{ l \in \mathbb{N}_t^{L_\xi}} \prod_{1 \leq i \leq {L_\xi}} a^{\ell_i}(Q\tilde Q)^{\frac{\ell_i}{2}} = \Big(\sum_{i=M- L_\xi}^{+\infty}  a^{i}(Q\tilde Q)^{\frac{i}{2}}\Big)^{L_\xi} = \Big( \frac{(a Q \tilde{Q})^{M -L_\xi}}{ 1 - a Q \tilde{Q} }\Big)^{L_\xi} = \exp ( o( N^2))  
 \end{equation} 
which allows us to complete the upper bound. The second point in particular illustrates why we chose to include the technical assumption that $\ln( \epsilon_N) = o(N)$.

\end{enumerate}

 For the lower bound, we still use the same definition for $\lambda$ but we must now be careful that the density of $\lambda$ is not bounded (and a priori may not even exist). Therefore we need to add an approximation step.
 
 That is, we need to approach any $\lambda$ by measures with bounded densities. For this, we will introduce for every measure $\nu \in \mathcal{P}( [0, + \infty[)$ the quantile function of $\nu$ defined for every $t \in [0,1[$ by 
 
 \[ Q_{\nu}(\xi) = \sup \{ x \in \mathbb{R}, \nu([0, x]) \leq t \}. \] 
  With this definition we have for $\nu \in \mathcal{P}( [0, + \infty[)$
 \[ \int_{0}^{+ \infty} f(x) d \nu(x) = \int_{0}^{1} f(Q_{\nu}(\xi)) dt .\]
 For $\zeta >0$ We then define $\nu^{(\zeta)}$ by 

\[ Q_{\nu^{(\zeta)}}(\xi) = \zeta t + Q_{\nu}(\xi) .\]
We can notice that for every $a <b$, $\nu^{(\zeta)}( [a,b]) \leq \frac{b-a}{\zeta}$ and therefore $\nu^{( \zeta)}$  is a measure which is continuous with respect to the Lebesgue measure and whose density is upper bounded by $\frac{1}{\zeta}$. Now, {restricting ourselves to case \textit{(i)} (the other cases are similar)} let us prove that for every $\lambda \in \mathcal{P}([0, + \infty[)$ 

\[ \lim_{\zeta \to 0} I^{(\xi)}( \lambda^{(\zeta)}) =  I^{(\xi)}( \lambda) . \]

For this, we can prove that $\tilde{H}^{(\xi)}( \lambda^{(\zeta)}), \tilde{M}^{(\xi)}( \lambda^{(\zeta)})$ and $\tilde{K}^{(\xi)}( \lambda^{(\zeta)})$ converge toward $\tilde{H}^{(\xi)}( \lambda), \tilde{M}^{(\xi)}( \lambda)$ and $\tilde{K}^{(\xi)}( \lambda)$ when $\eta$ goes to $0$. First, for the function $\tilde{M}^{(\xi)}$. 

\[ \tilde{M}^{(\xi)}( \lambda^{(\zeta)})  = \kappa \eta |\xi| \int_0^1 ( Q_{\lambda}(\xi) + \zeta t) dt =  \tilde{M}^{(\xi)}( \lambda) + \frac{ \eta\kappa \zeta |\xi| }{2} \]
the result is straightforward

For the function $K^{(\xi)}$ we have 
\begin{eqnarray*}
	\tilde{K}^{(\xi)}( \lambda^{(\zeta)}) &=& \kappa \int \ln( 1 - e^{-x\theta}) d  \lambda^{(\zeta)}(x) \\
	 &=& \kappa \int_0^1 \ln( 1 - e^{-(Q_{\lambda}(\xi) + \zeta t ) \theta}) d  t .\\
\end{eqnarray*}
Using that, for $\eta <1$,  $   - \ln( 1 - e^{-(Q_{\lambda}(\xi) + \zeta t) \theta}) ) \geq 0 $, we can use the monotone convergence theorem and deduce that $\tilde{K}^{(\xi)}( \lambda^{(\zeta)}) $ converges toward  $\tilde{K}^{(\xi)}( \lambda) $. 
And last for the function $H$, let us simply look at the term $L( \lambda^{(\zeta)}) = \int \int \ln | e^{- \theta x} - e^{ - \theta y}| d \lambda^{(\zeta)}(x) d \lambda^{(\zeta)}(y) $
\begin{eqnarray*}
	L( \lambda^{(\zeta)})  &=& \int_0^1 \int_0^1 \ln | e^{- \theta (Q_{\nu}(\xi) + \zeta t)} - e^{ - \theta(Q_{\nu}(u) + \zeta u) }| dt du \\
	&=&  2\int_0^1 \int_u^1 \ln ( e^{ - \theta(Q_{\nu}(\xi) +\zeta t)} - e^{ - \theta(Q_{\nu}(u) +\zeta u)} ) dt du \\
	&=& 2 \zeta \int_0^1 \int_u^1 \theta u dtdu +  2\int_0^1 \int_u^1 \ln ( e^{ - \theta(Q_{\nu}(\xi) +\eta (t-u))} - e^{ - \theta Q_{\nu}(u)} ) dt du\\
	&=& \frac{\zeta \theta}{3} +  2\int_0^1 \int_u^1 \ln ( e^{ - \theta(Q_{\nu}(\xi) +\zeta (t-u))} - e^{ - \theta Q_{\nu}(u)} ) dt du.\\
\end{eqnarray*}
We can then apply the monotone convergence theorem to prove that $L( \lambda^{(\zeta)})$ converges toward $L( \lambda)$. The convergence of $\tilde{H}^{(\xi)}( \lambda^{(\zeta)})$ toward $\tilde{H}^{(\xi)}( \lambda)$ follows.

From there, one can apply again the approximation step in Proposition \ref{prop:CompactSupport} to reduce ourselves to the case of a measure with bounded density and compact support. Now, proving the lower bound for a ball centered on a such a given measure follows exactly the same proof. The rest of the proof remains identical. In particular, reminding that $\tilde{\ell}_i^N$ is defined as 
\begin{equation}  \tilde{\ell}^N_{i }  = \sup  \{ j \in \N_t , \epsilon_N j  \leq y_{ L_\xi -i +1}. \} \end{equation}
the upper bound on the density ensures that for $N$ large enough so that $\epsilon_N \leq \zeta^{-1}$, we have $\tilde{\ell}_i^N \neq \tilde{\ell}_j^N$ for $i \neq j$ and that $\epsilon_N \tilde{\ell}_i \geq y_{L_\xi-i}$. Additionally, we still have 
\begin{eqnarray*}
	- \frac{N^2}{2}	Y^{(2)}_N( \tilde{\ell}^N ) + \sum_{ 1 \leq i < j \leq L_\xi} \frac{1}{\tilde{\ell}_j^N - \tilde{\ell}_i^N }  &=&   \sum_{\substack{1 \leq i <  j \leq L_\xi } } \ln ( \epsilon_N (\tilde{\ell}^N_j - \tilde{\ell}^N_i)) + \sum_{ 1 \leq i < j \leq L_\xi} \frac{1}{\tilde{\ell}_j^N - \tilde{\ell}_i^N }  \\
	& \geq & \sum_{\substack{1 \leq i <  j \leq L_\xi } } \ln ( \epsilon_N (\tilde{\ell}^N_j - \tilde{\ell}^N_i)) + \sum_{ 1 \leq i < j \leq L_\xi} \ln ( 1 + \frac{1}{\tilde{\ell}_j^N - \tilde{\ell}_i^N } ) \\
	& \geq &\sum_{\substack{1 \leq i <  j \leq L_\xi } } \ln ( \epsilon_N (\tilde{\ell}^N_j - \tilde{\ell}^N_i +1)) \\
	& \geq & \sum_{1 \leq i < j \leq L_\xi} \ln( y_{L_\xi -i +1} - y_{L_\xi -j +1}).
	\end{eqnarray*}
We then have 

\begin{eqnarray*} (L_\xi)^2 \int \int_{ x_1 > x_2} \ln(x_1 -x_2 ) d \lambda^{(\zeta)}(x_1)d \lambda^{(\zeta)}(x_2) &\leq& \sum_{i=1}^{ L_\xi +1}\sum_{j=i+1}^{ L_\xi +1} \ln( y_{j} - y_{i-1} )  +  \frac{1}{2}\sum_{i=1}^{ L_\xi +1}\ln( y_{i} - y_{i-1}) \\
& \leq & \sum_{1 \leq i < j \leq L_\xi} \ln( y_{i} - y_j) + O(N \ln N ).
 \end{eqnarray*}
Indeed we have $\texttt{M} + \beta \kappa \geq y_{i} - y_{j} \geq \frac{1}{\eta N}$. 
Regarding the $W_N^{(2)}$, since 

\[ \frac{1}{N} \sum_{i=1}^{N(1-\gamma^2-|\xi|)}\ln (1-e^{-\theta x-\beta |\xi|-\beta\theta\frac{j}{N}}) \geq ( 1-\gamma^2-|\xi| )\ln (1-e^{-\theta x}) \]
we can write

\begin{equation}
	W_N^{(2)}(\tilde{\ell}) \geq  \frac{\eta |\xi|}{N} \sum_{i=1}^{L_\xi}\epsilon_N \tilde{\ell}_i^N - \frac{ (1 - \gamma^2 - |\xi|)}{N} \sum_{i=1}^{L_\xi}\ln (1-e^{ - \epsilon_N \tilde{\ell}_i^N \theta} ).
\end{equation} 
From there it is easy to see that when $N$ goes to $\infty$, the right hand-side goes to $\tilde{K}^{(\xi)}(\lambda) + \tilde{M}^{(\xi)}( \lambda)$.
That concludes the proof.

\section{The equilibrium measure: subcritical and supercritical regimes}
\label{sec:rhp}
In the previous sections, we obtain a large deviation principle for the plane partition, i.e. we characterize the large deviation of each interlacing partition $l_i^{(\xi)}$. Specifically, we showed that the equilibrium measure (the asymptotic shape of the partition $l_i^{(\xi)}$) satisfies a large deviation principle with speed $N^2$ and rate function $I^{(\xi)}(\mu)$, see Theorem \ref{thm:mainLDP}. Furthermore, in Proposition \ref{prop:LDPl1}, we derive a large deviation principle for the length of the peak of the partition in terms of the same rate function $I^{(\xi)}(\mu)$. 
In this section, our goal is to prove Theorem \ref{thm:intro_main_shapes}, i.e. we want to obtain an explicit expression for the equilibrium measure of the functional $I^{(\xi)}(\mu)$ \eqref{eq:functional}.
We tackle these situations by rephrasing the minimization problem as a Riemann--Hilbert Problem (RHP) and solving it explicitly.
As we mentioned in the Introduction, we first reduce the non-standard rate function $I^{(\xi)}(\mu)$ \eqref{eq:functional} to the classical Muttalib--Borodin logarithmic energy functional. This is the content of the next proposition.

 \begin{proposition}
 \label{prop:bridge}
     In the same hypotheses as Theorem \ref{thm:mainLDP}, assume that  $\mu(dx) \equiv \mu(x)dx$, and define $\omega_\eta(x)=\frac{1}{\eta} x^{\frac{1}{\eta}-1}\mu(x^{\frac{1}{\eta}})\,,\omega_\theta(x)=\frac{1}{\theta} x^{\frac{1}{\theta}-1}\mu(x^{\frac{1}{\theta}})$. Then $\omega_\eta(x)\in\mathcal{P}^{\eta\beta\kappa}([0,e^{-\beta\eta(\gamma^2-\kappa)}])$ and $\omega_\theta(x)\in\mathcal{P}^{\theta\beta\kappa}([0,e^{-\beta\theta(\gamma^2-\kappa)}])$ are the unique minimizers of the functionals $I_\eta( \omega),I_\theta( \omega)$ respectively; here
     
     \begin{align}
	\label{eq:functional_eta}
	 &I_\eta( \omega) = -H_\eta(\omega) - K_\eta(\omega) - M_\eta(\omega),\\
     & I_\theta( \omega) = -H_\theta(\omega) - K_\theta(\omega) - M_\theta(\omega), \label{eq:functional_theta}
     \end{align}
     here $\kappa=\kappa(\xi)=L_\xi/N$, $H_\eta(\omega),\, K_\eta(\omega),\, M_\eta(\omega),\,H_\theta(\omega),\, K_\theta(\omega)$ and $M_\theta(\omega)$ have the following forms:
\begin{equation}
	H_\eta(\omega)=\frac{1}{2} \int \int\left(\ln (| x^{\frac{\theta}{\eta}} - y^{\frac{\theta}{\eta}}|) + \ln (| x - y| )\right) \omega(dx) \omega(dy)\,, \quad H_\theta(\omega)=\frac{1}{2} \int \int\left(\ln (| x^{\frac{\eta}{\theta}} - y^{\frac{\eta}{\theta}}|) + \ln (| x - y|) \right) \omega(dx) \omega(dy)\, ;
\end{equation}
     \begin{enumerate}[label=\roman*.]
         \item if $\xi\in (-\gamma^2,0]$, $\kappa=\gamma^2-|\xi|$  and
	\begin{equation}K_\eta(\omega)=\frac{1}{\kappa}\int\int_{\gamma^2-1}^{|\xi|}\ln(1-x^{\frac{\theta}{\eta}}e^{{ \beta}\theta u})du\, \omega(dx)\,, \quad K_\theta(\omega)=\frac{1}{\kappa}\int\int_{\gamma^2-1}^{|\xi|}\ln(1-xe^{{ \beta}\theta u})du\, \omega(dx)\,,
    \end{equation}
	\begin{equation}
	M_\eta(\omega)= \frac{|\xi|}{\kappa} \int  \ln(x)\, \omega(dx)\,, \quad M_\theta(\omega)= \frac{|\xi|\eta}{\kappa\theta} \int  \ln(x)\, \omega(dx)\,.
	\end{equation}
	
	\item if $\xi\in(0,1-\gamma^2]$, $\kappa=\gamma^2$ and
	\begin{equation}K_\eta(\omega)=\frac{1}{\kappa}\int\int_{0}^{1-\gamma^2-\xi}\ln(1-x^{\frac{\theta}{\eta}}e^{-{ \beta}\theta u})du\,\omega(dx)\,, \quad K_\theta(\omega)=\frac{1}{\kappa}\int\int_{0}^{1-\gamma^2-\xi}\ln(1-xe^{-{ \beta}\theta u})du\,\omega(dx) \,,\end{equation}
    
	\begin{equation}M_\eta(\omega)= \frac{\theta\xi }{\eta \kappa} \int  \ln(x)\, \omega(dx)\,,\quad M_\theta(\omega)= \frac{\xi}{ \kappa} \int  \ln(x)\, \omega(dx)\,.
    \end{equation}
	
	\item if $\xi\in(1-\gamma^2,1]$, $\kappa=1-\xi$ and
    
	\begin{equation}
    K_\eta(\omega)=\frac{1}{\kappa}\int\int_{{1-\gamma^2-\xi}}^{0}\ln(1-xe^{-\beta\eta u})du\, \omega(dx)\,,\quad K_\theta(\omega)=\frac{1}{\kappa}\int\int_{{1-\gamma^2-\xi}}^{0}\ln(1-x^{\frac{\eta}{\theta}}e^{-\beta\eta u})du\, \omega(dx)\,,
    \end{equation}
    
	\begin{equation}
    M_\eta(\omega)= \frac{\theta \xi }{\eta \kappa} \int  \ln(x)\,\omega(dx)\,, \quad M_\theta(\omega)= \frac{\xi }{\kappa} \int  \ln(x)\,\omega(dx)\,.\end{equation}
     \end{enumerate}
 \end{proposition}

 Therefore, if we can obtain an explicit expression for  $\omega_\eta(x),\omega_\theta(x)$, we would get one for $\mu(dx)$. So, as previously mentioned, we are naturally led to consider the following two model problems, which are analogous to the Model Problem \ref{mod_prob_1_intro}-\ref{mod_prob_2_intro}.
 
\begin{problem}
\label{mod_prob_1}
    Let $\nu >1$, consider the functional $\mathcal{I}_\nu[\omega_{\nu}]$ defined as

    \begin{equation}
    	\label{eq:simplified_functional_1}
       \mathcal{I}_\nu[\omega_{\nu}] =  - \frac{1}{2} \int_0^1 \int_0^1\left(\ln (| x^{\nu} - y^{\nu}|) + \ln( | x - y|)  \right)\omega_{\nu}(dx) \omega_{\nu}(dy)\,  - \frac{1}{\kappa}\int_0^1\int_{n_1}^{n_2} \ln(1 - x^\nu e^{-\beta \alpha u}) d u \omega_{\nu}(dx)- m_1 \int_0^1 \ln(x) \omega_{\nu}(dx) 
    \end{equation}
where $\alpha\,,\rho>0,m_2\geq 0, n_2\geq n_1$ and assume $n_1\alpha = -\rho(\gamma^2-\kappa)$, find $\omega_{\nu}(dx)\in \mathcal{P}^{\beta\rho\kappa}([0,e^{-\rho\beta(\gamma^2-\kappa)}])$, such that it minimize the previous functional. 
\end{problem}

\begin{problem}
\label{mod_prob_2}
    Let $\nu >1$, consider the functional $\mathcal{I}_1[\omega_1]$ defined as

    \begin{equation}
    	\label{eq:simplified_functional_2}
       \mathcal{I}_1[\omega_1] =  - \frac{1}{2} \int_0^1 \int_0^1\left(\ln (| x^{\nu} - y^{\nu}|) + \ln( | x - y|)  \right)\omega_1(dx) \omega_1(dy)\,  - \frac{1}{\kappa}\int_0^1\int_{n_1}^{n_2} \ln(1 - x e^{-\beta \alpha u}) d u \omega_1(dx)- m_1 \int_0^1 \ln(x) \omega_1(dx) 
    \end{equation}
where $\alpha\,,\rho>0,m_1\geq 0, n_2\geq n_1$ and assume $n_1\alpha = -\rho(\gamma^2-\kappa)$, find $\omega_1(dx)\in \mathcal{P}^{\beta\rho\kappa}([0,e^{-\rho\beta(\gamma^2-\kappa)}])$, such that it minimize the previous functional. 
\end{problem}

\begin{remark}
      The only difference between the previous two problems is that the potential is slightly different, indeed in the first case we have $\ln(1 - x^\nu e^{-\beta \alpha u})$ and in the second one $\ln(1 - x e^{-\beta \alpha u})$. Despite the difference is minimal, we need to define two different model problems for technical reasons.
\end{remark}

\begin{remark}
    We notice that one could also try to solve the model problem \ref{mod_prob_1} in the case $0<\nu<1$ to obtain the equilibrium measure $\omega_\nu(dx)$ for the all parameter values using Riemann surfaces - see for instance \cite{Charlier2022}. However, we decided to focus on two different model problems since the super critical regime becomes harder to track if $\nu<1$.
\end{remark}

    Given Theorem \ref{thm:mainLDP}, Proposition \ref{prop:bridge} and the model problems \ref{mod_prob_1}-\ref{mod_prob_2}, one can immediately deduce Corollary \ref{cor:intro_identification}.
    Therefore, to explicitly compute $\mu(x)dx$, we must solve the two model problems \ref{mod_prob_1}-\ref{mod_prob_2}.  
The advantage of these model problems is that they can be analyzed explicitly.
Following the same notation as in \cite{DiscreteOPbook}, we define three different types of intervals

\begin{definition}
For any sub-interval  $\frak{J}\subseteq(0,1)$ we say that it is a
    \begin{enumerate}
        \item[] \textbf{Void} if the lower constraint $f_1(x)\equiv0$ is active meaning that $\omega(dx)\equiv 0 $ for $x\in \frak{J}$
        \item[] \textbf{Saturated region} if the upper constraint $f_2(x)=(x \beta\delta)^{-1}$ is active, meaning that the equilibrium measure $\omega(dx)=f_2(x)dx$  for $x\in \frak{J}$
           \item[ ] \textbf{Band} if neither the upper constraint $f_2(x)$ or the lower constraint $f_1(x)$ are active for $x\in \frak{J}$. 
     \end{enumerate}
\end{definition}

The minimization problem we are considering is of the same kind of the one in \cite{Claeys2014} -- see also \cite{Wang2022}, where the author considered the same situation with $\nu\in\mathbb{N}$ -- thus we try to apply the same ideas, but with some important variations. Indeed, as we previously mentioned in the introduction, the minimization problems \ref{mod_prob_1}-\ref{mod_prob_2} differs for the one in the literature of Muttalib--Borodin ensembles for the presence of the upper constraint $(\beta \rho\kappa x)^{-1}$. Similar problems were considered in the classical logarithm energy case for some family of $q$-orthogonal polynomials \cite{byun2024qdeformedgaussianunitaryensemble,byun2025spectralanalysisqdeformedunitary,byun2026qdeformation}. So, in analogy to those works, we expect two different regimes. The first one, that we call \textit{subcritical} regime, corresponds to the case where this upper constraint is not active, meaning that the equilibrium measure does not have any saturated region. The second one, which we call \textit{supercritical} regime, corresponds to the case where there are some saturated regions.

We proceed as follows. First, we analyze the \textit{subcritical} regime, proving the first part of Theorem \ref{thm:intro_main_shapes}. Then, we focus on the \textit{supercritical} regime, proving the second part of the theorem.

    \begin{figure}
		\centering

\tikzset{every picture/.style={line width=0.75pt}} %set default line width to 0.75pt        

\begin{tikzpicture}[x=0.75pt,y=0.75pt,yscale=-1,xscale=1]

%Straight Lines [id:da5497181372805773] 
\draw    (131.33,160.92) -- (196.67,161.33) ;
\draw [shift={(196.67,161.33)}, rotate = 0.37] [color={rgb, 255:red, 0; green, 0; blue, 0 }  ][fill={rgb, 255:red, 0; green, 0; blue, 0 }  ][line width=0.75]      (0, 0) circle [x radius= 3.35, y radius= 3.35]   ;
\draw [shift={(131.33,160.92)}, rotate = 0.37] [color={rgb, 255:red, 0; green, 0; blue, 0 }  ][fill={rgb, 255:red, 0; green, 0; blue, 0 }  ][line width=0.75]      (0, 0) circle [x radius= 3.35, y radius= 3.35]   ;
%Curve Lines [id:da5497189488609093] 
\draw    (96,162.25) .. controls (97.33,37.33) and (216.67,82) .. (223,163.25) ;
%Curve Lines [id:da8671024879338345] 
\draw    (96,162.25) .. controls (98.67,258) and (213.33,301.33) .. (223,163.25) ;
\draw [shift={(223,163.25)}, rotate = 274] [color={rgb, 255:red, 0; green, 0; blue, 0 }  ][fill={rgb, 255:red, 0; green, 0; blue, 0 }  ][line width=0.75]      (0, 0) circle [x radius= 3.35, y radius= 3.35]   ;
\draw [shift={(96,162.25)}, rotate = 88.4] [color={rgb, 255:red, 0; green, 0; blue, 0 }  ][fill={rgb, 255:red, 0; green, 0; blue, 0 }  ][line width=0.75]      (0, 0) circle [x radius= 3.35, y radius= 3.35]   ;
%Curve Lines [id:da43173125375657295] 
\draw  [dash pattern={on 4.5pt off 4.5pt}]  (198,144.08) .. controls (237.4,114.53) and (471.17,122.42) .. (505.89,118.84) ;
\draw [shift={(507.33,118.67)}, rotate = 172.15] [color={rgb, 255:red, 0; green, 0; blue, 0 }  ][line width=0.75]    (10.93,-3.29) .. controls (6.95,-1.4) and (3.31,-0.3) .. (0,0) .. controls (3.31,0.3) and (6.95,1.4) .. (10.93,3.29)   ;
%Straight Lines [id:da06317003226318096] 
\draw    (573.17,108.83) -- (656.17,109.58) ;
\draw [shift={(656.17,109.58)}, rotate = 0.52] [color={rgb, 255:red, 0; green, 0; blue, 0 }  ][fill={rgb, 255:red, 0; green, 0; blue, 0 }  ][line width=0.75]      (0, 0) circle [x radius= 3.35, y radius= 3.35]   ;
\draw [shift={(573.17,108.83)}, rotate = 0.52] [color={rgb, 255:red, 0; green, 0; blue, 0 }  ][fill={rgb, 255:red, 0; green, 0; blue, 0 }  ][line width=0.75]      (0, 0) circle [x radius= 3.35, y radius= 3.35]   ;
%Straight Lines [id:da4060282825364421] 
\draw    (530.67,108) -- (568,20.67) ;
\draw [shift={(530.67,108)}, rotate = 293.15] [color={rgb, 255:red, 0; green, 0; blue, 0 }  ][fill={rgb, 255:red, 0; green, 0; blue, 0 }  ][line width=0.75]      (0, 0) circle [x radius= 3.35, y radius= 3.35]   ;
%Straight Lines [id:da11074583768547308] 
\draw    (530.67,108) -- (578,192) ;
%Curve Lines [id:da33887519566648483] 
\draw  [dash pattern={on 4.5pt off 4.5pt}]  (242.5,200.75) .. controls (287.05,208.18) and (471.26,228.34) .. (532.68,225.35) ;
\draw [shift={(534.5,225.25)}, rotate = 176.63] [color={rgb, 255:red, 0; green, 0; blue, 0 }  ][line width=0.75]    (10.93,-3.29) .. controls (6.95,-1.4) and (3.31,-0.3) .. (0,0) .. controls (3.31,0.3) and (6.95,1.4) .. (10.93,3.29)   ;
%Straight Lines [id:da7621799155736877] 
\draw    (568.5,232) -- (651.5,232.75) ;
\draw [shift={(651.5,232.75)}, rotate = 0.52] [color={rgb, 255:red, 0; green, 0; blue, 0 }  ][fill={rgb, 255:red, 0; green, 0; blue, 0 }  ][line width=0.75]      (0, 0) circle [x radius= 3.35, y radius= 3.35]   ;
\draw [shift={(568.5,232)}, rotate = 0.52] [color={rgb, 255:red, 0; green, 0; blue, 0 }  ][fill={rgb, 255:red, 0; green, 0; blue, 0 }  ][line width=0.75]      (0, 0) circle [x radius= 3.35, y radius= 3.35]   ;

% Text Node
\draw (202.67,166.73) node [anchor=north west][inner sep=0.75pt]    {$0$};
% Text Node
\draw (133.33,166.32) node [anchor=north east] [inner sep=0.75pt]    {$-1$};
% Text Node
\draw (225,166.65) node [anchor=north west][inner sep=0.75pt]    {$s_{b}$};
% Text Node
\draw (311,250.4) node [anchor=north west][inner sep=0.75pt]    {$J_{c_{0} ,c_{1}} \ :\ \mathbb{C} \setminus \overline{D} \ \rightarrow \mathbb{C} \setminus [ a,b]$};
% Text Node
\draw (296.33,61.9) node [anchor=north west][inner sep=0.75pt]    {$J_{c_{0} ,c_{1}} \ :\ D\setminus [ -1,0] \ \rightarrow H_{\nu } \setminus [ a,b]$};
% Text Node
\draw (119.67,115.23) node [anchor=north west][inner sep=0.75pt]    {$D$};
% Text Node
\draw (528.67,111.4) node [anchor=north east] [inner sep=0.75pt]    {$0$};
% Text Node
\draw (566.5,235.4) node [anchor=north east] [inner sep=0.75pt]    {$a$};
% Text Node
\draw (101.83,79.9) node [anchor=north west][inner sep=0.75pt]    {$\sigma _{+}$};
% Text Node
\draw (99.83,232.23) node [anchor=north west][inner sep=0.75pt]    {$\sigma _{-}$};
% Text Node
\draw (153.84,78.83) node [anchor=south] [inner sep=0.75pt]   [align=left] {(1)};
% Text Node
\draw (154.33,256.17) node [anchor=north west][inner sep=0.75pt]   [align=left] {(4)};
% Text Node
\draw (151.84,88) node [anchor=north] [inner sep=0.75pt]   [align=left] {(2)};
% Text Node
\draw (156.83,228.17) node [anchor=north west][inner sep=0.75pt]   [align=left] {(3)};
% Text Node
\draw (164,164.13) node [anchor=north] [inner sep=0.75pt]   [align=left] {(6)};
% Text Node
\draw (163.87,158.13) node [anchor=south] [inner sep=0.75pt]  [xslant=-0.04] [align=left] {(5)};
% Text Node
\draw (610,229.38) node [anchor=south] [inner sep=0.75pt]   [align=left] {(1)};
% Text Node
\draw (610,235.38) node [anchor=north] [inner sep=0.75pt]   [align=left] {(4)};
% Text Node
\draw (614.67,112.21) node [anchor=north] [inner sep=0.75pt]   [align=left] {(2)};
% Text Node
\draw (614.67,106.21) node [anchor=south] [inner sep=0.75pt]   [align=left] {(3)};
% Text Node
\draw (556.5,54) node [anchor=north west][inner sep=0.75pt]   [align=left] {(6)};
% Text Node
\draw (565.63,150) node [anchor=north west][inner sep=0.75pt]  [xslant=-0.04] [align=left] {(5)};
% Text Node
\draw (94,165.65) node [anchor=north east] [inner sep=0.75pt]    {$s_{a}$};
% Text Node
\draw (653.5,236.15) node [anchor=north west][inner sep=0.75pt]    {$b$};
% Text Node
\draw (562,111.07) node [anchor=north west][inner sep=0.75pt]    {$a$};
% Text Node
\draw (658.17,112.98) node [anchor=north west][inner sep=0.75pt]    {$b$};

\end{tikzpicture}
\caption{The transformation $J_{c_0,c_1}(s)$ mapping $D$ to $H_\nu\setminus[a,b]$ and $\C\setminus D$ to $\C\setminus[a,b]$. We highlight where the edges are mapped}
		\label{fig:RHP}
	\end{figure}

    \subsection{Sub-critical regime}
    \label{subsec:sub_crit}
    As we mentioned, in this regime we assume that the upper constraint is globally not active; therefore, we can use the same strategy as in \cite{Claeys2014}. The main difficulties is to identify the values $\beta_c$ such that the upper constraint becomes active. The two main results of this section are the following

    \begin{theorem}
        \label{thm:sub_crit_1}
      In the same notation as in Corollary \ref{cor:intro_identification}; consider the model problem \ref{mod_prob_1}, let $I^{\pm}$  be the inverse of $J_{c_0,c_1}(s)$ on $\sigma_\pm$ respectively. Defining

      \begin{equation}
        \label{eq:s_val_sub}
            s_1 = A\left(\frac{B-1}{A-B}\right)\,,\; s_2 = \frac{B-1}{A-B}\,,\;
            A  = \exp\left[ \frac{\alpha\beta\kappa}{\nu} \left( \nu + 1 + m_1 + \frac{\nu}{\kappa}(n_2-n_1) \right) \right],\;
            B = \exp\left[ \frac{\alpha\beta\kappa}{\nu} (1 + m_1) \right]
      \end{equation}

      and setting 

      \begin{equation}
        \label{eq:c_val_sub}
        c_1 = \frac{K_1 - K_2}{s_1 - s_2}\,,\; c_0 = \frac{K_2 s_1 - K_1 s_2}{s_1 - s_2}\,,\;K_1 = \left[ e^{n_1\alpha\beta} \left(\frac{A(B-1)}{B(A-1)} \right) \right]^  {\frac{1}{\nu}}\,,\;K_2 = \left[ e^{n_2\alpha\beta} \left( \frac{B-1}{A-1} \right) \right]^{\frac{1}{\nu}}\,,
    \end{equation}
    if 
    \begin{equation}
        \label{eq:b_sub}
        s_1 \leq s_b = -\frac{\nu-1}{2\nu} + \frac{1}{2\nu c_1}\sqrt{4c_0c_1\nu + c_1^2(\nu-1)^2} \,,
    \end{equation}
    then the equilibrium measure $\omega_\nu(\di x)\equiv  \omega_\nu(x)\di x$ has the following density

    \begin{equation}
        \omega_\nu(x) = \frac{1}{\pi \beta \rho\kappa x}\Arg\left(\frac{s_1-I_+(x)}{s_2-I_+(x)}\right)\mathds{1}_{x \in (a, b)}\,,
    \end{equation}
    where $a=J_{c_0,c_1}(s_a), b=J_{c_0,c_1}(s_b)$.

    \end{theorem}

    \begin{theorem}
        \label{thm:sub_crit_2}
      In the same notation as in Corollary \ref{cor:intro_identification}; consider the model problem \ref{mod_prob_2}, let $I^{\pm}$  be the inverse of $J_{c_0,c_1}(s)$ on $\sigma_\pm$ respectively. Defining

      \begin{equation}
        \label{eq:s_val_sub_2}
            s_1 = A\left(\frac{B-1}{A-B}\right)\,,\; s_2 = \frac{B-1}{A-B}\,,\;
            A  = \exp\left[ \alpha\beta\kappa \left( \nu + 1 + m_1 + \frac{1}{\kappa}(n_2-n_1) \right) \right],\;
            B = \exp\left[ \alpha\beta\kappa(1 + m_1) \right]
      \end{equation}

      and setting 

      \begin{equation}
        \label{eq:c_val_sub_2}
        c_1 = \frac{K_1 - K_2}{s_1 - s_2}\,,\; c_0 = \frac{K_2 s_1 - K_1 s_2}{s_1 - s_2}\,,\;K_1 =  e^{n_2 \alpha\beta} \left(  \frac{A(B-1)}{B(A-1)} \right)^  {\frac{1}{\nu}}\,,\;K_2 = e^{n_1\alpha\beta} \left( \frac{B-1}{A-1} \right) ^{\frac{1}{\nu}}\,,
    \end{equation}
    if 
    \begin{equation}
        s_1 \leq s_b = -\frac{\nu-1}{2\nu} + \frac{1}{2\nu c_1}\sqrt{4c_0c_1\nu + c_1^2(\nu-1)^2} \,,
    \end{equation}
    then the equilibrium measure $\omega_1(\di x)\equiv  \omega_1(x)\di x$ has the following density

    \begin{equation}
        \omega_1(x) = \frac{1}{\pi \beta \kappa\rho x}\Arg\left(\frac{s_1-I_+(x)}{s_2-I_+(x)}\right)\mathds{1}_{x \in (a, b)}\,,
    \end{equation}
    where $a=J_{c_0,c_1}(s_a), b=J_{c_0,c_1}(s_b)$.
    \end{theorem}

    Since the strategy of proofs are similar, we prove just Theorem \ref{thm:sub_crit_1}, and we point out the differences with the proof of Theorem \ref{thm:sub_crit_2}.
    
    \subsubsection{Proof of Theorem \ref{thm:sub_crit_1}}

    To simplify the notation, we drop the index $\nu$ of $\omega_\nu$ for this section. We now proceed by translating the model problem \ref{mod_prob_1} into a Riemann--Hilbert Problem (RHP). Proceeding as in the classical logarithmic potential case \cite{Saff2024,Bloom2016} the minimizer of the functional is characterized by the Euler-Lagrange equations:

	\begin{align}
		\int_0^1 \ln(\vert x^\nu -y^\nu\vert )\omega(\di y) +  \int_0^1\ln(\vert x -y\vert )\omega(\di y) + V(x) = \ell  & \quad x \in \frak{I}_0= (a,b)\\
  \label{eq:boring_bound}
		 \int_0^1 \ln(\vert x^\nu -y^\nu\vert )\omega(\di y) + \int_0^1\ln(\vert x -y\vert )\omega(\di y) + V(x) \leq \ell  & \quad x\not\in \frak{I}_0
	\end{align}
	for some $\ell \in \R$, we notice the change of sign in the derivative. Here $V(z)$ is defined as 
    \begin{equation}
        V(z) = \frac{1}{\kappa}\int_{n_1}^{n_2} \log(1-z^\nu e^{-\beta\alpha u})du + m_1\log(z)\,.
    \end{equation}

    \begin{remark}
        We notice that the potential $V(x)$ is concave in the interval $(0,1)$, thus, by following a standard argument as in \cite{Claeys2014}, -- see also \cite{Kuijlaars2000} for a more general treatment -- we can deduce that, if there are no saturated regions, the support of the equilibrium measure is a single band. 
    \end{remark}
	
	Define $g_\zeta(z) = \int_a^b\ln(z^\zeta - y^\zeta)\omega(\di y)$, then, following the standard notation for singular integrals \cite{Gakhovbook}, we deduce that for $x\in(a,b)$
	\begin{equation}
    \begin{split}
		&g_\zeta^+(x) =\lim_{\varepsilon\to0^+}g_{\zeta}(x+i\varepsilon)= \int_a^b\ln(\vert x^\zeta - y^\zeta\vert)\omega(\di y) +i\pi \int_x^b  \omega(\di y) \,,\\& g_\zeta^-(x) =\lim_{\varepsilon\to0^+}g_{\zeta}(x-i\varepsilon)  =\int_a^b\ln(\vert x^\zeta - y^\zeta\vert)\omega(\di y) - i\pi \int_x^b  \omega(\di y)\,.
    \end{split}
	\end{equation}
    We notice that the function $g_\zeta(z)$, for $\zeta>1$, is not well-defined in all $\C\setminus[a,b]$, but only in  $\mathbb{H}_\zeta\setminus[a,b]$, see Lemma \ref{lem:J_easy}.
	Using this notation and the previous equations, we deduce the following Euler-Lagrange equations for the equilibrium measure $\omega(dx)$ \cite{Saff2024}:
   {\mathtoolsset{showonlyrefs=false} 
    \begin{subequations} \label{eq:rhp_g_1}
        \begin{align}
			&g_\nu^+(x) +  g_1^-(x) + V(x)=\ell  &x \in \frak{I}_0\,,\\
			&g_1^+(x) - g_1^-(x) = g_\nu^+(x) - g_\nu^-(x)  = 2\pi i \int_x^b\omega(\di y)  &x \in \frak{I}_0\,.
	    \end{align}
    \end{subequations} }
	In particular, the functions $g_\nu(z),g_1(z)$ satisfy the following RHP
	\vspace{5pt}

	\begin{rhp}
		\label{rhp:g_small}
		\textbf{for $(g_\nu(z),g_1(z))$}
		
			\begin{itemize}
			\item[a.] $(g_\nu(z),g_1(z))$ are analytic in $(\mathbb{H}_\nu\setminus[a,b],\mathbb{C}\setminus[a,b])$
			\item[b.] $g_\nu(e^{-i\frac{\pi}{\nu}}x) = g_\nu(e^{i\frac{\pi}{\nu}}x) - 2\pi i$ for $x>0$ and $g_1^+(x) = g_1^-(x) + 2\pi i $ for $x<0$
			\item[c.] $g_\nu^+(x) +  g_1^-(x) = g_1^+(x) +  g_\nu^-(x)= -V(x) - \ell$ for $x\in(a,b)$
            \item[d.] $g_1(z) =  \ln(z) + O(z^{-1})$ as $z\to\infty$ in $\mathbb{C}\setminus[a,b]$
			\item[e.] $g_\nu(z) = \nu \ln(z) + O(z^{-\nu})$ as $z\to\infty$ in $\mathbb{H}_\nu\setminus[a,b]$
		\end{itemize}
	\end{rhp}
    In Figure \ref{fig:RHP} we sketch the contour for the functions $g_\nu(z),g_1(z)$ in RHP \ref{rhp:g_small}, we also highlight the transformation $J_{c_0,c_1}(s)$ that we use to solve this RHP.
	Consider the derivative of the previous function $G_\nu(z) = g_\nu'(z)$, $G_1(z)=g_1'(z)$, then from RHP \ref{rhp:g_small} and \eqref{eq:rhp_g_1} we deduce that $(G_\nu(z),G_1(z))$ solve the following RHP

\begin{rhp}
\label{rhp:g_big_1}
		\textbf{ for $(G_\nu(z),G_1(z))$}

	\begin{itemize}
	\item[a.] $(G_\nu(z),G_1(z))$ are analytic in $(\mathbb{H}_\nu\setminus[a,b],\mathbb{C}\setminus[a,b])$
	\item[b.] $G_\nu(e^{-i\frac{\pi}{\nu}}x) = e^{2\frac{\pi i}{\nu}}G_\nu(e^{i\frac{\pi}{\nu}}x)$ for $x\in\mathbb{R}_+$
	\item[c.] $G_\nu^+(x) +  G_1^-(x) = G_1^+(x) +  G_\nu^-(x)= -V'(x)$ for $x\in(a,b)$
        \item[d.] $G_1(z) = \frac{1}{z} + O(z^{-2})$ as $z\to\infty$ in $\mathbb{C}\setminus[a,b]$
	\item[e.] $G_\nu(z) = \frac{\nu}{z} + O(z^{-\nu-1})$ as $z\to\infty$ in $\mathbb{H}_\nu\setminus[a,b]$
	\item[f.] $G_1^+(x) -G_1^-(x)= G_\nu^+(x) - G_\nu^-(x)= -2 \pi i \omega(x)  $
\end{itemize}
\end{rhp}

	Consider now the following transformation
	
	\begin{equation}
		M(s) = \begin{cases}
			G_1(J_{c_0,c_1}(s)) \quad \textrm{outside } \sigma \\
			G_\nu(J_{c_0,c_1}(s)) \quad \textrm{inside } \sigma 
		\end{cases}\,.
	\end{equation}
	Therefore, $M(s)$ solves the following RHP
	
\begin{rhp}
	\textbf{for $M(s)$}

\begin{itemize}
	\item[a.] $M(s)$ are analytic in $\mathbb{C}\setminus\{\sigma \cup [-1,0]\}$
	\item[b.] $M^+(x) = e^{2\frac{\pi i}{\nu}}M^-(x)$ for $x\in (-1,0)$
	\item[c.] $M^+(s) + M^-(s) = -V'(J_{c_0,c_1}(s))$ for $s\in \sigma\setminus\{s_a,s_b\}$
	\item[d.] $\lim\limits_{s\to 0} M(s) = \frac{\nu}{J_{c_0,c_1}(s)}\left(  1 + o(1)\right)$
	\item[e.] $\lim\limits_{s\to \infty} M(s) = \frac{1}{J_{c_0,c_1}(s)}\left( 1 + o(1)\right)$
\end{itemize}
\end{rhp}
	Then, we can consider one last dressing transformation $N(s) = J_{c_0,c_1}(s)M(s)$; $N(s)$ solves the following RHP
	
\begin{rhp}
	\label{rhp:final_measure}
	\textbf{for $N(s)$}

\begin{itemize}
	\item[a.] $N(s)$ are analytic in $\mathbb{C}\setminus\sigma$
	\item[b.] $N^+(s) + N^-(s) = -J_{c_0,c_1}(s)V'(J_{c_0,c_1}(s)) = U(s)$ for $s\in \sigma\setminus\{s_a,s_b\}$
	\item[c.] $N(0) = \nu$, $N(-1)=0$
	\item[d.] $\lim\limits_{s\to \infty} N(s) = 1$
\end{itemize}
\end{rhp}
After some algebraic manipulations, the function $U(s)$ becomes

\begin{equation}
    U(s)= -m_1 + \frac{\nu}{\kappa \alpha\beta}\log\left(\frac{1- e^{-n_2\alpha\beta}J^\nu_{c_0,c_1}(s)}{1- e^{-n_1\alpha\beta}J^\nu_{c_0,c_1}(s)}\right)\,.
\end{equation} 
% Given this explicit expression for $U(s)$, one solve the RHP for $N(s)$ as 
% \begin{equation}
%     N(s) = \begin{cases}
%         1 - \frac{1}{2\pi i}\int_\sigma \frac{U(s)}{s-z}ds \qquad \text{outside } \sigma\\
%         -1 + \frac{1}{2\pi i}\int_\sigma \frac{U(s)}{s-z}ds \qquad \text{inside } \sigma\\
%     \end{cases}\,.
% \end{equation}
% We notice that, in general, the previous function does not satisfy the conditions $N(0)=\nu$ and $N(-1)=0$, therefore, we need to impose some conditions on $c_0,c_1$ to make sure that these two conditions are satisfied. 

To solve the previous RHP explicitly we need the following proposition.

\begin{proposition}
    \label{prop:ugly_integral}
        Following the same notation as before, for $z$ inside $\sigma$ the following holds

\begin{equation}
   \frac{1}{2\pi i} \int_\sigma \frac{\log\left(\frac{1 - e^{-n_2\alpha\beta J^\nu_{c_0,c_1}(s)}}{ 1- e^{-n_1\alpha\beta}J^\nu_{c_0,c_1}(s)}\right)}{s-z}ds = \log\left(\frac{1 - e^{-n_2\alpha\beta}J^\nu_{c_0,c_1}(z)}{1 - e^{-n_1\alpha\beta}J^\nu_{c_0,c_1}(z)}\right)  + \log\left(\frac{s_1 - z}{s_2-z}\right)
\end{equation}
where $s_1 = J^{-1}_{c_0,c_1}\left(e^{\frac{n_1\alpha\beta}{\nu}}\right), s_2 =  J^{-1}\left(e^{\frac{n_2\alpha\beta}{\nu}}\right)$, and both $s_1,s_2$ are inside $\sigma$, and $s_1\geq s_2$. 
    \end{proposition}
    \begin{proof} 
        We notice that the function $\displaystyle\log\left(\frac{1 - e^{-n_2\alpha\beta J^\nu_{c_0,c_1}(s)}}{ 1- e^{-n_1\alpha\beta}J^\nu_{c_0,c_1}(s)}\right)$ has a brunch-cut along the segment $(s_2,s_1)$ where $s_1 = J^{-1}_{c_0,c_1}\left(e^{\frac{n_1\alpha\beta}{\nu}}\right)$, $s_2 = J^{-1}_{c_0,c_1}\left(e^{\frac{n_2\alpha\beta}{\nu}}\right)$ are the unique preimage in the interior of the curve $\sigma$. Therefore, by residue calculation, one deduce the following:

        \begin{equation}
        \begin{split}
           \frac{1}{2\pi i}  \int_\sigma \frac{\log\left(\frac{1 - e^{-n_2\alpha\beta}J^\nu_{c_0,c_1}(s)}{1 - e^{-n_1\alpha\beta}J^\nu_{c_0,c_1}(s)}\right)}{s-z}ds & = \log\left(\frac{1 - e^{-n_2\alpha\beta}J^\nu_{c_0,c_1}(z)}{1 - e^{-n_1\alpha\beta}J^\nu_{c_0,c_1}(z)}\right) \\ &\quad + \frac{1}{\pi}\int_{s_1}^{s_2} \frac{\Arg_+(e^{n_2\alpha\beta} -  J^\nu_{c_0,c_1}(s)) - \Arg_+( e^{n_1\alpha\beta} - J^\nu_{c_0,c_1}(s))}{s-z} ds \\& = \log\left(\frac{1 - e^{-n_2\alpha\beta}J^\nu_{c_0,c_1}(z)}{1 - e^{-n_1\alpha\beta}J^\nu_{c_0,c_1}(z)}\right)  + \log\left(\frac{s_1 - z}{s_2-z}\right)\,.
        \end{split}
        \end{equation}
    \end{proof}
\begin{figure}
    \centering

\tikzset{every picture/.style={line width=0.75pt}} %set default line width to 0.75pt        

\begin{tikzpicture}[x=0.75pt,y=0.75pt,yscale=-0.75,xscale=0.75]

%Curve Lines [id:da8153785190451954] 
\draw [fill={rgb, 255:red, 126; green, 211; blue, 33 }  ,fill opacity=0.5 ]   (39.67,149.67) .. controls (39.67,-27) and (259.67,99) .. (259.67,151) ;
%Curve Lines [id:da31308690098447134] 
\draw [fill={rgb, 255:red, 208; green, 2; blue, 27 }  ,fill opacity=0.5 ]   (41,149.67) .. controls (41,326.33) and (261,203) .. (261,151) ;
\draw [shift={(261,151)}, rotate = 270] [color={rgb, 255:red, 0; green, 0; blue, 0 }  ][fill={rgb, 255:red, 0; green, 0; blue, 0 }  ][line width=0.75]      (0, 0) circle [x radius= 3.35, y radius= 3.35]   ;
\draw [shift={(41,149.67)}, rotate = 90] [color={rgb, 255:red, 0; green, 0; blue, 0 }  ][fill={rgb, 255:red, 0; green, 0; blue, 0 }  ][line width=0.75]      (0, 0) circle [x radius= 3.35, y radius= 3.35]   ;
%Shape: Polygon [id:ds739715022286942] 
\draw  [color={rgb, 255:red, 126; green, 211; blue, 33 }  ,draw opacity=0 ][fill={rgb, 255:red, 208; green, 2; blue, 27 }  ,fill opacity=0.5 ] (411.33,149.33) -- (470.33,21.73) -- (620.73,150.6) -- cycle ;
%Straight Lines [id:da16173262033080116] 
\draw    (70.87,150.33) -- (130.87,150.33) ;
\draw [shift={(130.87,150.33)}, rotate = 0] [color={rgb, 255:red, 0; green, 0; blue, 0 }  ][fill={rgb, 255:red, 0; green, 0; blue, 0 }  ][line width=0.75]      (0, 0) circle [x radius= 3.35, y radius= 3.35]   ;
\draw [shift={(70.87,150.33)}, rotate = 0] [color={rgb, 255:red, 0; green, 0; blue, 0 }  ][fill={rgb, 255:red, 0; green, 0; blue, 0 }  ][line width=0.75]      (0, 0) circle [x radius= 3.35, y radius= 3.35]   ;
%Straight Lines [id:da3396563135887133] 
\draw    (171.27,149.67) -- (231.27,149.67) ;
\draw [shift={(231.27,149.67)}, rotate = 0] [color={rgb, 255:red, 0; green, 0; blue, 0 }  ][fill={rgb, 255:red, 0; green, 0; blue, 0 }  ][line width=0.75]      (0, 0) circle [x radius= 3.35, y radius= 3.35]   ;
\draw [shift={(171.27,149.67)}, rotate = 0] [color={rgb, 255:red, 0; green, 0; blue, 0 }  ][fill={rgb, 255:red, 0; green, 0; blue, 0 }  ][line width=0.75]      (0, 0) circle [x radius= 3.35, y radius= 3.35]   ;
%Straight Lines [id:da7528501566723994] 
\draw    (411.33,149.33) -- (470.33,23) ;
%Curve Lines [id:da9272175923706907] 
\draw  [dash pattern={on 4.5pt off 4.5pt}]  (147,106.33) .. controls (212.67,16.33) and (257,30.33) .. (325.67,30.33) .. controls (393.99,30.33) and (444.49,19.11) .. (515.92,97.15) ;
\draw [shift={(517,98.33)}, rotate = 227.77] [color={rgb, 255:red, 0; green, 0; blue, 0 }  ][line width=0.75]    (10.93,-3.29) .. controls (6.95,-1.4) and (3.31,-0.3) .. (0,0) .. controls (3.31,0.3) and (6.95,1.4) .. (10.93,3.29)   ;
%Straight Lines [id:da18666217150159103] 
\draw [color={rgb, 255:red, 0; green, 0; blue, 0 }  ,draw opacity=0 ][fill={rgb, 255:red, 208; green, 2; blue, 27 }  ,fill opacity=0.5 ]   (41,149.67) -- (261,151) ;
%Shape: Polygon [id:ds13557922432998704] 
\draw  [color={rgb, 255:red, 126; green, 211; blue, 33 }  ,draw opacity=0 ][fill={rgb, 255:red, 126; green, 211; blue, 33 }  ,fill opacity=0.5 ] (411.33,149.33) -- (470.33,275.67) -- (620.73,150.6) -- cycle ;
%Straight Lines [id:da6733019991183092] 
\draw [fill={rgb, 255:red, 139; green, 56; blue, 56 }  ,fill opacity=0.84 ]   (411.33,149.33) -- (470.33,275.67) ;
\draw [shift={(411.33,149.33)}, rotate = 64.97] [color={rgb, 255:red, 0; green, 0; blue, 0 }  ][fill={rgb, 255:red, 0; green, 0; blue, 0 }  ][line width=0.75]      (0, 0) circle [x radius= 3.35, y radius= 3.35]   ;
%Straight Lines [id:da971767980654125] 
\draw    (456,149.38) -- (530.99,150.06) ;
\draw [shift={(530.99,150.06)}, rotate = 0.52] [color={rgb, 255:red, 0; green, 0; blue, 0 }  ][fill={rgb, 255:red, 0; green, 0; blue, 0 }  ][line width=0.75]      (0, 0) circle [x radius= 3.35, y radius= 3.35]   ;
\draw [shift={(456,149.38)}, rotate = 0.52] [color={rgb, 255:red, 0; green, 0; blue, 0 }  ][fill={rgb, 255:red, 0; green, 0; blue, 0 }  ][line width=0.75]      (0, 0) circle [x radius= 3.35, y radius= 3.35]   ;
%Straight Lines [id:da9492207990332147] 
\draw    (570.2,150.6) -- (620.73,150.6) ;
\draw [shift={(620.73,150.6)}, rotate = 0] [color={rgb, 255:red, 0; green, 0; blue, 0 }  ][fill={rgb, 255:red, 0; green, 0; blue, 0 }  ][line width=0.75]      (0, 0) circle [x radius= 3.35, y radius= 3.35]   ;
\draw [shift={(570.2,150.6)}, rotate = 0] [color={rgb, 255:red, 0; green, 0; blue, 0 }  ][fill={rgb, 255:red, 0; green, 0; blue, 0 }  ][line width=0.75]      (0, 0) circle [x radius= 3.35, y radius= 3.35]   ;

% Text Node
\draw (68.87,153.07) node [anchor=north east] [inner sep=0.75pt]    {$-1$};
% Text Node
\draw (128.87,153.07) node [anchor=north east] [inner sep=0.75pt]    {$0$};
% Text Node
\draw (169.27,153.07) node [anchor=north east] [inner sep=0.75pt]    {$s_{2}$};
% Text Node
\draw (229.27,153.07) node [anchor=north east] [inner sep=0.75pt]    {$s_{1}$};
% Text Node
\draw (37.67,153.07) node [anchor=north east] [inner sep=0.75pt]    {$s_{a}$};
% Text Node
\draw (261.67,154.4) node [anchor=north west][inner sep=0.75pt]    {$s_{b}$};
% Text Node
\draw (325.67,33.73) node [anchor=north] [inner sep=0.75pt]    {$J_{c_{0} ,c_{1}}$};
% Text Node
\draw (409.33,152.73) node [anchor=north east] [inner sep=0.75pt]    {$0$};
% Text Node
\draw (454,152.78) node [anchor=north east] [inner sep=0.75pt]    {$a$};
% Text Node
\draw (528.33,154.4) node [anchor=north east] [inner sep=0.75pt]    {$b$};
% Text Node
\draw (570.2,154) node [anchor=north] [inner sep=0.75pt]    {$J_{c_{0} ,c_{1}}( s_{1})$};
% Text Node
\draw (620.73,147.2) node [anchor=south] [inner sep=0.75pt]    {$J_{c_{0} ,c_{1}}( s_{2})$};

\end{tikzpicture}

    \caption{Contours of integration for Proposition \ref{prop:ugly_integral}}
    \label{fig:ungly_integral}
\end{figure}
The previous proposition allows us to compute the function $N(s)$ explicitly as 

    \begin{equation}
        N(z) =\begin{cases} -1 -m_1 + \frac{\nu}{\kappa\alpha\beta}\left(\log\left(\frac{1 - e^{-n_2\alpha\beta}J^\nu(z)}{1 - e^{-n_1\alpha\beta}J^\nu(z)}\right)  + \log\left(\frac{s_1 - z}{s_2-z}\right) \right) \qquad &z \text{ inside } \sigma\\
1 - \frac{\nu}{\kappa\alpha\beta} \log\left(\frac{s_1 - z}{s_2-z}\right) \qquad  &z \text{ outside } \sigma
\end{cases}\,.
    \end{equation}
We notice that, in general, the previous function \textit{does not} satisfy the conditions $N(0)=\nu$ and $N(-1)=0$, therefore, we need to impose some conditions on $c_0,c_1$ to make sure that these two conditions are satisfied. In particular, we need to impose the following system of equations:

\begin{equation}
    \begin{cases}
        \nu = N(0) = -1-m_1- \frac{\nu}{\kappa}(n_2-n_1) + \frac{\nu}{\kappa\alpha\beta}\log\left(\frac{s_1}{s_2}\right) \\
        0 = N(-1) = -1-m_1 + \frac{\nu}{\kappa\alpha\beta}\log\left(\frac{s_1+1}{s_2+1}\right) \\
        J^\nu(s_1) = e^{n_1\alpha\beta}\,,\qquad J^\nu(s_2) = e^{n_2\alpha\beta}
    \end{cases}\,.
\end{equation}
With tears, one can explicitly compute $s_1,s_2$ as follows:

      \begin{equation}
        \label{eq:A_B_sub}
            s_1 = A\left(\frac{B-1}{A-B}\right)\,,\; s_2 = \frac{B-1}{A-B}\,,\;
            A  = \exp\left[ \alpha\beta\kappa \left( \nu + 1 + m_1 + \frac{\nu}{\kappa}(n_2-n_1) \right) \right],\;
            B = \exp\left[ \frac{\alpha\beta\kappa}{\nu}(1 + m_1) \right]\,,
      \end{equation}
      and then $c_0,c_1$ as 

      \begin{equation}
        c_1 = \frac{K_1 - K_2}{s_1 - s_2}\,,\; c_0 = \frac{K_2 s_1 - K_1 s_2}{s_1 - s_2}\,,\;K_1 =  e^{n_1\alpha\beta} \left(  \frac{A(B-1)}{B(A-1)} \right)^  {\frac{1}{\nu}}\,,\;K_2 = e^{n_2\alpha\beta} \left( \frac{B-1}{A-1} \right) ^{\frac{1}{\nu}}\,,
    \end{equation}
we notice that the two previous equations coincides with \eqref{eq:s_val_sub}-\eqref{eq:c_val_sub} respectively.
We can now recover the equilibrium measure using the properties of $G_\nu(z)$ as follows:
\begin{equation}
\label{eq:measure_case1}
\omega(x) = \lim_{\varepsilon\to0^+}\frac{1}{2\pi i}\left( G_\nu(x-i\varepsilon) - G_\nu(x+i\varepsilon)\right) = \frac{1}{\pi\beta\rho\kappa x}\Arg\left(\frac{s_1 - I_+(x)}{s_2 - I_+(x)}\right)\mathds{1}_{x\in(a,b)}
\end{equation}
Furthermore, the equilibrium density $\omega(x)\leq \frac{1}{\rho\kappa \beta x}$, therefore the upper constraint is respected, and $\Arg\left(\frac{s_1 - I_+(x)}{s_2 - I_+(x)}\right)\geq 0$. The only caveat is that we are assuming that both $s_1,\,s_2$ are inside $\sigma$, therefore, we must restrict to the case where

\begin{equation}
    s_1\leq s_b\,,
\end{equation}
which is equivalent to \eqref{eq:b_sub}. This concludes the proof of the first theorem.

\begin{remark}
    We notice that we do not consider the case $n_2=n_1$ because in this case we expect the solution to blow-up at the hard-edge $e^{-\beta \rho (\gamma^2-\kappa)}$, therefore the upper constraint would be violated, and we need to consider a different ansatz for the solution, which is the content of the next section.
\end{remark}

\begin{remark}
    The proof of Theorem \ref{thm:sub_crit_2} is analogous, the only difference is that the function $U(s)$ becomes  
\begin{equation}
    U(s) = -m_1 + \frac{1}{\kappa\alpha\beta}\log\left(\frac{1- e^{-n_2\alpha\beta}J_{c_0,c_1}(s)}{1- e^{-n_1\alpha\beta}J_{c_0,c_1}(s)}\right)\,.
\end{equation}
Therefore, one has to obtain an analogous version of Proposition \ref{prop:ugly_integral} where we know consider $z$ \textbf{outside} of $\sigma$.
\end{remark}

\subsection{The supercritical regime}
\label{subsec:above_critical_regime}

We now consider the supercritical regime, which is when the construction in the previous section breaks down, which happens when $s_1>s_b\,$.
As before, we split the proof for the two model problems, but in this case we also have to consider a special subcase: $n_1=n_2$.
The main result of this section are the following: 
\begin{theorem}
        \label{thm:above_crit}
      In the same notation as in Corollary \ref{cor:intro_identification} and assume that $n_1\neq n_2$; consider the model problem \ref{mod_prob_1}, let $I^{\pm}$  be the inverse of $J_{c_0,c_1}(s)$ on $\sigma_\pm$ respectively. Defining

      \begin{equation}
            s_1 = A\left(\frac{B-1}{A-B}\right)\,,\; s_2 = \frac{B-1}{A-B}\,,\;
            A  = \exp\left[ \frac{\alpha\beta\kappa}{\nu} \left( \nu + 1 + m_1 + \frac{\nu}{\kappa}(n_2-n_1) \right) \right],\;
            B = \exp\left[ \frac{\alpha\beta\kappa}{\nu} (1 + m_1) \right]\,,
      \end{equation}
      and setting 

      \begin{equation}
        c_1 = \frac{K_1 - K_2}{s_1 - s_2}\,,\; c_0 = \frac{K_2 s_1 - K_1 s_2}{s_1 - s_2}\,,\;K_1 = \left[ e^{n_1\alpha\beta} \left(\frac{A(B-1)}{B(A-1)} \right) \right]^  {\frac{1}{\nu}}\,,\;K_2 = \left[ e^{n_2\alpha\beta} \left( \frac{B-1}{A-1} \right) \right]^{\frac{1}{\nu}}\,,
    \end{equation}
    if 
    \begin{equation}
        \label{eq:b_critical}
        s_1 \geq s_b = -\frac{\nu-1}{2\nu} + \frac{1}{2\nu c_1}\sqrt{4c_0c_1\nu + c_1^2(\nu-1)^2} \,,
    \end{equation}
    then the equilibrium measure $\omega_\nu(\di x)\equiv  \omega_\nu(x)\di x$ has the following density

    \begin{equation}
        \omega_\nu(x) = \begin{cases}\frac{1}{\pi \beta \rho\kappa x}\Arg\left(\frac{s_1-I_+(x)}{s_2-I_+(x)}\right) \qquad & x\in (a,b) \\
        \frac{1}{\beta \rho\kappa x} \qquad & x\in (b,x_{\max}) \\
        \end{cases}\,,
    \end{equation}
where $a=J_{c_0,c_1}(s_a), b=J_{c_0,c_1}(s_b)$. 
    \end{theorem}

    An analogous result holds for the model problem \ref{mod_prob_2}

    \begin{theorem}
        \label{thm:above_crit_2}
      In the same notation as in Corollary \ref{cor:intro_identification} and assume that $n_1\neq n_2$; consider the model problem \ref{mod_prob_2}, let $I^{\pm}$  be the inverse of $J_{c_0,c_1}(s)$ on $\sigma_\pm$ respectively. Defining

      \begin{equation}
        \label{eq:s_val_above_2}
            s_1 = A\left(\frac{B-1}{A-B}\right)\,\; s_2 = \frac{B-1}{A-B}\,,\;
            A  = \exp\left[ \alpha\beta\kappa \left( \nu + 1 + m_1 + \frac{1}{\kappa}(n_2-n_1) \right) \right],\;
            B = \exp\left[ \alpha\beta\kappa(1 + m_1) \right]
      \end{equation}
      and setting 

      \begin{equation}
        \label{eq:c_val_above_2}
        c_1 = \frac{K_1 - K_2}{s_1 - s_2}\,,\; c_0 = \frac{K_2 s_1 - K_1 s_2}{s_1 - s_2}\,,\;K_1 =  e^{n_2\alpha\beta} \left(  \frac{A(B-1)}{B(A-1)} \right)^  {\frac{1}{\nu}}\,,\;K_2 = e^{n_1\alpha\beta} \left( \frac{B-1}{A-1} \right) ^{\frac{1}{\nu}}\,,
    \end{equation}
    if 
    \begin{equation}
        s_1 \geq s_b = -\frac{\nu-1}{2\nu} + \frac{1}{2\nu c_1}\sqrt{4c_0c_1\nu + c_1^2(\nu-1)^2} \,,
    \end{equation}
    then the equilibrium measure $\omega_1(\di x)\equiv  \omega_1(x)\di x$ has the following density

    \begin{equation}
        \omega_1(x) = \begin{cases} \frac{1}{\pi \beta \kappa\rho x}\Arg\left(\frac{s_1-I_+(x)}{s_2-I_+(x)}\right) &\qquad x\in(a,b) \\
        \frac{1}{\beta \kappa \rho x} &\qquad x \in (b,x_{\max})\end{cases}\,,
    \end{equation}
    where $a=J_{c_0,c_1}(s_a), b=J_{c_0,c_1}(s_b)$.
    \end{theorem}

Furthermore, in the specific situation where $n_1=n_2$, we prove the following 

\begin{theorem}
    \label{thm:above_crit_case2}
    In the same notation as in Corollary \ref{cor:intro_identification} and assuming that $n_1 = n_2$; consider the model problem \ref{mod_prob_1} or \ref{mod_prob_2}, let $I^{\pm}$  be the inverse of $J_{c_0,c_1}(s)$ on $\sigma_\pm$ respectively. Defining

      \begin{equation}
        \label{eq:s_val_above_2_1}
            s_1 = A\left(\frac{B-1}{A-B}\right)\,,\; s_2 = \frac{B-1}{A-B}\,,\;
            A  = \exp\left[ \alpha\beta\kappa \left( \nu + 1 + m_1  \right) \right],\;
            B = \exp\left[ \alpha\beta\kappa(1 + m_1) \right]
      \end{equation}
      and setting 

      \begin{equation}
        \label{eq:c_val_above_2_1}
        c_1 = \frac{K_1 - K_2}{s_1 - s_2}\,,\; c_0 = \frac{K_2 s_1 - K_1 s_2}{s_1 - s_2}\,,\;K_1 =  e^{n_1\alpha\beta} \left(  \frac{A(B-1)}{B(A-1)} \right)^  {\frac{1}{\nu}}\,,\;K_2 = e^{n_1\alpha\beta} \left( \frac{B-1}{A-1} \right) ^{\frac{1}{\nu}}\,,
    \end{equation}
    then the equilibrium measure $\omega_1(\di x)\equiv  \omega_1(x)\di x = \omega_\nu(\di x)\equiv  \omega_\nu(x)\di x$ has the following density
     \begin{equation}
        \omega_\nu(x) =\omega_1(x)= \begin{cases}
            \frac{1}{\pi \beta\rho\kappa x}\left(\Arg\left(\frac{s_1 - I^+(x)}{ s_2 - I^+(x)}\right) \right)\quad & x\in (a,b) \\
            \frac{1}{\beta\rho\kappa x} \quad & x\in (b,1)
        \end{cases}\,.
    \end{equation}
    where $a=J_{c_0,c_1}(s_a), b=J_{c_0,c_1}(s_b)$.
\end{theorem}

\begin{remark}
    We notice that, we state the result with the explicit extreme $1$ for the support of the measure since the only cases where $n_1=n_2$ are the cases where $\xi=1-\gamma^2$, which implies that $\gamma^2=\kappa$, and therefore the support of the measure is $(0,1)$.
\end{remark}

The strategy of proof is slightly different from the previous case. For Theorem \ref{thm:above_crit}, we consider a different ansatz for the equilibrium measure. Specifically, our guess is that the equilibrium measure has a saturated region $\mathfrak{I}_+ =\left(b,e^{-\beta \rho (\gamma^2-\kappa)}\right)$ where the upper constraint is active, and a gap region $\mathfrak{I}_-=(0,a)$ where the measure is zero, and a band $\frak{I}_0=(a,b)$ where the measure is strictly positive and below the upper constraint. 

As in the previous section, since the proof of the first two results are analogous, we will just give the proof of Theorem \ref{thm:above_crit}, and we will just point out the differences in the proof of Theorem \ref{thm:above_crit_2}. Regarding Theorem \ref{thm:above_crit_case2}, we present a full proof.

\subsubsection{Proof of Theorem \ref{thm:above_crit}}

To keep the notation simple, we drop the subscript $\nu$ in the following proof, and we just write $\omega(dx)$ instead of $\omega_\nu(dx)$.
Following the mentioned heuristic, the solution of the minimization problem satisfies the following Euler-Lagrange equations

{
    \mathtoolsset{showonlyrefs=false}
    \begin{subequations}
        \label{eq:EL_constraint_2}
        \begin{align}
	    &\int_0^{\xm} \ln(|x-y|)\omega(y)dy+\int_0^{\xm} \ln(|x^\nu-y^\nu|)\omega(y)dy  + V(x) = \ell \quad x\in \mathfrak{I}_0 \,, \\
		&\int_0^{\xm} \ln(|x-y|)\omega(y)dy+\int_0^{\xm} \ln(|x^\nu-y^\nu|)\omega(y)dy + V(x) \geq \ell \quad x\in \mathfrak{I}_+\,,\\
        &\int_0^{\xm}\ln(|x-y|)\omega(y)dy+\int_0^{\xm} \ln(|x^\nu-y^\nu|)\omega(y)dy + V(x) \leq \ell \quad x\in \mathfrak{I}_-\,.
    \end{align}
    \end{subequations}

}
Where $\xm=e^{-\rho\beta(\gamma^2-\kappa)}$.
Given this ansatz, following the same notation as in the previous section, we can recast the previous E-L equations in the following RHP for the functions $(g_1(z),g_\nu(z))$:

\begin{rhp}
		\label{rhp:g_small_final}
		\textbf{for $(g_1(z),g_\nu(z))$}
		
			\begin{enumerate}[label=\alph*.]
			\item $(g_1(z),g_\nu(z))$ are analytic in $(\C\setminus[a,\xm],\mathbb{H}_\nu\setminus[a,\xm])$
			\item $g_\nu(e^{-i\frac{\pi}{\nu}}x) = g_\nu(e^{i\frac{\pi}{\nu}}x) - 2\pi i$ for $x>0$ and $g_1^+(x) = g_1^-(x) + 2\pi i $ for $x<0$
			\item$g_\nu^+(x) +  g_1^-(x) = g_1^+(x) +  g_\nu^-(x)=\ell -V(x)$ for $x\in(a,b)$
            \item $g_1(z) =  \ln(z) + O(z^{-1})$ as $z\to\infty$ in $\mathbb{C}\setminus[a,\xm]$
			\item $g_\nu(z) = \nu \ln(z) + O(z^{-\nu})$ as $z\to\infty$ in $\mathbb{H}_\nu\setminus[a,\xm]$
   \item $g_1^+(x)-g_1^-(x)=g_\nu^+(x)-g_\nu^-(x)=2\pi i\int_x^{\xm}\omega(y)dy$ for $x\in(a,b)$
   \item $g_1^+(x)-g_1^-(x)=g_\nu^+(x)-g_\nu^-(x)=\frac{2\pi i}{\beta\delta} \int_x^{\xm} \frac{\di y}{y}$ for $x\in(b,\xm)$
		\end{enumerate}
	\end{rhp}

 Proceeding as before in Subsection \ref{subsec:sub_crit}, we consider the same chain of transformation, leading to the following RHP for $N(s)$ 
\begin{rhp}
	\label{rhp:final_measure_2}
	\textbf{for $N(s)$}

\begin{enumerate}[label=\alph*.]
	\item $N(s)$ are analytic in $\mathbb{C}\setminus\{\sigma\cup(q_1,s_1)\}$
	\item $N^+(s) + N^-(s) = -J_{c_0,c_1}(s)V'(J_{c_0,c_1}(s)) = U(s)$ for $s\in \sigma\setminus\{s_{a},s_{b}\}$
	\item $N(0) = \nu$, $N(-1)=0$
	\item $\lim\limits_{s\to \infty} N(s) = 1$
  \item $N^+(s) -N^-(s) = -\frac{2\pi i}{\beta\delta}$ for $s\in(s_{b}, s_1)$
   \item $N^+(s) -N^-(s) = \frac{2\pi i}{\beta\delta}$ for $s\in(q_1,s_{b})$
\end{enumerate}
\end{rhp}
where
\begin{equation}
    \label{eq:q_val_above}
    J_{c_0,c_1}(q_1) = J_{c_0,c_1}(s_1)=\xm\,,
\end{equation}
in particular, we chose $q_1$ to be the root inside $\sigma$, and $s_1$ the root outside $\sigma$, so that $s_a<q_1\leq s_b\leq s_1$. In figure \ref{fig:contour_nuts} one can see an example of the contour for the previous RHP.

\begin{figure}[ht]
    \centering

\tikzset{every picture/.style={line width=0.75pt}} %set default line width to 0.75pt        

\begin{tikzpicture}[x=0.75pt,y=0.75pt,yscale=-0.7,xscale=0.7]
%uncomment if require: \path (0,488); %set diagram left start at 0, and has height of 488

%Curve Lines [id:da2736899466659515] 
\draw    (29.01,241.88) .. controls (29.22,196.38) and (37.7,163.72) .. (63.96,144.12) .. controls (82.24,130.47) and (109.14,123.15) .. (147.87,122.24) .. controls (265.63,124.46) and (293.07,182.15) .. (293.07,241.23) ;
\draw [shift={(33.6,194.52)}, rotate = 284.18] [fill={rgb, 255:red, 0; green, 0; blue, 0 }  ][line width=0.08]  [draw opacity=0] (8.93,-4.29) -- (0,0) -- (8.93,4.29) -- cycle    ;
\draw [shift={(98.25,128.4)}, rotate = 346.36] [fill={rgb, 255:red, 0; green, 0; blue, 0 }  ][line width=0.08]  [draw opacity=0] (8.93,-4.29) -- (0,0) -- (8.93,4.29) -- cycle    ;
\draw [shift={(244.67,144.75)}, rotate = 32.68] [fill={rgb, 255:red, 0; green, 0; blue, 0 }  ][line width=0.08]  [draw opacity=0] (8.93,-4.29) -- (0,0) -- (8.93,4.29) -- cycle    ;
\draw [shift={(29.01,241.88)}, rotate = 270.27] [color={rgb, 255:red, 0; green, 0; blue, 0 }  ][fill={rgb, 255:red, 0; green, 0; blue, 0 }  ][line width=0.75]      (0, 0) circle [x radius= 3.35, y radius= 3.35]   ;
%Curve Lines [id:da6062136104174394] 
\draw    (29.01,241.88) .. controls (29.6,288.46) and (38.85,321.43) .. (67.19,340.46) .. controls (85.5,352.75) and (111.77,359.22) .. (148.8,359.79) .. controls (266.55,356.64) and (293.54,299.92) .. (293.08,241.23) ;
\draw [shift={(37.69,301.5)}, rotate = 251.87] [fill={rgb, 255:red, 0; green, 0; blue, 0 }  ][line width=0.08]  [draw opacity=0] (8.93,-4.29) -- (0,0) -- (8.93,4.29) -- cycle    ;
\draw [shift={(111.53,356.66)}, rotate = 190.58] [fill={rgb, 255:red, 0; green, 0; blue, 0 }  ][line width=0.08]  [draw opacity=0] (8.93,-4.29) -- (0,0) -- (8.93,4.29) -- cycle    ;
\draw [shift={(254.61,330.71)}, rotate = 144.57] [fill={rgb, 255:red, 0; green, 0; blue, 0 }  ][line width=0.08]  [draw opacity=0] (8.93,-4.29) -- (0,0) -- (8.93,4.29) -- cycle    ;
%Straight Lines [id:da8527563711173991] 
\draw    (70.14,241.46) -- (120.43,241.46) ;
\draw [shift={(120.43,241.46)}, rotate = 0] [color={rgb, 255:red, 0; green, 0; blue, 0 }  ][fill={rgb, 255:red, 0; green, 0; blue, 0 }  ][line width=0.75]      (0, 0) circle [x radius= 3.35, y radius= 3.35]   ;
\draw [shift={(70.14,241.46)}, rotate = 0] [color={rgb, 255:red, 0; green, 0; blue, 0 }  ][fill={rgb, 255:red, 0; green, 0; blue, 0 }  ][line width=0.75]      (0, 0) circle [x radius= 3.35, y radius= 3.35]   ;
%Straight Lines [id:da3695584882240348] 
\draw    (192.51,241.25) -- (242.79,241.25) ;
\draw [shift={(242.79,241.25)}, rotate = 0] [color={rgb, 255:red, 0; green, 0; blue, 0 }  ][fill={rgb, 255:red, 0; green, 0; blue, 0 }  ][line width=0.75]      (0, 0) circle [x radius= 3.35, y radius= 3.35]   ;
\draw [shift={(192.51,241.25)}, rotate = 0] [color={rgb, 255:red, 0; green, 0; blue, 0 }  ][fill={rgb, 255:red, 0; green, 0; blue, 0 }  ][line width=0.75]      (0, 0) circle [x radius= 3.35, y radius= 3.35]   ;
%Straight Lines [id:da24738064245647318] 
\draw    (242.79,241.25) -- (268.5,241.25) -- (293.08,241.25) ;
\draw [shift={(293.08,241.25)}, rotate = 0] [color={rgb, 255:red, 0; green, 0; blue, 0 }  ][fill={rgb, 255:red, 0; green, 0; blue, 0 }  ][line width=0.75]      (0, 0) circle [x radius= 3.35, y radius= 3.35]   ;
\draw [shift={(242.79,241.25)}, rotate = 0] [color={rgb, 255:red, 0; green, 0; blue, 0 }  ][fill={rgb, 255:red, 0; green, 0; blue, 0 }  ][line width=0.75]      (0, 0) circle [x radius= 3.35, y radius= 3.35]   ;
%Straight Lines [id:da38551683660960356] 
\draw    (293.08,241.25) -- (340,241.08) -- (384.21,240.93) ;
\draw [shift={(384.21,240.93)}, rotate = 359.8] [color={rgb, 255:red, 0; green, 0; blue, 0 }  ][fill={rgb, 255:red, 0; green, 0; blue, 0 }  ][line width=0.75]      (0, 0) circle [x radius= 3.35, y radius= 3.35]   ;
%Straight Lines [id:da11309478478395096] 
\draw [color={rgb, 255:red, 0; green, 0; blue, 0 }  ,draw opacity=1 ]   (469.65,100.85) -- (517.63,100.85) -- (565.07,100.85) ;
\draw [shift={(565.07,100.85)}, rotate = 0] [color={rgb, 255:red, 0; green, 0; blue, 0 }  ,draw opacity=1 ][fill={rgb, 255:red, 0; green, 0; blue, 0 }  ,fill opacity=1 ][line width=0.75]      (0, 0) circle [x radius= 3.35, y radius= 3.35]   ;
\draw [shift={(469.65,100.85)}, rotate = 0] [color={rgb, 255:red, 0; green, 0; blue, 0 }  ,draw opacity=1 ][fill={rgb, 255:red, 0; green, 0; blue, 0 }  ,fill opacity=1 ][line width=0.75]      (0, 0) circle [x radius= 3.35, y radius= 3.35]   ;
%Straight Lines [id:da04113940009346129] 
\draw [color={rgb, 255:red, 0; green, 0; blue, 0 }  ,draw opacity=1 ]   (565.07,100.85) -- (592.5,100.5) -- (620.5,100.5) ;
\draw [shift={(620.5,100.5)}, rotate = 0] [color={rgb, 255:red, 0; green, 0; blue, 0 }  ,draw opacity=1 ][fill={rgb, 255:red, 0; green, 0; blue, 0 }  ,fill opacity=1 ][line width=0.75]      (0, 0) circle [x radius= 3.35, y radius= 3.35]   ;
%Straight Lines [id:da5756195067232176] 
\draw [color={rgb, 255:red, 0; green, 0; blue, 0 }  ,draw opacity=1 ]   (443.67,100.25) -- (464.67,30.25) ;
%Straight Lines [id:da04089301901047093] 
\draw [color={rgb, 255:red, 0; green, 0; blue, 0 }  ,draw opacity=1 ]   (464.67,169.25) -- (443.67,100.25) ;
%Straight Lines [id:da5910711604292318] 
\draw [color={rgb, 255:red, 0; green, 0; blue, 0 }  ,draw opacity=1 ]   (468.48,419) -- (524.5,419) -- (563.89,419) ;
\draw [shift={(563.89,419)}, rotate = 0] [color={rgb, 255:red, 0; green, 0; blue, 0 }  ,draw opacity=1 ][fill={rgb, 255:red, 0; green, 0; blue, 0 }  ,fill opacity=1 ][line width=0.75]      (0, 0) circle [x radius= 3.35, y radius= 3.35]   ;
\draw [shift={(468.48,419)}, rotate = 0] [color={rgb, 255:red, 0; green, 0; blue, 0 }  ,draw opacity=1 ][fill={rgb, 255:red, 0; green, 0; blue, 0 }  ,fill opacity=1 ][line width=0.75]      (0, 0) circle [x radius= 3.35, y radius= 3.35]   ;
%Straight Lines [id:da9726894321934262] 
\draw [color={rgb, 255:red, 0; green, 0; blue, 0 }  ,draw opacity=1 ]   (563.89,419) -- (590,419.35) -- (618.8,419.74) ;
\draw [shift={(618.8,419.74)}, rotate = 0.78] [color={rgb, 255:red, 0; green, 0; blue, 0 }  ,draw opacity=1 ][fill={rgb, 255:red, 0; green, 0; blue, 0 }  ,fill opacity=1 ][line width=0.75]      (0, 0) circle [x radius= 3.35, y radius= 3.35]   ;
%Curve Lines [id:da813168548290159] 
\draw  [dash pattern={on 3.75pt off 3pt on 7.5pt off 1.5pt}]  (122.5,159.25) .. controls (107.21,113.64) and (197.81,101.08) .. (287.03,99.18) .. controls (357.46,97.69) and (413.25,100.37) .. (427.8,100.93) ;
\draw [shift={(429.75,101)}, rotate = 181.83] [color={rgb, 255:red, 0; green, 0; blue, 0 }  ][line width=0.75]    (10.93,-3.29) .. controls (6.95,-1.4) and (3.31,-0.3) .. (0,0) .. controls (3.31,0.3) and (6.95,1.4) .. (10.93,3.29)   ;
%Curve Lines [id:da8047639430231129] 
\draw  [dash pattern={on 3.75pt off 3pt on 7.5pt off 1.5pt}]  (184.5,397) .. controls (193.68,400.67) and (277.18,411.89) .. (324.57,414.79) .. controls (370.54,417.59) and (427.19,414.17) .. (447.7,413.09) ;
\draw [shift={(449.5,413)}, rotate = 177.1] [color={rgb, 255:red, 0; green, 0; blue, 0 }  ][line width=0.75]    (10.93,-3.29) .. controls (6.95,-1.4) and (3.31,-0.3) .. (0,0) .. controls (3.31,0.3) and (6.95,1.4) .. (10.93,3.29)   ;

% Text Node
\draw (287.03,96.78) node [anchor=south] [inner sep=0.75pt]  [color={rgb, 255:red, 0; green, 0; blue, 0 }  ,opacity=1 ]  {$J_{c_{0} ,c_{1}} \ :\ D\setminus [ -1,0] \ \rightarrow \mathbb{H}_{\nu } \setminus [ a ,b]$};
% Text Node
\draw (324.57,418.19) node [anchor=north] [inner sep=0.75pt]  [color={rgb, 255:red, 0; green, 0; blue, 0 }  ,opacity=1 ]  {$J_{c_{0} ,c_{1}} \ :\ \mathbb{C} \setminus \overline{D} \ \rightarrow \mathbb{C} \setminus [ a,b]$};
% Text Node
\draw (27.01,245.28) node [anchor=north east] [inner sep=0.75pt]    {$s_{a}$};
% Text Node
\draw (295.07,244.63) node [anchor=north west][inner sep=0.75pt]    {$s_{b}$};
% Text Node
\draw (190.51,244.65) node [anchor=north east] [inner sep=0.75pt]    {$s_{2}$};
% Text Node
\draw (386.21,244.33) node [anchor=north west][inner sep=0.75pt]    {$s_{1}$};
% Text Node
\draw (72.14,245.26) node [anchor=north west][inner sep=0.75pt]    {$-1$};
% Text Node
\draw (122.43,245.26) node [anchor=north west][inner sep=0.75pt]    {$0$};
% Text Node
\draw (471.65,104.25) node [anchor=north west][inner sep=0.75pt]    {$a$};
% Text Node
\draw (567.07,104.25) node [anchor=north west][inner sep=0.75pt]    {$b$};
% Text Node
\draw (470.48,422.4) node [anchor=north west][inner sep=0.75pt]    {$a$};
% Text Node
\draw (565.89,422.4) node [anchor=north west][inner sep=0.75pt]    {$b$};
% Text Node
\draw (61.96,140.72) node [anchor=south east] [inner sep=0.75pt]    {$( 1)$};
% Text Node
\draw (65.96,147.52) node [anchor=north west][inner sep=0.75pt]    {$( 2)$};
% Text Node
\draw (69.19,337.06) node [anchor=south west] [inner sep=0.75pt]    {$( 3)$};
% Text Node
\draw (65.19,343.86) node [anchor=north east] [inner sep=0.75pt]    {$( 4)$};
% Text Node
\draw (268.5,237.85) node [anchor=south] [inner sep=0.75pt]    {$( 5)$};
% Text Node
\draw (268.5,244.65) node [anchor=north] [inner sep=0.75pt]    {$( 6)$};
% Text Node
\draw (340,237.68) node [anchor=south] [inner sep=0.75pt]    {$( 7)$};
% Text Node
\draw (340,244.48) node [anchor=north] [inner sep=0.75pt]    {$( 8)$};
% Text Node
\draw (517.63,104.25) node [anchor=north] [inner sep=0.75pt]    {$( 2)$};
% Text Node
\draw (517.63,97.45) node [anchor=south] [inner sep=0.75pt]    {$( 3)$};
% Text Node
\draw (592.5,97.1) node [anchor=south] [inner sep=0.75pt]    {$( 6)$};
% Text Node
\draw (592.5,103.9) node [anchor=north] [inner sep=0.75pt]    {$( 5)$};
% Text Node
\draw (524.5,415.6) node [anchor=south] [inner sep=0.75pt]    {$( 1)$};
% Text Node
\draw (524.5,422.4) node [anchor=north] [inner sep=0.75pt]    {$( 4)$};
% Text Node
\draw (590,422.75) node [anchor=north] [inner sep=0.75pt]    {$( 8)$};
% Text Node
\draw (590,415.95) node [anchor=south] [inner sep=0.75pt]    {$( 7)$};
% Text Node
\draw (240.79,244.65) node [anchor=north east] [inner sep=0.75pt]    {$q_{1}$};
% Text Node
\draw (622.5,103.9) node [anchor=north west][inner sep=0.75pt]    {$x_{\max}$};
% Text Node
\draw (620.8,423.14) node [anchor=north west][inner sep=0.75pt]    {$x_{\max}$};
% Text Node
\draw (287,165.4) node [anchor=north west][inner sep=0.75pt]    {$\sigma $};

\end{tikzpicture}

    \caption{The contour in the supercritical regime}
    \label{fig:contour_nuts}
\end{figure}
Since the potential $V(x)$ did not change from the previous case, also the function $U(s)$ remains the same, and it is given by

\begin{equation}
    U(s)= -m_1 + \frac{1}{\rho\kappa\beta}\log\left(\frac{1- e^{-n_2\alpha\beta}J^\nu(s)}{1- e^{-n_1\alpha\beta}J^\nu(s)}\right)\,.
\end{equation}
Applying Proposition \ref{prop:ugly_integral}, one can explicitly solve the previous RHP as follows

\begin{equation}
    N(s) = \begin{cases}
         1 - \frac{1}{\rho\kappa\beta}\log\left(\frac{q_1-s}{s_2-s}\right) - \frac{1}{\beta\rho\kappa}\ln\left(\frac{s-s_1}{s-q_1}\right) \qquad &\text{outside } \sigma\\
        -1 - m_1 + \frac{1}{\rho\kappa\beta}\left[ \log\left(\frac{1- e^{-n_2\alpha\beta}J^\nu(s)}{1- e^{-n_1\alpha\beta}J^\nu(s)}\right) + \log\left(\frac{q_1-s}{s_2-s}\right) \right] + \frac{1}{\beta\rho\kappa}\ln\left(\frac{s-s_1}{s-q_1}\right)\qquad &\text{inside } \sigma\\
    \end{cases}
\end{equation}
where 
\begin{equation}
    \label{eq:s_val_above}
    J^\nu_{c_0,c_1}(q_1) = e^{n_1\alpha\beta} = x_{\max}\,, \quad J^\nu_{c_0,c_1}(s_2) = e^{n_2\alpha\beta}\,.
\end{equation}
In particular, the previous solution can be rewritten as follows:

\begin{equation}
    \label{eq:simple_N}
    N(s) = \begin{cases}
         1 - \frac{1}{\rho\kappa\beta}\log\left(\frac{s_1-s}{s_2-s}\right) \qquad &\text{outside } \sigma\\
        -1 - m_1 + \frac{1}{\rho\kappa\beta}\left[ \log\left(\frac{1- e^{-n_2\alpha\beta}J^\nu(s)}{1- e^{-n_1\alpha\beta}J^\nu(s)}\right) + \log\left(\frac{s_1-s}{s_2-s}\right) \right] \qquad &\text{inside } \sigma\\
    \end{cases}
\end{equation}
As in the previous case, the function $N(s)$ does not always respect the constraint c. of the RHP, therefore, we must impose that:

\begin{equation}
 \label{eq:normalization_above_crit}
    \begin{cases}
        \nu = N(0) = -1-m_1- \frac{\nu}{\kappa}(n_2-n_1) + \frac{1}{\rho\kappa\beta}\log\left(\frac{s_1}{s_2}\right)  \\
        0 = N(-1) = -1-m_1 + \frac{1}{\rho\kappa\beta}\log\left(\frac{s_1+1}{s_2+1}\right) \\
    \end{cases}\,.
\end{equation}
We notice that we have exactly $4$ parameters $c_0,c_1,s_1,s_2$ and $4$ independent equations \eqref{eq:q_val_above}-\eqref{eq:s_val_above}-\eqref{eq:normalization_above_crit}, where the only restriction is that $s_1$ has to be outside $\sigma$ and $s_2$ inside. Setting $s_1,s_2,c_0,c_1$  as in \eqref{eq:intro_constants_1} one can solve the previous system explicitly.
Finally, we can explicitly compute density $\omega(x)$ as follows:

    \begin{equation}
        \omega(x) = \begin{cases}
            \frac{1}{\pi \beta\rho\kappa x}\Arg\left(\frac{s_1 - I^+(x)}{ s_2 - I^+(x)}\right) \quad & x\in (a,b) \\
            \frac{1}{\beta\rho\kappa x} \quad & x\in (b,\xm)
        \end{cases}\,.
    \end{equation}
To conclude our proof of Theorem \ref{thm:above_crit}, we need to show that the E-L equations \eqref{eq:EL_constraint_2} are satisfied by the equilibrium measure. Define the function $h(x)$ as follows:

\begin{equation}
    h(x) = G_1^+(x)+G_\nu^-(x) + V'(x)\,.
\end{equation}
We constructed the equilibrium measure such that $h(x)=0$ on $(a,b)$, by direct calculation one can show that $h'(x)<0$ for $x\in(0,a)$ and $h(x)>0$ for $x\in(b,\xm)$, therefore by integration the E-L equations \eqref{eq:EL_constraint_2} are satisfied. This concludes the proof of Theorem \ref{thm:above_crit}.

\begin{remark}
    The proof of Theorem \ref{thm:above_crit_2} follows the same lines, the only difference is that in this case the function $U(s)$ is given by
\begin{equation}
    U(s) = -m_1 + \frac{1}{\rho\kappa\beta}\log\left(\frac{1- e^{-n_2\alpha\beta}J_{c_0,c_1}(s)}{1- e^{-n_1\alpha\beta}J_{c_0,c_1}(s)}\right)\,.
\end{equation}
\end{remark}

\subsection{Proof of Theorem \ref{thm:above_crit_case2}}

We can follow exactly the same steps as in the previous case, the only difference is that in this case the function $U(s)$ is given by
\begin{equation}
    U(s) = -m_1\,.
\end{equation}
Therefore, we can solve the RHP \ref{rhp:final_measure_2} for $N(s)$ as follows:

\begin{equation}
    N(s)=\begin{cases}
        1 - \frac{1}{\beta\rho\kappa}\ln\left(\frac{s-s_1}{s-q_1}\right) & \quad s \text{ outside } \sigma\\
        -1 -m_2 + \frac{1}{\beta\rho\kappa}\ln\left(\frac{s-s_1}{s-q_1}\right) &\quad s \text{ inside } \sigma\\
    \end{cases}\,.
\end{equation}
Then, we must impose the two conditions in c., therefore 

\begin{equation}
    \begin{cases}
        -1-m_1 +\frac{1}{\beta\rho\kappa}\ln\left(\frac{s_1 }{q_1}\right) = \nu \\[2pt]
        -1-m_1 + \frac{1}{\beta\rho\kappa}\ln\left(\frac{1+s_1 }{1+q_1}\right)=0
    \end{cases}
\end{equation}
Relabelling $q_1$ as $s_2$, one can show that the previous system of equation can be solved as in \eqref{eq:s_val_above_2_1}-\eqref{eq:c_val_above_2_1}. Furthermore, one can compute the equilibrium measure as

\begin{equation}
\begin{split}
        \omega(x) & = -\frac{G^+_1(x) - G^-_1(x)}{2\pi i} = - \frac{N^-(I^+(x))-N^-(I^-(x))}{2\pi ix}\\
        & = \begin{cases}
            \frac{1}{\pi \beta\rho\kappa x}\Arg\left(\frac{s_1 - I^+(x)}{ s_2 - I^+(x)}\right)\quad & x\in (a,b) \\
            \frac{1}{\beta\rho\kappa x} \quad & x\in (b,1)
        \end{cases}\,.
\end{split}
\end{equation}
Finally, by direct computations one verifies that \eqref{eq:EL_constraint_2} are satisfied.

\bibliographystyle{siam}
\bibliography{ldp4lpp.bib}

\end{document}